\documentclass[11pt,a4paper]{article}
\usepackage[utf8]{inputenc}
\usepackage[english]{babel}

\usepackage{amsfonts}
\usepackage{amssymb}
\usepackage{amsmath}
\usepackage{amssymb}
\usepackage{mathrsfs}
\usepackage{amsthm}
\usepackage{graphicx}
\usepackage{caption} 
\usepackage{wrapfig}
\usepackage{xcolor}
\usepackage{dsfont}
\usepackage{empheq}
\usepackage{thmtools}
\usepackage{thm-restate}

\usepackage{hyperref}
\hypersetup{
  unicode = true,                           % Non-Latin characters in Acrobat's bookmarks
  pdftoolbar = true,                        % Show Acrobat's toolbar?
  pdfmenubar = true,                        % Show Acrobat's menu?
  pdffitwindow = true,                      % Page fit to window when opened
  pdfauthor = {Chainais-Hillairet, Herda, Lemaire, Moatti},              % Authors
  pdftitle = {Long-time behaviour of hybrid finite volume schemes for advection-diffusion equations: linear and nonlinear approaches},                        % Title
  pdfnewwindow = true,                      % Links in new window
  colorlinks = true,                        % False: boxed links; true: colored links
  linkcolor = blue,                         % Color of internal links
  citecolor = magenta,                      % Color of links to bibliography
  filecolor = violet,                       % Color of file links
  urlcolor = violet                         % Color of external links
}

\usepackage{authblk}

\newcommand{\email}[1]{\href{mailto:#1}{#1}}

\usepackage{amsmath,amsfonts,amssymb,mathtools}
\usepackage{upgreek}
\usepackage{nicefrac}
\usepackage{cancel}
\usepackage{comment}
\usepackage{enumerate}
\usepackage{graphicx}
\usepackage{tikzscale}
\usepackage{xcolor}
\usepackage{paralist}
\usepackage{pgfplots}
\pgfplotsset{every axis/.append style={
                    label style={font=\Large},
                    tick label style={font=\Large},
                    legend style={font=\Large}
                    }}
\usepackage{subcaption}
\usepackage{etoolbox}
\usepackage[skins]{tcolorbox}
\makeatletter
\patchcmd{\ttlh@hang}{\parindent\z@}{\parindent\z@\leavevmode}{}{}
\patchcmd{\ttlh@hang}{\noindent}{}{}{}
\makeatother
\usepackage{multirow}
\usetikzlibrary{patterns}

\pgfplotsset{compat=1.15}
\newcommand{\logLogSlopeTriangle}[5]
{
    \pgfplotsextra
    {
        \pgfkeysgetvalue{/pgfplots/xmin}{\xmin}
        \pgfkeysgetvalue{/pgfplots/xmax}{\xmax}
        \pgfkeysgetvalue{/pgfplots/ymin}{\ymin}
        \pgfkeysgetvalue{/pgfplots/ymax}{\ymax}

        \pgfmathsetmacro{\xArel}{#1}
        \pgfmathsetmacro{\yArel}{#3}
        \pgfmathsetmacro{\xBrel}{#1-#2}
        \pgfmathsetmacro{\yBrel}{\yArel}
        \pgfmathsetmacro{\xCrel}{\xArel}
        \pgfmathsetmacro{\lnxB}{\xmin*(1-(#1-#2))+\xmax*(#1-#2)} % in [xmin,xmax].
        \pgfmathsetmacro{\lnxA}{\xmin*(1-#1)+\xmax*#1} % in [xmin,xmax].
        \pgfmathsetmacro{\lnyA}{\ymin*(1-#3)+\ymax*#3} % in [ymin,ymax].
        \pgfmathsetmacro{\lnyC}{\lnyA+#4*(\lnxA-\lnxB)}
        \pgfmathsetmacro{\yCrel}{\lnyC-\ymin)/(\ymax-\ymin)}

        \coordinate (A) at (rel axis cs:\xArel,\yArel);
        \coordinate (B) at (rel axis cs:\xBrel,\yBrel);
        \coordinate (C) at (rel axis cs:\xCrel,\yCrel);

        \draw[#5]   (A)-- node[pos=0.5,anchor=north] {\scriptsize{1}}
                    (B)-- 
                    (C)-- node[pos=0.,anchor=west] {\scriptsize{#4}} %% node[pos=0.5,anchor=west] {#4}
                    (A);
    }
}

\newtheorem{theorem}{Theorem}
\newtheorem{prop}{Proposition}
\newtheorem{lemma}{Lemma}
\newtheorem{rem}{Remark}
\newtheorem{corollary}{Corollary}
\usepackage[scale = .8]{geometry}

\DeclareMathOperator{\e}{e}
\DeclareMathOperator{\divergence}{div}

\newcommand{\R}{\mathbb{R}}
\newcommand{\A}{\mathbb{A}}
\newcommand{\B}{\mathbb{B}}
\newcommand{\M}{\mathcal{M}}
\newcommand{\E}{\mathcal{E}}
\newcommand{\D}{\mathcal{D}}
\newcommand{\Diss}{\mathbb{D}}
\newcommand{\N}{\mathbb{E}}
\newcommand{\F}{\mathcal{F}}

\renewcommand{\G}{\mathcal{G}}
\newcommand{\Gu}{\underline{\G}}
\newcommand{\s}{\sigma}
\newcommand{\V}{\underline{V}}
\renewcommand{\u}{\underline{u}}
\renewcommand{\v}{\underline{v}}
\newcommand{\w}{\underline{w}}
\newcommand{\ro}{\underline{\rho}}

\newcommand{\logg}{\underline{w}}
\newcommand{\rac}{\underline{\xi}}

\newcommand{\1}{\mathds{1}}
\newcommand{\one}{\underline{1}}
\newcommand{\SG} {Scharfetter--Gummel }
\newcommand{\Mat}{\mathbb{M}}
\newcommand{\etal}{et al.~}
\newcommand{\dd}{\mathrm{d}}

\newcommand{\defi}{\mathrel{\mathop:}=}
\newcommand{\ifed}{=\mathrel{\mathop:}}

\numberwithin{equation}{section}

\title{Long-time behaviour of hybrid finite volume schemes for {advection}-diffusion equations: linear and nonlinear approaches}

\author[1]{Claire Chainais-Hillairet\footnote{\email{claire.chainais@univ-lille.fr}}}
\affil[1]{Univ.~Lille, CNRS, Inria, UMR 8524 - Laboratoire Paul Painlev\'e, F-59000 Lille, France}
\author[2]{Maxime Herda\footnote{\email{maxime.herda@inria.fr}}}
\author[2]{Simon Lemaire\footnote{\email{simon.lemaire@inria.fr}}}
\author[2]{Julien Moatti\footnote{\email{julien.moatti@inria.fr}}}
\affil[2]{Inria, Univ.~Lille, CNRS, UMR 8524 - Laboratoire Paul Painlev\'e, F-59000 Lille, France}

\begin{document}
\maketitle
\begin{abstract}
We are interested in the long-time behaviour of approximate solutions to anisotropic and heterogeneous linear advection-diffusion equations in the framework of hybrid finite volume (HFV) methods on general polygonal/polyhedral meshes.
We consider two linear methods, as well as a new, nonlinear scheme, for which we prove the existence and the positivity of discrete solutions.
We show that the discrete solutions to the three schemes converge exponentially fast in time towards the associated discrete steady-states.
To illustrate our theoretical findings, we present some numerical simulations assessing long-time behaviour and positivity. 
We also compare the accuracy of the schemes on some numerical tests in the stationary case.
\end{abstract}



\tableofcontents
\section{Introduction}

We are interested in the numerical approximation of linear advection-diffusion equations on bounded domains.
These equations constitute the main building block in the modelling of more complex problems stemming from physics (e.g., porous media flows \cite{Bear:88}, or corrosion models \cite{BBCHD:10}), biology, or electronics (semi-conductor devices modelling \cite{VanRo:50}).
Thus, designing reliable numerical schemes to approximate their solutions is a pre-requisite before discretising more complex models. 
Our aim here is the preservation of some key physical properties of these equations at the discrete level, on a large variety of meshes.

Let $\Omega$ be an open, bounded, connected polytopal subset of $\R^d$, $d \in \{2,3\}$, with boundary $\partial\Omega$ divided into two disjoint open subsets $\Gamma^D$ and $\Gamma^N$, in such a way that $\partial \Omega = \overline{\Gamma^D} \cup \overline{\Gamma^N}$.
We consider the following problem: Find $u : \R_+ \times \Omega \to \R$ solution to
\begin{equation} \label{evol:mixed:ino}
	\left\{
	\begin{split}
		\partial _ t u - \divergence ( \Lambda  (\nabla u + u \nabla \phi )  ) &= f &&\text{ in } \R_+ \times \Omega, \\
		u &=  {g^D}&&\text{ on }  \R_+ \times \Gamma^D, \\
		\Lambda (\nabla u + u \nabla \phi ) \cdot n &= {g^N} &&\text{ on } \R_+ \times \Gamma^N,\\		
	 	u(0,\cdot)&= u^{in} &&\text{ in } \Omega ,
	\end{split}
	\right.
\end{equation}
where $n$ is the unit normal vector to $\partial\Omega$ pointing outward $\Omega$, and the data satisfy:
\begin{itemize}
	\item $\Lambda \in L^\infty(\Omega;\R^{d\times d})$ is a symmetric and uniformly elliptic diffusion tensor: there exist {$\lambda_\flat,\lambda_\sharp$} with $0 < \lambda_\flat \leq \lambda_\sharp<\infty$ such that,
for a.e.~$x$ in $\Omega$, $\xi \cdot \Lambda (x) \xi\geq \lambda_\flat|\xi|^2$ and {$|\Lambda(x)\xi| \leq  \lambda_\sharp |\xi|$} for all $\xi \in \R^d$; 
	\item $\phi \in C^1(\overline{\Omega})$ is a regular potential {from which derives the advection field $V^{\phi}\defi -\Lambda\nabla\phi$, assumed to satisfy $V^\phi\in H({\rm div};\Omega)$};
	\item $f \in L^2(\Omega)$ is a source term;
	\item $g^D \in H^{\frac{1}{2}}(\Gamma^D)$ is a Dirichlet datum, assumed to be the trace on $\Gamma^D$ of $u^D\in H^1(\Omega)$ satisfying $\|u^D\|_{H^1(\Omega)}\leq C \|g^D\|_{H^{\nicefrac12}(\Gamma^D)}$ for a given $C>0$;
        \item $g^N \in L^2(\Gamma^N)$ is a Neumann datum; 
	\item $u^{in} \in L^2(\Omega)$ is an initial datum.
\end{itemize}
When $|\Gamma^D|=0$, we assume that the compatibility condition $\int_{\Omega}f+\int_{\partial\Omega}g^N=0$ holds true{, and we denote by $M$ the initial mass such that  $M=\int_{\Omega}u^{in}$, which is known to be preserved along time: $\int_{\Omega}u(t)=M$ for almost every $t>0$.} For further use, {we also let in that case} $u^M\defi\frac{M}{|\Omega|}\in\R$, and we refer to this quantity as the mass lifting.
 Advection-diffusion models of the form \eqref{evol:mixed:ino} enjoy certain structural
 properties. First, when the data $f$, $g^D$, $g^N$, and $u^{in}$ are positive, then the solution $u$ is also positive. Second, the asymptotics $t \to \infty$, the so-called long-time behaviour of the solutions, is well understood (see \cite{BLMVi:14,CarTo:98,CarTo:00,Tosca:99} for related models). Indeed, the solution $u$ to \eqref{evol:mixed:ino} converges exponentially fast when $t \to \infty$ towards the steady-state $u^\infty$, solution to the stationary problem 
\begin{equation}  \label{sta:mixed:ino}
	\left\{
	\begin{split}
		- \divergence \left ( \Lambda ( \nabla u^\infty + u ^\infty \nabla \phi ) \right ) &= f &&\text{ in } \Omega, \\
		 u^\infty  &= g^D&&\text{ on } \Gamma^D, \\
		 \Lambda ( \nabla u ^\infty + u^\infty  \nabla \phi )  \cdot n &= g^N &&\text{ on } \Gamma^N, 
	\end{split}
	\right.
\end{equation}
with additional constraint $\int_{\Omega}u^\infty=M$ when $|\Gamma^D|=0$. The question of the long-time behaviour has been widely studied in the context of many-particle systems, for which the second law of thermodynamics ensures a relaxation of the transient phenomena towards an equilibrium. From a mathematical point of view, this evolution is strongly related to the dissipation of an entropy functional. Such a vision based on entropy dissipation has given birth to the so-called entropy method. As highlighted by Arnold \etal in \cite{ACDDJ:04}, the successful use of the entropy method in kinetic theory paves the way to extended applications on various dissipative systems. We refer the reader to the book \cite{Junge:16} of J\"ungel for a presentation of some of these applications.
In \cite{BLMVi:14}, Bodineau \etal proposed an entropy functional adapted to drift-diffusion equations with non-homogeneous Dirichlet boundary conditions. A direct adaptation of their method allows to conclude in the present case on the exponential convergence in time of the solution to Problem~\eqref{evol:mixed:ino} towards the solution to Problem~\eqref{sta:mixed:ino}.

Under appropriate assumptions on the data (a sufficient condition, also valid for more general advection fields, is to assume that $\divergence V^\phi\geq 0$ a.e.~in $\Omega$ and $V^\phi\cdot n\leq 0$ a.e.~on $\Gamma^N$), the stationary Problem \eqref{sta:mixed:ino} is coercive and its well-posedness is straightforward. It turns out that, even if such assumptions on the data are not fulfilled, for an advection field of the form $V^{\phi}= -\Lambda\nabla\phi$, the problem is still coercive in the new unknown $\rho^\infty = u^\infty\,\e^{\phi}$, so that one  can conclude on well-posedness by solving the problem in the new unknown. Concerning the evolution Problem~\eqref{evol:mixed:ino}, the same arguments show the existence and uniqueness of a global weak solution. For general advection fields (not necessarily deriving from a potential), we refer the reader to the results of Droniou \cite{Droni:02} (for mixed Dirichlet-Neumann boundary conditions), and Droniou and V\'azquez \cite{DrVaz:09} (for pure Neumann boundary conditions) for detailed statements about well-posedness and regularity of the solutions.

When it comes to numerical approximation, the accuracy of the method is not the only important feature.
In some applications (e.g., in subsurface modelling, where the mesh often results from seismic analysis), the mesh must be taken as a datum of the problem, and the numerical method needs to be adapted so as to handle potentially fairly general meshes. In some other applications (e.g., power plant simulation), the preservation of the positivity of the solutions (or better, of the monotonicity properties of the equation) is an important quality criterion. In yet some other applications (e.g., nuclear waste repository management), finally, the reliability of the simulations in very large time proves to be crucial for sustainability purposes.
The positivity and long-time behaviour of discrete solutions have been closely studied in the context of standard two-point flux approximation (TPFA) finite volume schemes, for isotropic diffusion (i.e., $\Lambda = \lambda\,I_d$ with $\lambda : \Omega \to \R^\star_+$) on orthogonal meshes.
In \cite{FiHer:17}, Filbet and Herda studied the long-time behaviour of a TPFA scheme for nonlinear boundary-driven Fokker--Planck equations, adapting to the discrete setting the arguments of \cite{BLMVi:14}.
In \cite{CHHer:20}, Chainais-Hillairet and Herda proved on a variety of models that a whole family of TPFA schemes (the so-called $B$-schemes) preserves the exponential decay towards discrete steady-states.
{The results of~\cite{FiHer:17} and~\cite{CHHer:20} are valid for general advection fields, and a choice of data $|\Gamma^D|>0$, $f=0$, $g^D>0$, $g^N=0$, and $u^{in}\geq 0$.} We also refer to \cite{LiLiu:18,ChCoo:70,I:69,BuDel:10} for related schemes and similar issues.
However, these TPFA schemes suffer from an intrinsic limitation: the mesh needs to be $\Lambda$-orthogonal, which, in practice, restricts their use to isotropic diffusion tensors and (standard) orthogonal meshes.
In order to overcome this limitation, several linear finite volume methods using auxiliary unknowns have been designed (cf.~\cite{Droni:14} for a presentation of some of these schemes).
As highlighted by Droniou in \cite{Droni:14}, these methods however suffer from a lack of monotonicity, and so do not preserve the positivity of discrete solutions. 
As a possible remedy, Canc\`es and Guichard introduced in \cite{CaGui:17} (see also the seminal paper~\cite{CaGui:16}), for a class of models encompassing \eqref{evol:mixed:ino} for pure Neumann boundary conditions and a choice of data $f=0$, $g^N = 0$, and $u^{in}\geq 0$ with $M>0$, a nonlinear vertex approximate gradient (VAG) scheme, designed so as to preserve at the discrete level the positivity of the solutions and the entropy structure of the models, for arbitrary anisotropic diffusions and general meshes. Following the same ideas, Canc\`es \etal devised and analysed in \cite{CCHKr:18} a (nonlinear) positivity-preserving discrete duality finite volume (DDFV) scheme, whose discrete entropy structure and long-time behaviour were fully studied in \cite{CCHHK:20}, based on the adaptation to the discrete setting of nonlinear functional inequalities. The DDFV scheme at hand is however limited to the two-dimensional case, and its adaptation to a three-dimensional framework seems difficult (cf.~\cite{Droni:14}).
Let us also mention the work~\cite{SAEF:17} (and the references therein), in which a general framework for the convergence analysis of positivity-preserving nonlinear cell-centred finite volume methods on general meshes is introduced.
On another level, it is known that, given adequate assumptions hold on the data, the solutions to Problem~\eqref{evol:mixed:ino} are regular in space (at least locally). This suggests that the use of high-order methods shall be an interesting track in order to increase the accuracy at fixed computational cost.
Recently introduced by Di Pietro \etal in \cite{DPELe:14}, hybrid high-order (HHO) methods can be seen as an arbitrary-order generalisation of hybrid finite volume (HFV) schemes, that were introduced by Eymard \etal in \cite{EGaHe:10} as yet another way to overcome the limitations of TPFA schemes. HFV methods hinge on cell and face unknowns (whence the vocable hybrid), and as such benefit from a unified 2D/3D formulation. {HFV methods have also been bridged to the larger family of hybrid mimetic mixed (HMM) methods in \cite{DrEGH:10}.}
In view of the above elements, the study of HFV methods appears to be a natural first step in order to design structure-preserving high-order (HHO) schemes for Problem~\eqref{evol:mixed:ino}, that shall both increase the accuracy at fixed computational burden, and preserve the key properties (positivity and long-time behaviour) of the model at hand.  

In this article, we study three different HFV schemes for Problem~\eqref{evol:mixed:ino}. The first one is the {HFV variant of the HMM family of schemes} introduced and analysed in the stationary setting by Beir\~ao da Veiga \etal in \cite{BdVDM:11,Droni:10} {(note that an arbitrary-order (HHO) generalisation of this scheme has been proposed in~\cite{DPDEr:15})}. It is a linear scheme, based on a discretisation of the diffusive and advective fluxes, that is well-posed under a coercivity condition.
The second scheme is also a linear one. Its construction is based on exponential fitting, and takes inspiration from ideas in \cite{BMaPi:89} {(it also shares some features with the works~\cite{LBLMa:16,Madio:16} and~\cite{FiHer:17}, which cover general advection fields)}. This scheme is unconditionally coercive.
These two linear schemes are not expected to preserve positivity, which motivates the introduction of the third method.
For pure Neumann boundary conditions, and a choice of data $f=0$, $g^N = 0$, and $u^{in}\geq 0$ with $M>0$ {(see Appendix~\ref{Ap:NLmixted} for the case of mixed Dirichlet-Neumann boundary conditions)}, we introduce a nonlinear HFV scheme, that is devised along the lines of the nonlinear VAG and DDFV schemes of \cite{CaGui:17} and \cite{CCHKr:18,CCHHK:20}, so as to guarantee the positivity of discrete solutions. Our first result, stated in Theorem~\ref{th:existencenonlin}, is the existence of (positive) solutions to this nonlinear scheme.
In a second time, we investigate the long-time behaviour of the three schemes at hand.
We establish in Theorems~\ref{Th:hmm},~\ref{Th:omega}, and~\ref{Th:nonlin} the exponential decay in time of their discrete solutions towards the associated discrete steady-states.
We numerically validate our theoretical findings on a set of test-cases and, for completeness, we also compare the accuracy of the three schemes on stationary problems.

The article is organised as follows. In Section~\ref{HFV}, we introduce the HFV framework (mesh, discrete unknowns and discrete operators) on a steady variable diffusion problem.
In Section~\ref{description}, we introduce the three schemes for the transient advection-diffusion problem, and we discuss their well-posedness.
In Section~\ref{Time}, we study the long-time behaviour of the three schemes, and prove exponential decay to equilibrium.
In Section~\ref{Numerical}, we discuss the implementation of the schemes, and provide a numerical validation of our theoretical results, as well as a comparison of the stationary schemes in terms of accuracy. Appendices~\ref{Appendix_funcinequalities}{, \ref{Ap:NLmixted}, and} \ref{Appendix_technicalresults} finally collect some functional inequalities and the proofs of {supplementary and} auxiliary results.

%%%%%%%---------------------------------------------------------
%%%%%%%---------------------------------------------------------

\section{Hybrid finite volume discretisation of a variable diffusion problem} \label{HFV}

The aim of this section is to recall  the HFV framework on a steady variable diffusion problem, which corresponds to \eqref{sta:mixed:ino} without advection term ($V^\phi=0$). For a detailed presentation of the method, we refer to \cite{EGaHe:10}.

%\subsection{Hybrid finite volume framework}\label{HFV}

\subsection{Mesh and discrete unknowns} \label{sec:discre}

The definitions and notation we adopt for the discretisation are essentially the same as in \cite{EGaHe:10}.
A discretisation of the (open, bounded) polytopal set $\Omega \subset \R ^d $, $d\in\{2,3\}$, is defined as a triplet $\mathcal{D}\defi( \M, \E, \mathcal{P})$, where:
\begin{itemize}
		\item $\M$ (the mesh) is a partition of $\Omega$, i.e., a finite family of nonempty disjoint (open, connected) {polytopal} subsets $K$ of $\Omega$ (the mesh {cells}) such that (i) for all $K \in \M$, $|K|>0$, and (ii) $\overline{\Omega} = \bigcup _ {K \in \M } \overline{K} $.
	
		\item $\E$ (the set of faces) is a partition of the mesh skeleton $\bigcup_{K\in\M}\partial K$, i.e., a finite family of nonempty disjoint (open, connected) subsets $\s$ of $\overline{\Omega}$ (the mesh faces, or mesh edges if $d = 2$) such that (i) for all $\s\in\E$, $|\s|>0$ and there exists $\mathcal{H}_\s$ affine hyperplane of $\R^d$ such that $\s\subset\mathcal{H}_\s$, and (ii) $\bigcup_{K\in\M}\partial K=\bigcup_{\s\in\E}\overline{\s}$. We assume that, for all $K \in \M$, there exists $\E_K \subset \E$ (the set of faces of $K$) such that $\partial K = \bigcup_{\s \in \E_K} \overline{\s}$. For $\s \in \E$, we let $\M_\s \defi \lbrace K \in  \M \mid \s \in \E_K \rbrace $ be the set of cells whose $\s$ is a face. Then, for all $\s \in \E$, either $\M_\s=\{K\}$ for a cell $K\in\M$, in which case $\s$ is a boundary face ($\s \subset \partial \Omega$) and we note $\s \in \E_{ext}$, or $\M_\s = \lbrace K,L \rbrace$ for two cells $K,L\in\M$, in which case $\s$ is an interface and we note $\s = K|L \in \E_{int} $.
		
		\item $\mathcal{P}$ (the set of cell centres) is a finite family $\{x_K\}_{K \in \M}$ of points of $\Omega$ such that, for all $K \in \M$, (i) $x_K \in K$, and (ii) $K$ is star-shaped with respect to $x_K$. Moreover, we assume that the Euclidean (orthogonal) distance $d_{K,\s}$ between $x_K$ and the affine hyperplane $\mathcal{H}_\s$ containing $\s$ is positive (equivalently, the cell $K$ is strictly star-shaped with respect to $x_K$).
\end{itemize}
For a given discretisation $\D$, we denote by $h_\D>0$ the size of the discretisation (the meshsize), defined by $h_\D \defi\underset{K \in \M}{\sup} h_K $ where, for all $K \in \M$, $h_K\defi\underset{x,y \in \overline{K}}{\sup} |x-y|$ is the diameter of the cell $K$.
For all $\s\in\E$, we let ${\overline{x}_\s} \in \s$ be the barycentre of $\s$.
Finally, for all $K\in\M$, and all $\s \in \E_K$, we let $n_{K,\s} \in \R^d$ be the unit normal vector to $\s$ pointing outward $K$, and $P_{K, \s}$ be the (open) pyramid of base $\s$ and apex $x_K$ (notice that, when $d=2$, $P_{K,\s}$ is always a triangle). Since $|\s|$ and $d_{K,\s}$ are positive, we have $|P_{K,\s}|=\frac{|\s|d_{K,\s}}{d}>0$.
We depict on Figure \ref{Maillage} an example of discretisation. Notice that the mesh cells are not assumed to be convex, neither $x_K$ is assumed to be the barycentre of $K\in\M$. {Notice that hanging nodes are seamlessly handled with our assumptions, so that meshes with non-conforming cells are allowed (see the orange cross in Figure \ref{Maillage}; the cell $K$ therein is treated as an hexagon).}
\begin{figure}[h!]
\begin{center}
\def\svgwidth{0.8\textwidth}
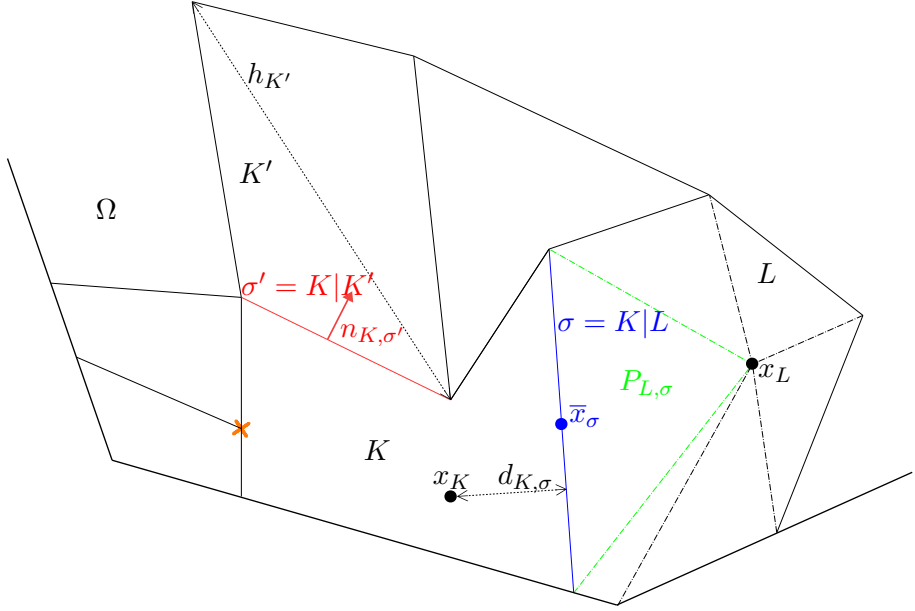
\end{center}
\caption{Two-dimensional discretisation and corresponding notation.}
\label{Maillage}
\end{figure}
We consider the following measure of regularity for the discretisation ({which is slightly stronger than the ones advocated in~\cite[Eq.~(4.1)]{EGaHe:10} or in~\cite[Eq.~(7.8)-(7.9)]{DEGGH:18}}): 
\begin{equation} \label{def:regmesh}
  \theta_\D \defi \max \left (  \underset{K \in \M, \s \in \E_K}{\max} \frac{h_K}{d_{K,\s}} , \underset{\s \in \E, K \in \M_\s}{\max} \frac{h_K^{d-1}}{|\s|} \right ).
\end{equation}
Notice that $\theta_\D \geq 1$, and that for all $K \in \M$, 
\[
	h_K^d \geq |K| = \sum_{\s \in \E_K} |P_{K,\s}| = \sum_{\s \in \E_K} \frac{|\s|d_{K,\s}}{d} 
	\geq \sum_{\s \in \E_K} \frac{h_K^d}{d \theta_\D^2}
	= \frac{|\E_K|}{d \theta_\D^2} h_K^d.
\]
Thus, the number of faces of any mesh cell is uniformly bounded: 
\begin{equation}\label{def:boundfaces}
	\forall K \in \M,\qquad  |\E_K|\leq d \theta_\D^2.
\end{equation}
Also, it is an easy matter to verify that $\underset{\s = K | L \in \E_{int}}{\max}\!\!\!\max\left( \frac{d_{K,\s}}{d_{L,\s}},\frac{d_{L,\s}}{d_{K,\s}}\right)\leq\theta_{\D}^{\frac{d}{d-1}}$.
{Given $\mathcal{F}$ a family of discretisations, we say that $\mathcal{F}$ is uniformly regular if there exists $\theta \geq 1$ such that for all $\D \in \mathcal{F}$, $\theta_\D \leq \theta$.}

We now introduce the set of (hybrid, cell- and face-based) discrete unknowns: 
\begin{equation*}
		\V _\D  \defi \big\lbrace \v_{\D} =\big( (v_K )_{K \in \M } ,  (v_\s)_{\s \in \E} \big) : v_K \in \R\;\forall K\in\M, v_\s \in \R\;\forall\s\in\E \big\rbrace.
\end{equation*}
{Given a mesh cell $K\in\M$, we let $\V_K\defi\R\times\R^{|\E_K|}$ be the restriction of $\V_{\D}$ to $K$, and $\v_K=\big(v_K,(v_{\s})_{\s\in\E_K}\big)\in\V_K$ be the restriction of a generic element $\v_{\D}\in\V_{\D}$ to $K$.}
Also, for $\v_{\D}\in\V_{\D}$, we let $v_\M : \Omega \to \R$ and $v_\E : \bigcup _{K\in\M} \partial K \to \R$ be the piecewise constant functions such that
$$v_{\M\mid K} = v_K \text{ for all } K\in\M, \quad\text{ and }  \quad v_{\E\mid\s} = v_\s \text{ for all } \s \in \E.$$
In what follows, for any set $X\subset\overline{\Omega}$, we denote by $(\cdot,\cdot)_X$ the inner product in $L^2(X;\R^l)$, for $l\in\{1;d\}$.
In particular, we have
$\displaystyle( w_\M,  v_\M)_{\Omega} = \sum_{K \in \M } |K| w_K v_K$ and $\displaystyle( w_\E,  v_\E )_{\partial \Omega} =  \sum_{\s \in \E_{ext} } |\s| w_\s v_\s$.
For further use, we let $\one_{\D}$ denote the element of $\V_{\D}$ with all coordinates equal to $1$.
Also, given a function $f: \R \to \R$, and with a slight abuse in notation, we denote by $f(\v_{\D})$ the element of $\V_{\D}$ whose coordinates are the $(f(v_K))_{K\in \M}$ and the $(f(v_\s))_{\s \in \E}$.
Finally, we let the product {$\w_\D\times\v_{\D}$} denote the element of $\V_{\D}$ whose $i$-th coordinate is the product of the $i$-th coordinates of $\w_{\D}$ and $\v_{\D}$.

When considering mixed Dirichlet-Neumann boundary conditions, we assume that the discretisation $\D$ is compliant with the partition $\partial \Omega = \overline{\Gamma^D} \cup \overline{\Gamma ^N}$ of the boundary of the domain, in the sense that the set $\E_{ext}$ can be split into two (necessarily disjoint) subsets $\E_{ext}^D\defi\left \{ \s \in \E_{ext} \mid \s \subset \Gamma^D \right \}$ and $\E_{ext}^N\defi\left \{ \s \in \E_{ext} \mid \s \subset \Gamma^N \right \}$ such that $\E_{ext} = \E_{ext}^D \cup \E_{ext}^N$. Notice that as soon as $|\Gamma^D|>0$, $|\E_{ext}^D|\geq 1$. We define the following subspace of $\V_{\D}$, enforcing strongly a homogeneous Dirichlet boundary condition on $\Gamma^D$:
$$\V^D_{\D,0} \defi \left \lbrace \v_{\D} \in \V_{\D} : v_\s = 0 \;\forall \s \in \E_{ext}^D\right \rbrace.$$
In view of the upcoming analysis, we define a discrete counterpart of the $H^1$ seminorm. Locally to any cell $K\in\M$, we let, for any $\v_K\in\V_K$, $|\v_K|_{1,K}^2\defi\sum _{\s \in \E _K} \frac{|\s|}{d_{K,\s}}(v_K - v_\s)^2$. At the global level, for any $\v_{\D} \in \V_\D$, we let
\begin{equation*}
|\v_{\D}|_{1, \D} \defi \sqrt{\sum_ {K \in \M} |\v_K|_{1,K}^2 }.
\end{equation*}
Notice that $|\cdot|_{1,\D}$ does not define a norm on $\V_\D$, but if $|\v_{\D}|_{1, \D}= 0$, then there is $c \in \R$ such that $\v_{\D} = c \,\one_{\D}$ ($\v_{\D}$ is constant).
Thus, $|\cdot|_{1,\D}$ defines a norm on the space $\V^D_{\D,0}$ as soon as $|\Gamma^D| > 0$, as well as on the space of zero-mass vectors
$$\V_{\D,0}^N \defi \left \lbrace \v_{\D} \in \V_{\D} : \int_{\Omega}v_{\M} = 0 \right \rbrace.$$
{For further use, and to allow for a seamless treatment of pure Neumann boundary conditions, we introduce the notation $\V_{\D,0}$, to denote either $\V_{\D,0}^N$ whenever $|\Gamma^D|=0$, or $\V_{\D,0}^D$ otherwise.} 
%% c{[In the sequel, the discretisation is assumed to be fixed, and we often omit the subscript $\D$ for brevity, thus writing for instance $\V$, $\V_0$, $\V_0^D$, and  $\V_0^N$ instead of $\V_\D$, $\V_{\D,0}$, $\V^D_{\D,0}$, and $\V_{\D,0}^N$.]}

\subsection{Discrete problem}
  %Foundations of the hybrid finite volume discretisation}

The HFV method hinges on the definition of a discrete gradient operator $\nabla_{\D}$, that maps any element $\v_{\D}\in\V_{\D}$ to a piecewise constant $\R^d$-valued function on the pyramidal submesh of $\M$ formed by all the $P_{K,\s}$'s, for $K\in\M$ and $\s\in\E_K$. More precisely, for all $K\in\M$, and all $\s\in\E_K$,
\begin{equation*}
  {{\nabla_{\D}\v_{\D}}_{\mid K}\defi\nabla_K\v_K\quad\text{with}\quad{\nabla_K \v_K}_{\mid P_{K,\s}} \defi \nabla_{K,\s} \v_K= G_K \v_K + S_{K,\s} \v_K\in\R^d},
\end{equation*}
where $G_K\v_K$ is the consistent part of the gradient given by
$$G_K\v_K=\frac{1}{|K|}\sum_{\s'\in\E_K}|\s'|(v_{\s'}-v_K)n_{K,\s'}=\frac{1}{|K|}\sum_{{\s'}\in\E_K}|\s'|v_{\s'}n_{K,\s'},$$
and $S_{K,\s}\v_K$ is a stabilisation given, for some free parameter $\eta>0$, by
\begin{equation} \label{def:stabilisation}
	S_{K,\s}\v_K=\frac{\eta}{d_{K,\s}}\big(v_\s-v_K-G_K\v_K\cdot(\overline{x}_\s-x_K)\big)n_{K,\s}.
\end{equation}
{\begin{rem}[Choice of $\eta$]\label{rem:stab}
    There are two specific values of the stabilisation parameter $\eta>0$ for which one recovers known numerical schemes from the literature:
    \begin{itemize}
      \item[(i)] for $\eta=\sqrt{d}$, one recovers the original HFV scheme of~\cite{EGaHe:10}, that coincides with the TPFA scheme on super-admissible meshes (see~\cite[Lemma 2.10]{EGaHe:10});
      \item[(ii)] for $\eta=d$, one recovers the Discrete Geometric Approach (DGA) of~\cite{CST:10}, later bridged to the non-conforming finite element setting in~\cite{DPLe:15}.
    \end{itemize}
    The influence of the value of $\eta$ on the numerical results has been investigated in~\cite{BDPE:15} for anisotropic diffusion problems. It is shown that the above two values are appropriate choices (neither under- nor over-penalised).
\end{rem}}

Let us consider the stationary problem~\eqref{sta:mixed:ino}, without advection term ($V^\phi=0$).
Our aim is to write {an} HFV discretisation of this steady variable diffusion problem.
Locally to any cell $K\in\M$, we introduce the discrete bilinear form $a^{\Lambda}_K:\V_K\times\V_K\to\R$ such that, for all $\u_K,\v_K\in\V_K$,
\begin{equation} \label{def:aKL}
  {a^{\Lambda}_K}(\u_K,\v_K)\defi\sum_{\s\in\E_K}|P_{K,\s}|\nabla_{K,\s}\v_K\cdot\Lambda_{K,\s}\nabla_{K,\s}\u_K=(\Lambda\nabla_K \u_K, \nabla_K \v_K)_K,
\end{equation}
where we set {$\Lambda_{K,\s}\defi\frac{1}{|P_{K,\s}|}\int_{P_{K,\s}}\Lambda$}.
At the global level, we let $a^{\Lambda}_{\D}:\V_{\D}\times\V_{\D}\to\R$ be the discrete bilinear form such that, for all $\u_{\D},\v_{\D}\in\V_{\D}$,
\begin{equation*}
  a^{\Lambda}_\D (\u_{\D},\v_{\D})\defi\sum_{K\in\M}a^{\Lambda}_K(\u_K,\v_K)=(\Lambda\nabla_\D \u_{\D},\nabla_\D \v_{\D})_{\Omega}.
\end{equation*}
{The discrete HFV problem then reads: Find $\underline{u}^z_{\D}\in\V_{\D,0}$ such that
\begin{equation} \label{eq:dis.pro}
  a^{\Lambda}_{\D}(\underline{u}^z_{\D},\v_{\D})=( f, v_{\M} )_{\Omega}+( g^N, v_{\E})_{\Gamma^N}-a^{\Lambda}_{\D}(\underline{u}^l_{\D},\v_\D)\qquad\forall\v_{\D}\in\V_{\D,0},
\end{equation}
where $\underline{u}^l_{\D}\in\V_{\D}$ is equal
\begin{itemize}
  \item[(i)] either, when $|\Gamma^D|>0$, to the HFV interpolate $\u_{\D}^D$ of the known lifting $u^D$ of the Dirichlet datum $g^D$ (satisfying $|\underline{u}^D_{\D}|_{1,\D}\leq C_{l,\Gamma^D}\|g^D\|_{H^{\nicefrac12}(\Gamma^D)}$, where $C_{l,\Gamma^D}>0$ only depends on the discretisation $\D$ through $\theta_{\D}$),
  \item[(ii)] or, when $|\Gamma^D|=0$, to $\u_{\D}^M\defi u^M\underline{1}_{\D}$, where we recall that $u^M=\frac{M}{|\Omega|}$ is the mass lifting (remark that $a^{\Lambda}_{\D}(\underline{u}^M_{\D},\v_\D)=0$ for all $\v_{\D}\in\V_\D$),
\end{itemize}
and the approximation of the solution to~\eqref{sta:mixed:ino}, denoted $\u^\infty_{\D}\in\V_{\D}$, is finally defined as 
\begin{equation}\label{def.uinf}
\u^\infty_{\D}\defi\underline{u}^z_{\D}+\underline{u}^l_{\D}.
\end{equation} 
Let us note that the superscript $z$ stands for ``zero", while $l$ stands for ``lifting".}

Problem~\eqref{eq:dis.pro} defines a finite volume method, in the sense that it can be equivalently rewritten under a conservative form, with local mass balance, flux equilibration at interfaces, and boundary conditions.
For all $K\in\M$, and all $\s\in\E_K$, the normal diffusive flux $ - \int_{\s} \Lambda \nabla u^\infty \cdot n_{K,\s}$ is approximated by the following numerical flux:
\begin{equation} \label{diff.flux}
	{F_{K,\s}^{\Lambda} (\u_K)} = \sum_{\s ' \in \E_K } A_K ^{\s \s'} (u_K - u_{\s'} ),
\end{equation}
where the $A_K^{\s\s'}$ are defined by 
\begin{equation} \label{def:flux:A}
		A_K ^{\s\s'} \defi \sum_{\s'' \in \E_K } |P_{K,\s''}|\,y_K ^{\s'' \s } \cdot \Lambda _{K, \s''} y_K^{\s''\s'},
\end{equation} 
and the $y_K^{\s\s'}\in\R^d$ only depend on the geometry of the discretisation $\D$ (see, for example, \cite[Eq.~(2.22)]{EGaHe:10} for an exact definition with $\eta = \sqrt{d}$).
For all $K\in\M$, one can express the local discrete bilinear form $a_K^{\Lambda}$ in terms of the local fluxes $\big(F^{\Lambda}_{K,\s}\big)_{\s\in\E_K}$: for all $\u_K,\v_K\in\V_K$, 
\begin{equation} \label{aKL-flux}
	a_K^{\Lambda}(\u_K, \v_K) = \sum_{\s \in \E_K } F^{\Lambda}_{K,\s}(\u_K)(v_K - v_\s).
\end{equation}
As for the VAG~\cite{CaGui:17} and DDFV~\cite{CCHKr:18,CCHHK:20} schemes, we can also express the local discrete bilinear form in a different way, which will be useful in the sequel: 
\begin{equation} \label{def:a_K}
	a_K^{\Lambda}(\u_K, \v_K) = \delta_K \v_K \cdot \A _K  \delta_K \u_K , 
\end{equation}
where, for all $\v_K\in\V_K$, $\delta_K\v_K\in\R^{|\E_K|}$ is defined by
$${\delta_K \v_K} \defi \left ( v_K - v_\s \right )_{\s \in \E_K },$$
and $\A_K \in \R^{|\E_K|\times|\E_K|}$ is the symmetric (because $\Lambda$ is) positive semi-definite matrix whose entries are the $A_K ^{\s\s'}$, that can actually be proved to be nonsingular (cf.~Lemma~\ref{lemma:cond}).

\subsection{Well-posedness} \label{sse:wp}

%We first recall the following hybrid discrete Poincar\'e inequalities (cf.~\cite[Lemmas B.25 and B.32, $p=2$] {DEGGH:18}).
%\begin{prop}[Discrete Poincar\'e inequalities] \label{prop:poinca}
%  Let $\D$ be a given discretisation of $\Omega$, with regularity parameter $\theta_\D$.
%  There exists $C_{PW}>0$, only depending on $\Omega$, $d$, and $\theta_\D$ such that 
%  \begin{equation} \label{poincawir}
%    \forall \v_{\D} \in \V_{\D,0}^N , \qquad \|v_\M \|_{L^2(\Omega)} \leq {C_{PW}} | \v_{\D} |_{1,\D}.
%  \end{equation}
%  Assume that $|\Gamma^D| > 0$.
%  Then, there exists $C_{P,\Gamma^D}>0$, only depending on $\Omega$, $d$, $\Gamma^D$, and $\theta_\D$ such that 
%\begin{equation} \label{poincadir}
%	\forall \v_{\D} \in \V_{\D,0}^D, \qquad \|v_\M \|_{L^2(\Omega)} \leq {C_{P,\Gamma^D}} | \v_{\D} |_{1,\D}.
%\end{equation}
%\end{prop}

As for the continuous case, the well-posedness of HFV methods for diffusion problems relies on a coercivity argument.
Let $K\in\M$, and reason locally. By definition~\eqref{def:aKL} of the local discrete bilinear form $a_K^{\Lambda}$, and from the bounds on the diffusion coefficient, we have $\lambda_\flat \|\nabla_K \v_K\|_{L^2(K;\R^d)}^2 \leq a_K^{\Lambda}(\v_K,\v_K) \leq \lambda_\sharp \|\nabla_K \v_K\|_{L^2(K;\R^d)}^2$ for all $\v_K\in\V_K$.
Furthermore, the following comparison result holds (cf.~\cite[Lemma 13.11, $p=2$]{DEGGH:18} and its proof): there exist {$\alpha_\flat,\alpha_\sharp$} with $0<\alpha_\flat\leq\alpha_\sharp<\infty$, only depending on $\Omega$, $d$, and $\theta_\D$ such that 
$\alpha_\flat |\v_K|_{1,K}^2\leq \|\nabla_K \v_K\|_{L^2(K;\R^d)} ^2\leq  \alpha_\sharp |\v_K|_{1,K}^2$ for all $\v_K\in\V_K$.
Combining both estimates, we infer a local coercivity and boundedness result:
\begin{equation} \label{def:localcoercivity}
	\forall \v_K\in \V_K, \qquad\lambda_\flat\alpha_\flat |\v_K|_{1,K}^2
	\leq a^{\Lambda}_K(\v_K,\v_K)
	\leq  \lambda_\sharp\alpha_\sharp |\v_K|_{1,K}^2.
\end{equation}
Summing over $K \in \M$, we get the following global estimates:
\begin{equation} \label{def:globalcoercivity}
	\forall \v_{\D} \in \V_{\D}, \qquad  \lambda_\flat\alpha_\flat | \v_{\D} |_{1,\D}^2 \leq a^{\Lambda}_\D(\v_{\D},\v_{\D}) \leq \lambda_\sharp \alpha_\sharp | \v_{\D} |_{1,\D}^2.
\end{equation}
The well-posedness of Problem~\eqref{eq:dis.pro}-\eqref{def.uinf} follows.
{
\begin{prop}[Well-posedness] \label{pro:wp}
  There exists a unique solution $\u^\infty_{\D}\in\V_{\D}$ to Problem~\eqref{eq:dis.pro}-\eqref{def.uinf}, which satisfies $|\u^\infty_{\D}|_{1,\D}\leq C\left(\|f\|_{L^2(\Omega)}+\|g^N\|_{L^2(\Gamma^N)}+\|g^D\|_{H^{\nicefrac12}(\Gamma^D)}\right)$, for some $C>0$ depending on the data, and on the discretisation $\D$ only through $\theta_\D$.
\end{prop}
%Before proving Proposition~\ref{pro:wp}, let us recall the hybrid discrete Poincar\'e inequalities (cf.~\cite[Lemmas B.25 and B.32, $p=2$] {DEGGH:18}).
%
%\begin{prop}[Discrete Poincar\'e inequalities] \label{prop:poinca}
%  Let $\D$ be a given discretisation of $\Omega$, with regularity parameter $\theta_\D$.
%  There exists $C_{PW}>0$, only depending on $\Omega$, $d$, and $\theta_\D$ such that 
%  \begin{equation} \label{poincawir}
%    \forall \v_{\D} \in \V_{\D,0}^N , \qquad \|v_\M \|_{L^2(\Omega)} \leq {C_{PW}} | \v_{\D} |_{1,\D}.
%  \end{equation}
%  Assume that $|\Gamma^D| > 0$.
%  Then, there exists $C_{P,\Gamma^D}>0$, only depending on $\Omega$, $d$, $\Gamma^D$, and $\theta_\D$ such that 
%\begin{equation} \label{poincadir}
%	\forall \v_{\D} \in \V_{\D,0}^D, \qquad \|v_\M \|_{L^2(\Omega)} \leq {C_{P,\Gamma^D}} | \v_{\D} |_{1,\D}.
%\end{equation}
%\end{prop}

\begin{proof}
  The existence/uniqueness of $\u^z_\D\in\V_{\D,0}$ solution to~\eqref{eq:dis.pro} (and in turn of $\u^\infty_{\D}=\underline{u}^z_{\D}+\underline{u}_{\D}^l$) is a direct consequence of the coercivity estimate~\eqref{def:globalcoercivity}, and of the fact that $|\cdot|_{1,\D}$ defines a norm on $\V_{\D,0}$ (recall that $\V_{\D,0}$ denotes either $\V_{\D,0}^N$ when $|\Gamma^D|=0$, or $\V_{\D,0}^D$ otherwise). To prove the bound on $|\u^\infty_{\D}|_{1,\D}$, we use the triangle inequality:
  $$|\u^\infty_\D|_{1,\D}\leq|\underline{u}^z_\D|_{1,\D}+|\underline{u}_\D^l|_{1,\D}.$$
 {To estimate the first term, we test Problem~\eqref{eq:dis.pro} with $\v_{\D}=\underline{u}^z_{\D}\in\V_{\D,0}$, we use~\eqref{def:globalcoercivity}, and we apply the Cauchy--Schwarz inequality. We get
  $$\lambda_\flat\alpha_\flat | \underline{u}^z_{\D} |_{1,\D}^2\leq\|f\|_{L^2(\Omega)}\|u^z_{\M} \|_{L^2(\Omega)}+\|g^N\|_{L^2(\Gamma^N)}\|u^z_{\E}\|_{L^2(\Gamma^N)}+\lambda_\sharp\alpha_\sharp\,|\underline{u}_{\D}^l|_{1,\D}| \underline{u}^z_{\D} |_{1,\D}.$$
Using a discrete Poincar\'e inequality, recalled in Proposition~\ref{prop:poinca}, and applied to $\underline{u}^z_{\D}\in\V_{\D,0}$, as well as the discrete trace inequality of~\cite[Eq.~(B.58), $p=2$]{DEGGH:18} combined with a discrete Poincar\'e inequality, we obtain
\[
	| \underline{u}^z_{\D} |_{1,\D} \leq C_1 (\|f\|_{L^2(\Omega)}+	\|g^N\|_{L^2(\Gamma^N)} ) + C_2	|\underline{u}_{\D}^l|_{1,\D} .
\]
It remains to estimate the norm of the lifting $|\underline{u}_\D^l|_{1,\D}$:
\begin{itemize}
\item[(i)] if $|\Gamma^D|>0$, one has $\underline{u}_\D^l =\underline{u}^D_\D$ (interpolate of the lifting $u^D$), therefore $|\underline{u}_\D^l|_{1,\D}\leq C_{l,\Gamma^D}\|g^D\|_{H^{\nicefrac12}(\Gamma^D)}$;
\item[(ii)] if $|\Gamma^D|= 0$, since $\underline{u}_\D^l = \u_{\D}^M = u^M \one_\D $, the lifting is constant and $|\underline{u}_\D^l|_{1,\D} = 0$.
\end{itemize}
}
\end{proof}
}

%% The local coercivity result~\eqref{def:localcoercivity}, through the relation~\eqref{def:a_K}, has direct consequences on the matrices $(\A_K)_{K \in \M}$.
%% In particular, the matrices $(\A_K)_{K \in \M}$ are nonsingular. Also, if one defines, for all $K\in\M$, $\B_K \in \R^{|\E_K|\times|\E_K|}$ the diagonal matrix with entries
%% \begin{equation} \label{BK}
%%   B_K^{\s\s} \defi \sum_{\s' \in \E_K} |A_K^{\s\s'}|\quad\text{ for all }\s\in\E_K,
%% \end{equation}
%% then there exists $C_B>0$, only depending on $\Lambda$, $\Omega$, $d$, and $\theta_\D$, such that 
%% $$\forall K \in \M, \forall w  \in \R^{|\E_K|}, \quad w \cdot \A_K w \leq w \cdot \B_K w \leq C_B \, w \cdot \A_K w.$$
%% The precise statements of the latter results are collected in Lemma~\ref{lemma:cond}, and postponed (along with their proof) until Appendix~\ref{ap:AK}.

\section{Definition of the schemes and well-posedness} \label{description}

In this section, we introduce and study the well-posedness of three HFV schemes, two linear ones and a nonlinear scheme, for the time-dependent advection-diffusion problem~\eqref{evol:mixed:ino}. For the first two (linear) schemes, we introduce and study in the first place their steady versions on Problem~\eqref{sta:mixed:ino}.
For the nonlinear scheme, by anticipation of the asymptotic analysis of Section~\ref{asnon}, we restrict our study to the case where the (positive) solution to Problem~\eqref{evol:mixed:ino} converges in long time towards the so-called thermal equilibrium (see~\eqref{eq:therm}). However, as it will be verified numerically in Section~\ref{acc} in the stationary setting, our scheme is applicable to more general data.
We consider a fixed spatial discretisation $\D$ of $\Omega$, which satisfies the conditions detailed in Section~\ref{sec:discre}, and a fixed time step $\Delta t > 0$ for the time discretisation.
%% c{For $n \in \mathbb{N}$, we let $t^n \defi n \Delta t$.}

{
\begin{rem}[Linear schemes and nonhomogeneous data] \label{rem:ls}
  The linearity of Problem~\eqref{evol:mixed:ino} implies that, (i) if $|\Gamma^D|>0$, the shifted variable $u^z \defi u - u^D$ (recall that $u^D$ is a known lifting of the Dirichlet datum $g^D$) satisfies an advection-diffusion equation with zero Dirichlet boundary condition on $\Gamma^D$, and (ii) otherwise, the shifted variable $u^z\defi u-u^M$ (recall that $u^M=\frac{M}{|\Omega|}$ is the mass lifting) satisfies a (compatible) pure Neumann advection-diffusion equation with zero-mass constraint. Thus, without loss of generality, we can restrict our study to the homogeneous case $g^D=0$ or $M=0$. For an example (in the steady, purely diffusive case) of how to handle at the discrete level nonhomogeneous data $g^D$ or $M$, we refer the reader to Problem~\eqref{eq:dis.pro}-\eqref{def.uinf} and Proposition~\ref{pro:wp} above.
  Notice that such manipulations are possible for linear schemes only.
\end{rem}
}

\subsection{Standard HFV scheme}

We present here the HFV variant of the HMM family of schemes introduced in \cite{BdVDM:11,Droni:10}.

\subsubsection{Stationary problem} \label{sse:hfv-sta}

{We consider Problem \eqref{sta:mixed:ino} with $g^D = 0$ when $|\Gamma^D|>0$, or $M=\int_{\Omega}u^\infty=0$ otherwise (cf.~Remark~\ref{rem:ls}).}
Locally to any cell $K\in\M$, we introduce the discrete bilinear form $a_K:\V_K\times\V_K\to\R$ such that, for all $\u_K,\v_K\in\V_K$, 
\begin{equation} \label{eq:aK}
	{a_K(\u_K,\v_K) \defi a_K^{\Lambda}(\u_K,\v_K)+a_K^{\phi}(\u_K,\v_K)},
\end{equation}
where the diffusive part $a_K^{\Lambda}$ is defined by~\eqref{def:aKL} (and rewrites as~\eqref{aKL-flux} in terms of the local diffusive fluxes $F_{K,\s}^\Lambda(\u_K)$ given by~\eqref{diff.flux}), and the advective part $a^{\phi}_K$ is defined by
\begin{equation} \label{aKV-flux}
  a^\phi_K(\u_K,\v_K)\defi\sum_{\s\in\E_K}F_{K,\s}^{\phi}(\u_K)(v_K-v_\s),
\end{equation}
with $F^{\phi}_{K,\s}(\u_K)$ an approximation of the normal advective flux $\int_\s u^\infty V^\phi\cdot n_{K,\s}$.
In order to define the numerical advective fluxes, we need to introduce some data. We set $V^\phi_{K,\s} \defi \frac{1}{|\s|} \int_\s V^\phi\cdot n_{K, \s}$, and $\mu_\s \defi\min\big(1,\min_{K\in\M_\s}{\rm Sp}(\Lambda_K)\big)>0$ where $\Lambda_K\defi\frac{1}{|K|}\int_K \Lambda$. {We could as well use the finer local P\'eclet number introduced in \cite{ChBo:21}, namely consider the value $\mu_\s \defi\min\big(1,\min_{K\in\M_\s}(n_{K,\s}\cdot \Lambda_K n_{K,\s})\big)$, but we choose here to stick to the formula advocated in~\cite{BdVDM:11,Droni:10}.}
We also consider a Lipschitz continuous function $A : \R \to \R $,
 satisfying the following conditions:
\begin{equation}\label{condB}
\begin{array}{r }
A(0) = 0, \\
\forall s \in \R , \; A(-s) - A(s) = s,\\
\forall s \in \R , \;  A(-s) + A(s) \geq 0.
\end{array}
\end{equation}
{Notice that $A = B-1$, where $B$ is the classical function used for the $B$-schemes introduced in \cite{CHDr:11}. In the $B$-schemes framework, advection and diffusion are simultaneously treated in the definition of the numerical flux. Here, as in~\cite{BdVDM:11}, only the advective part is considered, whence the fact that $A = B-1$.}
Standard choices of $A$ functions include: 
\begin{itemize}
\item the centred discretisation: $A : s \mapsto -\frac{s}{2}$;
\item the upwind discretisation: $A : s \mapsto \max(-s,0)$;
\item the \SG discretisation: $A : s \mapsto 
\left \{
			\begin{array}{r }
		\frac{s}{\e^s -1} -1\;\text{ if } s \neq 0  \\
		0 \;\text{ if } s = 0
			\end{array}
	\right . .
$
\end{itemize}
We eventually define, for all $K\in\M$, and all $\s\in\E_K$, the numerical advective flux: for all $\u_K \in \V_K$,
\begin{equation} \label{def_conv}
	{F_{K,\s} ^{\phi}} (\u_K) \defi 
		|\s|\frac{\mu_\s }{d_{K,\s}} \left ( 
		A \left ( - \frac{d_{K,\s}}{ \mu_\s}  V^\phi_{K,\s} \right ) u_K 
		-  A \left ( \frac{d_{K,\s}}{ \mu_\s}  V^\phi_{K,\s} \right ) u_\s  \right ).
\end{equation}
Letting $a_{\D}:\V_{\D}\times\V_{\D}\to\R$ be the (global) discrete bilinear form such that, for all $\u_\D,\v_\D\in\V_\D$,
\begin{equation} \label{eq:aD}
  a_\D(\u_\D,\v_\D)\defi\sum_{K\in\M}a_K(\u_K,\v_K),
\end{equation}
and recalling that $\V_{\D,0}$ denotes either $\V_{\D,0}^N$ whenever $|\Gamma^D|=0$, or $\V_{\D,0}^D$ otherwise, the discrete problem reads: Find $\u^\infty_\D\in\V_{\D,0}$ such that
\begin{equation} \label{sccomp3}
  a_\D (\u^\infty_\D, \v_\D)  = (f, v_\M)_\Omega + (g^N, v_\E)_{\Gamma^N}\qquad\forall\v_\D\in\V_{\D,0}.
\end{equation}
Remark that, for pure Neumann boundary conditions with $M\neq 0$, as opposed to the purely diffusive case of Problem~\eqref{eq:dis.pro}, $a_\D(\u^M_\D,\v_\D)\neq 0$ a priori for $\v_\D\in\V^N_{\D,0}$.
The well-posedness of~\eqref{sccomp3} is discussed in the following proposition.
\begin{prop}[Well-posedness] \label{prop:coerc}
  Let $A$ be a Lipschitz continuous function satisfying~\eqref{condB}. If the advection field {$V^\phi$} satisfies the two following conditions:
  \begin{subequations}\label{cond:coerc}
    \begin{align}
      & (i) \text{ almost everywhere on } \Gamma^N, \,  V^\phi \cdot n  \leq 0, \label{cond:coerc1}  \\
      & (ii) \text{ there exists } \beta <\frac{2 \lambda_\flat\alpha_\flat}{C_{P}^2} \text{such that, almost everywhere in } 
          \Omega, \, \divergence V^\phi \geq -\beta, \label{cond:coerc2} 
    \end{align}
  \end{subequations}
  where $\lambda_\flat\alpha_\flat$ is the coercivity constant of~\eqref{def:globalcoercivity}, and $C_P$ is either equal to $C_{PW}$ if $|\Gamma^D|=0$ or to $C_{P,\Gamma^D}$ otherwise (where $C_{PW},C_{P,\Gamma^D}$ are the Poincar\'e constants of Proposition~\ref{prop:poinca}), then there exists $\kappa >0$, only depending on $\Lambda$, $\beta$, $\Omega$, $d$, $\Gamma^D$, and $\theta_\D$ such that
  \begin{equation} \label{def:kappa}
    \forall \v_\D \in \V_{\D,0}, \qquad  a_\D (\v_\D, \v_\D)  \geq \kappa |\v_{\D}|^2_{1,\D}.
  \end{equation}
  Consequently, there exists a unique solution $\u^\infty_\D\in\V_{\D,0}$ to Problem~\eqref{sccomp3}. {Moreover, one has $|\u^\infty_\D|_{1,\D}\leq C\big(\|f\|_{L^2(\Omega)}+\|g^N\|_{L^2(\Gamma^N)}\big)$, for some $C>0$ depending on the data, and on the discretisation $\D$ only through $\theta_\D$.}
\end{prop}
\begin{proof}
  Let $K\in\M$, and $\s\in\E_K$. Let $s_{K,\s}\defi\frac{d_{K,\s}}{\mu_\s}V^\phi_{K,\s}$ and $\zeta_{K,\s}\defi A(-s_{K,\s})+A(s_{K,\s})$. According to~\eqref{condB}, $\zeta_{K,\s}\geq 0$, and we have $A(-s_{K,\s})=\frac{s_{K,\s}+\zeta_{K,\s}}{2}$ and $-A(s_{K,\s})=\frac{s_{K,\s}-\zeta_{K,\s}}{2}$. Consequently, for all $\v_K\in\V_K$,
  $$\big(A(-s_{K,\s})v_K-A(s_{K,\s})v_\s\big)(v_K-v_\s)=\frac{1}{2}s_{K,\s}(v_K^2-v_\s^2)+\frac{1}{2}\zeta_{K,\s}(v_K-v_\s)^2\geq\frac{1}{2}s_{K,\s}(v_K^2-v_\s^2).$$
  Recalling~\eqref{aKV-flux} and~\eqref{def_conv}, we infer that, for all $\v_\D\in\V_\D$,
  $$a^\phi_{\D}(\v_\D,\v_\D)\defi\sum_{K\in\M}a_K^\phi(\v_K,\v_K)\geq\frac{1}{2}\sum_{K\in\M}\sum_{\s\in\E_K}|\s|V_{K,\s}^\phi(v_K^2-v_\s^2).$$
  Since, for all $K\in\M$, $\sum_{\s\in\E_K}|\s|V^\phi_{K,\s}=\int_K\divergence V^\phi$ and, for all $\s\in\E_{int}$, $\sum_{K\in\M_\s}|\s|V^\phi_{K,\s}=0$, we have, for all $\v_\D\in\V_{\D,0}$,
  \begin{equation} \label{lemma:coerc}
    a_\D^\phi(\v_\D,\v_\D)\geq\frac{1}{2}(\divergence V^\phi,v_\M^2)_\Omega-\frac{1}{2}(V^\phi\cdot n,v_\E^2)_{\Gamma^N}.
  \end{equation}
  Combining {\eqref{eq:aK} with~\eqref{eq:aD}, \eqref{cond:coerc} with~\eqref{lemma:coerc}}, and the coercivity result~\eqref{def:globalcoercivity}, we deduce that, for all $\v_\D\in\V_{\D,0}$,
  \begin{equation} \label{garding}
    a_\D(\v_\D,\v_\D) \geq \lambda_\flat\alpha_\flat |\v_\D|_{1,\D}^2 - \frac{\beta}{2} \|v_\M\|_{L^2(\Omega)}^2.
  \end{equation}
  Using a Poincar\'e inequality from Proposition~\ref{prop:poinca}, one has, for all $\v_\D\in\V_{\D,0}$,
  $$a_\D(\v_\D,\v_\D)\geq \left(\lambda_\flat\alpha_\flat - \frac{\beta C_{P}^2}{2}\right)|\v_\D|_{1,\D}^2,$$
  therefore the estimate~\eqref{def:kappa} holds for $\kappa = \lambda_\flat\alpha_\flat - \frac{\beta C_{P}^2}{2}$, which is positive according to~\eqref{cond:coerc2}.
  The existence/uniqueness of $\u^\infty_\D\in\V_{\D,0}$ solution to~\eqref{sccomp3} is a direct consequence of the coercivity estimate~\eqref{def:kappa}, and of the fact that $|\cdot|_{1,\D}$ defines a norm on $\V_{\D,0}$. The continuous dependency of $\u^\infty_{\D}$ with respect to the data can then be proved as in the proof of Proposition~\ref{pro:wp}. 
\end{proof}
{
\begin{rem}[Assumptions on the advection field] \label{rem:ass}
  The well-posedness result of Proposition~\ref{prop:coerc} %, as well as the long-time asymptotic result of Theorem~\ref{Th:hmm} below, 
  does not use the fact that $V^\phi$ is related to the gradient of a potential, and thus extends to general advection fields. Even better, under a smallness assumption on the meshsize $h_\D$, it is actually possible to prove well-posedness for Problem~\eqref{sccomp3} without assumptions~\eqref{cond:coerc} on the advection field. The starting point to prove so is a discrete G\aa{}rding inequality like~\eqref{garding}, which can be easily obtained in full generality from~\eqref{lemma:coerc} (in the case $|\Gamma^N|>0$, it is obtained from the multiplicative discrete trace inequality of~\cite[Eq.~(B.57), $p=2$]{DEGGH:18} and holds for $h_\D$ sufficiently small). The proof then proceeds by contradiction, assuming that a discrete inf-sup condition does not hold in the limit $h_\D\to 0$, and using a compactness argument (cf.~\cite[Lemmas B.27-B.33, $p=2$]{DEGGH:18}){, together with the unconditional well-posedness of the continuous problem~\eqref{sta:mixed:ino} (cf.~\cite{Droni:02})}.
%  
%   Notice however that, even if well-posedness for Problem~\eqref{sccomp3} can be obtained (for small mesh sizes) without assumptions~\eqref{cond:coerc}, the analysis of the long-time behaviour of discrete solutions does not seem straightforward without these assumptions because some notion of coercivity seems necessary to conclude. The two other HFV schemes described in this article, at the price of being restricted to advection fields of the form $V^\phi$, are not subjected to such constraints.
\end{rem}
}
\begin{rem}[Choice of $A$] \label{rem:SG}
  The choice of the function $A$ is of great importance. In particular, the \SG approximation is rather classical in various contexts. First introduced in \cite{SchGu:69} in the framework of TPFA schemes, this approximation of the flux ensures the preservation of the so-called thermal (or Gibbs) equilibrium at the discrete level, which has the form: 
  \begin{equation} \label{eq:therm}
    u^\infty_{th}={\bar \rho}\e^{-\phi}
  \end{equation}
  where ${\bar \rho} \in \R_+$ is prescribed by the data. For instance, for pure Neumann boundary conditions, $f=0$, and $g^N=0$, we have ${\bar \rho} = M/\int_\Omega \e^{-\phi}$.   The discrete solution obtained with the TPFA scheme is then the interpolate of $u_{th}^\infty$. This property is no more true for the hybrid scheme and we observe numerically that, for ${\bar \rho}\neq 0$, the discrete solution $\u^\infty_\D\in\V_\D$ is in general not the HFV interpolate of $u_{th}^\infty$.
However, as explained in~\cite[pp.~553-554]{Droni:10}, provided the parameters $(\mu_\s)_{\s \in \E}$ are well-chosen, the \SG flux ensures an automatic upwinding to the scheme in the advection-dominated regime (whereas it degenerates towards the centred scheme in the diffusion-dominated regime).
\end{rem}

\subsubsection{Evolution problem}

{We consider Problem~\eqref{evol:mixed:ino} with $g^D=0$ when $|\Gamma^D|>0$, or $M=\int_{\Omega}u^{in}=0$ otherwise (see Remark~\ref{rem:ls}).}
We use a backward Euler discretisation in time, and the HFV discretisation introduced in Section~\ref{sse:hfv-sta} in space. The discrete problem reads: Find $\left(\u_{\D}^n\in\V_{\D,0}\right)_{n\geq 1}$ such that
\begin{subequations}\label{sch:HMM}
        \begin{empheq}[left = \empheqlbrace]{align}
	            \frac{1}{\Delta t } 
    	       (u^{n}_\M - u^{n-1}_\M, v_\M)_{\Omega} 
	            +  a_\D (\u_\D^{n}, \v_\D)
	             &=  (f, v_\M)_{\Omega} + (g^N, v_\E)_{\Gamma^N}&\quad&\forall\v_{\D}\in\V_{\D,0},\label{sch1:HMM}  \\
            u^0_K &= \frac{1}{|K|}\int_K u^{in } &\quad&\forall K\in\M, \label{sch2:HMM} 
        \end{empheq}
\end{subequations}
where $a_\D$ is defined by~\eqref{eq:aK} and \eqref{eq:aD}. Since $a_\D$ is coercive, the bilinear form in \eqref{sch1:HMM} is also coercive, so the scheme~\eqref{sch:HMM} is well-posed under the assumptions~\eqref{cond:coerc} on the advection field, as a straightforward consequence of Proposition~\ref{prop:coerc}.

\begin{rem}[Pure Neumann case] \label{rem:Neumannevol}
  When considering pure Neumann boundary conditions, and contrary to the stationary case, one can actually seek at each time step for a solution to Problem~\eqref{sch:HMM} in $\V_\D$, i.e., it is not necessary to seek for a solution in the constrained space $\V_{\D,0}^N$. Indeed, testing~\eqref{sch1:HMM} by $\underline{1}_\D\in\V_\D$, and using that $\int_{\Omega}f+\int_{\partial\Omega}g^N=0$ and $a_\D(\u_\D^n,\underline{1}_\D)=0$, one can automatically infer that $\int_{\Omega}u_\M^n=\int_{\Omega}u_\M^{n-1}$ for all $n\geq 1$, that is, $\int_{\Omega}u_{\M}^n=\int_{\Omega}u_\M^0=\int_{\Omega}u^{in}=M=0$ for all $n\geq 1$, i.e., $\u^n_\D\in\V^N_{\D,0}$ for all $n\in\mathbb{N}^\star$.
\end{rem}

\subsection{Exponential fitting HFV scheme}

Following ideas in~\cite{BMaPi:89} {(cf.~also~\cite{LBLMa:16,Madio:16} and~\cite{FiHer:17} for general advection fields)} in the context of finite element methods, we aim to design an unconditionally (i.e., without the need for assumptions~\eqref{cond:coerc} on the advection field $V^\phi$) {coercive} scheme for the advection-diffusion problem in the HFV framework.

\subsubsection{Stationary problem} \label{subsubsec:omega:sta}

We consider Problem~\eqref{sta:mixed:ino}. The strategy advocated in~\cite{BMaPi:89} is based on the following observation: at the continuous level, if $u^\infty$ is a solution to~\eqref{sta:mixed:ino}, letting
\begin{equation} \label{def:ome}
  \omega \defi \e^{-\phi},
\end{equation}
we can introduce the Slotboom change of variable $\rho^\infty \defi \frac{u^\infty}{\omega}$ (see \cite{Marko:86,MRiSc:90}). Then, noticing that $\nabla u^\infty + u^\infty \nabla \phi = \omega \nabla \rho^\infty -  \rho^\infty \omega \nabla \phi +\rho^\infty\omega \nabla \phi =\omega \nabla \rho^\infty$, the Slotboom variable $\rho^\infty$ equivalently solves the following pure diffusion problem:
{
\begin{equation} \label{def:sta:modeq} 
	\left\{
	\begin{split}
		- \divergence \left ( \omega\Lambda \nabla \rho^\infty\right ) &= f &&\text{ in } \Omega, \\
		 \rho^\infty  &= \omega^{-1}g^D&&\text{ on } \Gamma^D, \\
		 \omega\Lambda \nabla \rho^\infty \cdot n &= g^N &&\text{ on } \Gamma^N, 
	\end{split}
	\right.
\end{equation}
with additional constraint $\int_{\Omega}\omega\rho^\infty=M$ when $|\Gamma^D|=0$.
Following Remark~\ref{rem:ls} (with $\rho$ instead of $u$, $\rho^D$ lifting of $\omega^{-1}g^D$ instead of $u^D$ lifting of $g^D$, and $\rho^M\defi\frac{M}{\int_{\Omega}\omega}\in\R$ instead of $u^M$), we consider Problem~\eqref{def:sta:modeq} with $\rho^\infty= 0$ on $\Gamma^D$ when $|\Gamma^D|>0$, or $\int_{\Omega}\omega\rho^\infty=0$ otherwise (which is equivalent to consider Problem~\eqref{sta:mixed:ino} with $g^D=0$ or $M=\int_\Omega u^\infty=0$).}
Since $\phi$ is continuous on $\overline{\Omega}$, there exist {$\omega_\flat,\omega_\sharp$} with $0<\omega_\flat\leq\omega_\sharp<\infty$, only depending on $\phi$ and $\Omega$, such that $\omega_\flat\leq\omega(x)\leq\omega_\sharp$ for all $x\in\overline{\Omega}$. We then denote by $L^2_\omega(\Omega)$ the $\omega$-weighted $L^2$ space on $\Omega$.

At the discrete level, instead of discretising~\eqref{sta:mixed:ino}, we approximate the solution to~\eqref{def:sta:modeq}.
For any $K\in\M$, we let $a^\omega_K : \V_K \times \V_K \to \R$ be the discrete bilinear form such that, for all $\ro_K,\v_K\in\V_K$,
\begin{equation} \label{def:omega:bilin}
	a^\omega_K (\ro_K,\v_K) \defi (\omega\Lambda\nabla_K\ro_K,\nabla_K\v_K)_K,
\end{equation}
and, classically, we let $a_\D^\omega:\V_\D\times\V_\D\to\R$ be the corresponding global discrete bilinear form obtained by sum of the local contributions.
To account for the change of variable, we let {$\V_{\D,0}^\omega$} be the space $\V_{\D,0}^D$ when $|\Gamma^D|>0$, and the space $\left \lbrace \v_{\D} \in \V_{\D} : \int_{\Omega}\omega v_{\M} = 0 \right \rbrace$ otherwise.
The discrete problem reads: Find $\ro_\D^\infty\in\V^\omega_{\D,0}$ such that 
\begin{equation} \label{def:omega:sta}
	 a_\D ^\omega (\ro_\D^\infty, \v_\D) 
	= (f, v_\M)_{\Omega} 
		+ (g^N, v_\E)_{\Gamma^N}\qquad\forall\v_\D\in\V^\omega_{\D,0}.
\end{equation}
Remark that, for pure Neumann boundary conditions with $M\neq 0$, as for Problem~\eqref{eq:dis.pro}, letting $\ro_\D^M\defi\rho^M\one_\D$, $a_\D^\omega(\ro^M_\D,\v_\D)=0$ for all $\v_\D\in\V_\D$.
Letting $\underline{\omega}_\D\in\V_\D$ be the HFV interpolate of $\omega$, i.e.,
\begin{equation} \label{def-omega}
  \omega_K \defi \frac{1}{|K|} \int_K \omega\quad\forall K\in\M, \qquad  \omega_\s \defi\frac{1}{|\s|} \int_\s \omega\quad\forall\s\in\E,
\end{equation}
the approximation of the solution to Problem~\eqref{sta:mixed:ino} is finally defined as the product $\u^\infty_\D\defi\underline{\omega}_\D\times\ro^\infty_\D$, that is
\begin{equation} \label{def:omega:u}
	u^\infty_K = \omega_K\rho_K^\infty\quad\forall K\in\M,  \qquad u_\s^\infty = \omega_\s\rho_\s^\infty \quad\forall\s\in\E.
\end{equation}
Remark that $\u_{\D}^\infty\in\V_{\D,0}$.
Reasoning as in Section~\ref{sse:wp}, one can easily prove that, for all $\v_\D\in\V_\D$,
\begin{equation} \label{ef-coer}
  a_\D^\omega (\v_\D, \v_\D) \geq \omega_\flat \lambda_\flat \alpha_\flat\, |\v_\D|_{1,\D}^2.
\end{equation}
This estimate is instrumental to infer well-posedness for Problem~\eqref{def:omega:sta}-\eqref{def:omega:u}.
{
\begin{prop}[Well-posedness] \label{pro:wpef}
  There exists a unique solution $\u^\infty_\D\in\V_{\D,0}$ to Problem~\eqref{def:omega:sta}-\eqref{def:omega:u}, which satisfies $|\u^\infty_{\D}|_{1,\D}\leq C\left(\|f\|_{L^2(\Omega)}+\|g^N\|_{L^2(\Gamma^N)}\right)$, for some $C>0$ depending on the data, and on the discretisation $\D$ only through $\theta_\D$.
\end{prop}
\begin{proof} 
  The existence/uniqueness of $\ro^\infty_{\D}\in\V^\omega_{\D,0}$ solution to~\eqref{def:omega:sta} (and in turn of $\u^\infty_{\D}\in\V_{\D,0}$) is a direct consequence of the coercivity estimate~\eqref{ef-coer}, and of the fact that $|\cdot|_{1,\D}$ clearly defines a norm on $\V^\omega_{\D,0}$. To prove the bound on $|\u^\infty_{\D}|_{1,\D}$, we use the fact that $\u_\D^\infty=\underline{\omega}_\D\times\ro_\D^\infty$. For all $K\in\M$,
  $$|\u^\infty_K|_{1,K}^2=\sum_{\s\in\E_K}\frac{|\s|}{d_{K,\s}}(\omega_K\rho_K^\infty-\omega_\s\rho_\s^\infty)^2\leq 2\sum_{\s\in\E_K}\frac{|\s|}{d_{K,\s}}\left(\omega_\s^2(\rho^\infty_K-\rho^\infty_\s)^2 + (\rho_K^\infty)^2(\omega_K-\omega_\s)^2\right).$$
  By definition~\eqref{def-omega} of $\underline{\omega}_\D$, and local stability of the HFV interpolant (cf.~\cite[Proposition B.7, $p=2$]{DEGGH:18} and its proof), we infer
  $$|\u^\infty_K|_{1,K}^2\leq 2\,\omega_\sharp^2\,|\ro^\infty_K|_{1,K}^2+2\,(\rho_K^\infty)^2C_{sta}^2\|\nabla\omega\|_{L^2(K;\R^d)}^2,$$
  with $C_{sta}>0$ only depending on $d$ and $\theta_\D$.
  Since $\nabla\omega=-\omega\nabla\phi$, $\phi\in C^1(\overline{\Omega})$, and $\omega>0$, we have
  $$\|\nabla\omega\|_{L^2(K;\R^d)}^2\leq\omega_\sharp\,\sup_{x\in\overline{\Omega}}|\nabla\phi(x)|^2|K|\,\omega_K.$$
  Summing over $K\in\M$ then yields
  $$|\u_\D^\infty|_{1,\D}^2\leq 2\,\omega_\sharp^2\,|\ro^\infty_\D|_{1,\D}^2+2\,C_{sta}^2\,\omega_\sharp\,\sup_{x\in\overline{\Omega}}|\nabla\phi(x)|^2\|\rho^\infty_\M\|_{L^2_\omega(\Omega)}^2.$$
  When $|\Gamma^D|>0$, $\|\rho^\infty_\M\|_{L^2_\omega(\Omega)}^2\leq\omega_\sharp\|\rho^\infty_\M\|_{L^2(\Omega)}^2\leq\omega_\sharp \,C_{P,\Gamma^D}^2|\ro^\infty_\D|_{1,\D}^2$, where we have applied the discrete Poincar\'e inequality~\eqref{poincadir} to $\ro_\D^\infty\in\V_{\D,0}^D$. Otherwise, $\ro_\D^\infty\in\V_{\D,0}^\omega$ satisfies $\int_{\Omega}\omega\rho^\infty_\M=0$, and one can use~\cite[Lemma 5.2]{CCHHK:20} to infer that $\|\rho^\infty_\M\|_{L^2_\omega(\Omega)}\leq 2\|\rho^\infty_\M-\frac{1}{|\Omega|}\int_{\Omega}\rho^\infty_\M\|_{L^2_\omega(\Omega)}$, and finally get that $\|\rho^\infty_\M\|_{L^2_\omega(\Omega)}\leq 2\,\sqrt{\omega_\sharp}\,C_{PW}|\ro^\infty_\D|_{1,\D}$ applying the discrete Poincar\'e inequality~\eqref{poincawir} to $\ro^\infty_\D-\frac{1}{|\Omega|}\int_{\Omega}\rho^\infty_\M\underline{1}_\D\in\V_{\D,0}^N$.
  In any case, we end up bounding $|\u_\D^\infty|_{1,\D}$ by $|\ro^\infty_\D|_{1,\D}$, with multiplicative constant depending on the data, and on the discretisation $\D$ only through $\theta_\D$. The rest of the proof consists in bounding $|\ro_\D^\infty|_{1,\D}$, and proceeds as in the proof of Proposition~\ref{pro:wp}, using that $\|\rho^\infty_\M\|_{L^2(\Omega)}\leq 2\,\sqrt{\frac{\omega_\sharp}{\omega_\flat}}\,C_{PW}|\ro^\infty_\D|_{1,\D}$ when $|\Gamma^D|=0$.
\end{proof}
}
\noindent
In the sequel, the HFV scheme \eqref{def:omega:sta}-\eqref{def:omega:u} will be referred to as ``exponential fitting". Notice that no assumption on $V^\phi$ is needed to ensure its coercivity.

\begin{rem}[Preservation of the thermal equilibrium] \label{rem:thermeq}
  As for the exponential fitting scheme of~\cite{BMaPi:89}, the method \eqref{def:omega:sta}-\eqref{def:omega:u} preserves the thermal equilibrium, in the sense that $\u_\D^\infty={\bar \rho}\,\underline{\omega}_\D$ when $u^\infty=u^\infty_{th}$ (see~\eqref{eq:therm}). 
%  Indeed, for pure Neumann boundary conditions, and a choice $f=0$, $g^N=0$, the solution to~\eqref{def:omega:sta} is $\ro_\D^\infty=\underline{0}_\D$, hence $\u_\D^\infty=\underline{0}_\D$. By Remark~\ref{rem:ls}, for $M\neq 0$, $\ro_\D^\infty=\ro_\D^M=\frac{M}{\int_{\Omega}\omega}\underline{1}_\D$, and hence {$\u^\infty_\D=\frac{M}{\int_{\Omega}\omega}\underline{\omega}_\D$}. Thus, $\u^\infty_\D$ is the HFV interpolate of $u^\infty_{th}$ given by~\eqref{eq:therm}. 
  This property is analogous to what holds true for the TPFA \SG scheme.
  %% , and is of great importance in our context. Indeed, if the evolutive scheme has a good long-time behaviour, the preservation of the thermal equilibrium ensures that the discrete solution converges (in time) towards an interpolation of the exact solution.
\end{rem}

\subsubsection{Evolution problem}

We consider Problem~\eqref{evol:mixed:ino}.
Following the previous strategy, letting $\rho\defi\frac{u}{\omega}$, one can show that $\rho$ equivalently solves the following transient pure diffusion problem:
\begin{equation} \label{def:evol:modeq}
	\left\{
	\begin{split}
		\omega\partial _ t \rho - \divergence ( \omega\Lambda  \nabla \rho ) &= f &&\text{ in } \R_+ \times \Omega, \\
		\rho &=  \omega^{-1}g^D&&\text{ on }  \R_+ \times \Gamma^D, \\
		\omega\Lambda \nabla \rho \cdot n &= g^N &&\text{ on } \R_+ \times \Gamma^N,\\		
	 	\rho(0,\cdot)&= \rho^{in} &&\text{ in } \Omega,
	\end{split}
	\right.
\end{equation}
where $\rho^{in}\defi\frac{u^{in}}{\omega}$. {Following Remark~\ref{rem:ls}, we consider Problem~\eqref{def:evol:modeq} with $\rho=0$ on $\Gamma^D$ when $|\Gamma^D|>0$, or $\int_{\Omega}\omega\rho^{in}=0$ otherwise (which is equivalent to consider Problem~\eqref{evol:mixed:ino} with $g^D=0$ or $M=\int_{\Omega}u^{in}=0$).}

At the discrete level, instead of discretising~\eqref{evol:mixed:ino}, we approximate the solution to~\eqref{def:evol:modeq}. We use a backward Euler discretisation in time, and the HFV discretisation introduced in Section~\ref{subsubsec:omega:sta} in space. The discrete problem reads: Find $\big(\ro_\D^n\in\V_{\D,0}^\omega\big)_{n\geq 1}$ such that
\begin{subequations}\label{sch:omega:evol}
        \begin{empheq}[left = \empheqlbrace]{align}
	            \frac{1}{\Delta t } 
    	       (\rho^{n}_\M - \rho^{n-1}_\M, \omega_\M v_\M)_{\Omega} 
	            +  a^\omega_\D (\ro_\D^{n}, \v_\D)
	             &=  (f, v_\M)_{\Omega} + (g^N, v_\E)_{\Gamma^N}&\quad&\forall\v_{\D}\in\V^\omega_{\D,0},\label{sch1:omega}  \\
            \rho^0_K &= \frac{1}{\omega_K|K|}\int_K u^{in } &\quad&\forall K\in\M, \label{sch2:omega} 
        \end{empheq}
\end{subequations}
where $a^\omega_\D$ is defined (locally) by~\eqref{def:omega:bilin}, and the approximation $\u^n_\D\in\V_{\D,0}$ of the solution to Problem~\eqref{evol:mixed:ino} is finally defined as $\u^n_\D\defi\underline{\omega}_\D\times\ro_\D^n$, i.e., according to~\eqref{def:omega:u} (with superscript $n$ instead of $\infty$). Once again, because of the coercivity of $a_\D^\omega$, the scheme~\eqref{sch:omega:evol} is unconditionally well-posed, as a straightforward consequence of Proposition~\ref{pro:wpef}.

\begin{rem}[Pure Neumann case] \label{rem:omega:ev:neu}
{When considering pure Neumann boundary conditions, as for the standard HFV scheme (see Remark \ref{rem:Neumannevol}), one can seek for a solution to Problem~\eqref{sch:omega:evol} in $\V_\D$.}
\end{rem}

\subsection{Nonlinear HFV scheme} \label{sse:nonlin}

We are interested in the evolution problem~\eqref{evol:mixed:ino}. We restrict our study to the pure Neumann case ($|\Gamma^D|=0$), and to the choice of data $f=0$, $g^N=0$, and $u^{in}\geq 0$ with $M=\int_{\Omega}u^{in}>0$ {(see Appendix~\ref{Ap:NLmixted} for the case of mixed Dirichlet-Neumann thermal equilibrium boundary conditions)}.
Under these assumptions, it is known that the solution $u$ to Problem~\eqref{evol:mixed:ino} is {strictly positive} on $\R_+^\star \times \Omega$.
Furthermore, in long time, $u(t)$ converges towards the thermal equilibrium. Indeed, we easily verify that the function $u^\infty=\rho^M\e^{-\phi}$, where we recall that $\rho^M=\frac{M}{\int_{\Omega}\e^{-\phi}}>0$ (cf.~Section~\ref{subsubsec:omega:sta}), solves the steady problem~\eqref{sta:mixed:ino} with same data.
Since $u>0$ on $\R_+^\star\times\Omega$, we can rewrite the flux $J = - \Lambda (\nabla u  + u \nabla\phi)$ under the nonlinear form
$$J = -u \Lambda \nabla \big( \log(u) +\phi\big)={-u \Lambda \nabla \log\left(\frac{u}{u^\infty}\right)}.$$
At the continuous level, introducing this nonlinearity enables to highlight the following entropy/dissipation structure of the model at hand: testing the equation against $\log\left(\frac{u}{u^\infty}\right)$, we get
\begin{equation} \label{nonlin:relation}
  \frac{{\rm d}}{{\rm d}t} \N (t) + \Diss(t) = 0,
\end{equation}
where 
\begin{equation}\label{nonlin:entro}
		\N(t) \defi \int_\Omega u^\infty\,\Phi_1\left(\frac{u(t)}{u^\infty}\right) \qquad\text{ and }\qquad \Diss(t) \defi \int_\Omega u(t)  \,\Lambda \nabla \log\left(\frac{u(t)}{u^\infty}\right) \cdot \nabla \log\left(\frac{u(t)}{u^\infty}\right),
\end{equation}
with $\Phi_1 (s) = s \log (s) - s + 1$ for all $s>0$. Since $\Phi_1\geq 0$, the relative entropy $\N(t)$ is a non-negative quantity (as well as the relative dissipation $\Diss(t)$). The entropy/dissipation structure~\eqref{nonlin:relation}-\eqref{nonlin:entro} is instrumental to prove the exponential convergence in time of the solution $u(t)$ to Problem~\eqref{evol:mixed:ino} towards the equilibrium $u^\infty$.
From the above nonlinear expression of the flux $J$, we build a nonlinear hybrid discretisation of the problem, leading to a scheme designed along the same principles as the nonlinear VAG and DDFV schemes of~\cite{CaGui:17} and~\cite{CCHKr:18,CCHHK:20}. This scheme is devised so as to ensure the positivity of discrete solutions, as well as to preserve at the discrete level the entropy/dissipation structure (and the long-time behaviour) of the model. The choice of designing a nonlinear HFV scheme is driven by the prospect of the design of hybrid high-order (HHO) schemes which could have similar features. 

\subsubsection{Definition of the scheme and key properties of discrete solutions}

In the sequel, a vector of discrete unknowns $\v_\D \in \V_\D$ will be called positive if and only if, for all $K \in \M$ and all $\s \in \E$, $v_K  > 0$ and $v_\s > 0$.
Recall the definition~\eqref{def:ome} of $\omega=\e^{-\phi}$, as well as the definition~\eqref{def-omega} of the HFV interpolate $\underline{\omega}_\D\in\V_\D$ of $\omega$. Remark that $\underline{\omega}_\D$ is positive.
If $\u_\D\in\V_\D$ is positive, one can then define $\w_\D$ as the element of $\V_\D$ such that
\begin{equation} \label{def:g}
  w_K\defi\log\left(\frac{u_K}{\omega_K}\right)\quad\forall K\in\M,\qquad w_\s\defi\log\left(\frac{u_\s}{\omega_\s}\right)\quad\forall\s\in\E.
\end{equation}
In what follows, to emphasise the dependency of $\w_\D$ upon $\u_\D$, we sometimes write $\w_\D(\u_\D)$.
%% The definition~\eqref{def:g} of $\w_\D$ will be justified a posteriori by~\eqref{sch3:w}.
Locally to any cell $K\in\M$, we define an approximation of
\begin{equation*}
 (u,v) \mapsto -\int_K J\cdot\nabla v= \int_K u \,\Lambda\nabla\log\left(\frac{u}{u^\infty}\right) \cdot  \nabla v
\end{equation*}
under the form
$${T_K(\u_K,\w_K,\v_K)\defi\int _K r_K(\u_K) \,\Lambda\nabla_K \w_K \cdot \nabla_K \v_K},$$
for all $\u_K\in\V_K$ positive and all $\v_K\in\V_K$, where {$r_K :\big(\V_K\big)_+^\star \to \R^\star_+$} is a local reconstruction operator.
Since $r_K(\u_K)$ is a (positive) constant on $K$, we have
\begin{equation} \label{eq:TK}
  T_K(\u_K,\w_K,\v_K)=r_K(\u_K) a^\Lambda_K\big(\w_K,\v_K\big),
\end{equation}
where $a^\Lambda_K$ is defined by~\eqref{def:aKL}. Following~\eqref{def:a_K}, one can equivalently reformulate~\eqref{eq:TK} using the local matrix $\A_K$ defined by~\eqref{def:flux:A}:
\begin{equation} \label{eq:TK-AK}
  T_K(\u_K,\w_K,\v_K)=r_K(\u_K)\,\delta_K\v_K\cdot\A_K\delta_K\w_K.
\end{equation}
As already pointed out in the analysis of the nonlinear DDFV scheme of \cite{CCHKr:18,CCHHK:20}, the definition of the local reconstruction operator is crucial to guarantee the existence of solutions and a good long-time behaviour to the scheme.
The most natural choice in the HFV context would obviously be $r_K(\u_K) = u_K$, however it turns out that such a reconstruction embeds too few information on $\u_K$ to conclude, as already suggested in \cite{Cances:18}. Therefore, we use a richer reconstruction, described below, which embeds information from both the local cell and face unknowns.
For $\u_K \in \V_K$ positive, we let
\begin{equation} \label{eq:rK}
  r_K (\u_K) \defi f_{|\E_K|} \left ( \big(m (u_K , u_\s) \big)_{\s \in \E_K} \right ),
\end{equation}
with {$m : (\R_+^\star)^2 \to \R ^\star_+ $} and, for $k\geq 1$ integer, $f_k :(\R_+^\star )^k  \to  \R_+^\star$, such that 
\begin{subequations}\label{hyp:g}
        \begin{align}
            & \mbox{$m$ is non-decreasing with respect to both its variables,}\label{hyp:g:monotony}\\
          & \mbox{$m(x,x)=x$ for all $x\in \R_+^\star$ and $m(y,x)=m(x,y)$ for all $(x,y)\in ( \R_+^\star)^2$,}\label{hyp:g:cons+cons}\\
            &\mbox{$m(\lambda x,\lambda y)=\lambda m(x,y)$  for all $\lambda >0$ and all $(x,y)\in( \R_+^\star )^2$,}\label{hyp:g:homogeneity}\\
            &\mbox{$\displaystyle\frac{y-x}{\log (y)-\log (x)} \leq m(x,y)\leq \max(x,y)$ for all $(x,y)\in ( \R_+^\star )^2$, $x\neq y$, \label{hyp:g:bounds}}
        \end{align}
\end{subequations}
and 
\begin{equation} \label{def:f}
		f_k (x_1, \ldots, x_k ) = \frac{1}{k} \sum_{i =1 } ^k x_i \qquad\text{ or  }\qquad f_k (x_1, \ldots ,x_k ) = \max (x_1, \ldots, x_k ).
\end{equation}
{Note that, for all $(x,y) \in ( \R_+^\star )^2$, one has 
\[
	\frac{y-x}{\log(y) - \log(x)} \leq \left ( \frac{\sqrt{x} + \sqrt{y}}{2} \right )^2 \leq \frac{x+y}{2} \leq \max(x,y), 
\]
and each expression of the previous sequence is a mean function $m$ satisfying the properties \eqref{hyp:g}.}
Heuristically, $r_K(\u_K)$ computes an average of the unknowns attached to the cell $K$, especially it contains information about all the local face unknowns.
%% {Such a choice was suggested in \cite{Cances:18}.}
As far as the properties~\eqref{hyp:g}-\eqref{def:f} are concerned, they will be instrumental to prove Lemma~\ref{lemma:positivity} and Proposition~\ref{prop:Fisher} below.
As now standard, we finally let $T_\D$ be such that, for all $\u_\D\in\V_\D$ positive, and all $\v_\D\in\V_\D$,
\begin{equation} \label{sch2}
  T_\D(\u_\D,\w_\D,\v_\D)\defi\sum_{K\in\M}T_K(\u_K,\w_K,\v_K),
\end{equation}
where the local contributions $T_K$ are defined by~\eqref{eq:TK}.

Using a backward Euler discretisation in time, and the HFV discretisation we have just introduced in space, our discrete problem reads: Find $\big(\u_\D^n\in\V_\D\big)_{n\geq 1}$ such that
\begin{subequations}\label{sch:nonlin}
        \begin{empheq}[left = \empheqlbrace]{align}
	            \frac{1}{\Delta t } 
    	       (u^{n}_\M - u^{n-1}_\M, v_\M)_{\Omega} 
	            +  T_\D (\u_\D^{n}, \w_\D(\u_\D^{n}), \v_\D)
	             &= 0\qquad\qquad\qquad\,\forall\v_{\D}\in\V_{\D},\label{sch1}  \\
            u^0_K &= \frac{1}{|K|}\int_K u^{in } \qquad\forall K\in\M. \label{sch4} 
        \end{empheq}
\end{subequations}
Notice that if $(\u_\D^n)_{n\geq 1}$ solves Problem~\eqref{sch:nonlin}, then, necessarily, $\u_\D^n$ is positive for all $n\geq 1$. Therefore, in the sequel, we will speak about the positive solutions to~\eqref{sch:nonlin}. Notice also that $u^0_\M$ may vanish in some cells of the mesh, since we only impose that $u^{in}\geq 0$ (but $u^0_\M$ cannot be identically zero in $\Omega$ since $M>0$). Notice finally that $u^0_\E$ needs not be defined, as the scheme only uses $u^0_\M$.
%{\begin{rem}[Practical realisation] \label{rem:nonlinrealisation}
%  In practice, one can consider instead of Problem~\eqref{sch:nonlin} a slightly simplified  and equivalent version in which one replaces the multiplicative constant $\rho^M$ by $1$.
%%   
%%   
%%   of the latter in which, locally to any $K\in\M$, $T_K$ is still defined as in~\eqref{eq:TK}, but with a different definition than~\eqref{def:g} for $\w_K(\u_K)$, in which one replaces the multiplicative constant $\rho^M$ by $1$. It is an easy matter to realise that the two schemes are actually identical. Indeed, if we introduce $\ro_\D\in\V_\D$ positive such that $\u_\D=\underline{\omega}_\D\times\ro_\D$ holds true, then we see that~\eqref{def:g} is equivalent to $\w_\D=\log(\ro_\D)-\log(\rho^M)\underline{1}_\D$. Since $\w_\D$ only intervenes in the discrete problem through its discrete gradient, then the constant contribution in $\w_\D$ can be seamlessly discarded (by replacing $\rho^M$ by $1$).
%\end{rem}}

Testing~\eqref{sch1} with $\v_\D=\underline{1}_\D$, and remarking that $T_\D(\u_\D^n,\w_\D(\u_\D^{n}),\underline{1}_\D)=0$ for all $n\geq 1$, we immediately infer the following discrete mass conservation property.
\begin{prop}[Mass conservation]\label{prop:mass}
  If $\big(\u^n_\D\in\V_\D\big) _{n \geq 1}$ is a (positive) solution to~\eqref{sch:nonlin}, then
  $$\forall n \in \mathbb{N}^\star, \qquad\int_\Omega u^{n}_\M = \int_\Omega u^0_\M = \int_\Omega u^{in} = M.$$
\end{prop}
\noindent
Following Proposition~\ref{prop:mass}, a discrete steady-state $\u^\infty_\D\in\V_\D$ of~\eqref{sch:nonlin} shall satisfy
\begin{equation} \label{def:nonlin:sta}
  T_\D (\u^\infty_\D, \w_\D(\u^\infty_\D), \v_\D) = 0\qquad\forall\v_\D\in\V_\D,
\end{equation}
and $\int_{\Omega}u^\infty_\M=M$. Letting $\w_\D^\infty\defi\w_\D(\u_\D^\infty)$, and testing~\eqref{def:nonlin:sta} with $\v_\D=\w_\D^\infty$, by~\eqref{eq:TK} and~\eqref{sch2}, since $r_K(\u_K^\infty)>0$ for all $K\in\M$, we necessarily have $a_\D^\Lambda(\w_\D^\infty,\w_\D^\infty)=0$, which yields $|\w_\D^\infty|_{1,\D}=0$ by the coercivity property~\eqref{def:globalcoercivity}. Hence, $\w_\D^\infty=c\,\underline{1}_\D$ for some constant $c\in\R$, and since $\int_{\Omega}u_\M^\infty=M$, by~\eqref{def:g}, we necessarily have $\e^c=\rho^M$, that is $c=\log(\rho^M)$ and $\w_\D^\infty=\log(\rho^M)\underline{1}_\D$. As a consequence, again by~\eqref{def:g}, $\u_\D^\infty=\rho^M\underline{\omega}_\D$, i.e., $\u_\D^\infty$ is the HFV interpolate of $u^\infty_{th}$. Thus, just like the exponential fitting scheme (cf.~Remark~\ref{rem:thermeq}), the nonlinear scheme preserves the thermal equilibrium. We notice that $\w_\D$, first defined by \eqref{def:g}, can actually be modified up to an additive constant without any impact on the scheme \eqref{sch:nonlin}. Hence, we can redefine $\w_\D$ as
\begin{equation} \label{sch3:w}
  w_K\defi\log\left(\frac{u_K}{u_K^\infty}\right)\quad\forall K\in\M,\qquad w_\s\defi\log\left(\frac{u_\s}{u_\s^\infty}\right)\quad\forall\s\in\E.
\end{equation}
%This justifies a posteriori the definition~\eqref{def:g} of $\w_\D$.}
Another important consequence of the fact that $\u_\D^\infty$ is the HFV interpolate of $u^\infty_{th}$ is the following. Letting {$u^\infty_\flat\defi\frac{M\omega_\flat}{|\Omega|\omega_\sharp}>0$ and $u^\infty_\sharp\defi\frac{M\omega_\sharp}{|\Omega|\omega_\flat}>0$} (recall that $\omega_\flat$ and $\omega_\sharp$ only depend on $\phi$ and $\Omega$), we have $u^\infty_\flat\leq u^\infty_{th}\leq u^\infty_\sharp$, but we also have that
\begin{equation} \label{bornesstationnaires}
  u^\infty_\flat\underline{1}_\D \leq \u^\infty_\D \leq u^\infty_\sharp\underline{1}_\D,
\end{equation}
where the inequalities shall be understood coordinate-wise. {In other words, the continuous bounds on the steady-state are transferred to the discrete level.}

Given a (positive) solution $\big(\u^n_\D\in\V_\D\big)_{n \geq 1}$ to~\eqref{sch:nonlin}, we define the following discrete versions of the relative entropy and dissipation introduced in~\eqref{nonlin:entro}: for all $n\geq 1$,
\begin{equation} \label{entro1}
		\N ^n \defi \int_\Omega u^\infty_\M \Phi_1 \left ( \frac{u^n_\M}{u^\infty_\M} \right ) \qquad\text{ and }\qquad
		\Diss^n \defi T_\D \big(\u_\D^n, \w_\D^n , \w_\D^n\big),
\end{equation}
where we let $\w_\D^n\defi\w_\D(\u_\D^n)$, and we recall that $\Phi_1 (s) = s \log (s) - s + 1$ for all $s>0$. Notice that $\N ^n \geq 0$ for all $n\geq 1$ since $\Phi_1\geq 0$. {For further use, we extend the function $\Phi_1$ by continuity to 0, letting $\Phi_1(0)=1$. We can then define, in case there exists $K\in\M$ such that $u^0_K=0$, $\N^0\geq 0$ according to~\eqref{entro1}.} As far as $\Diss^n$ is concerned, by~\eqref{eq:TK} and~\eqref{sch2}, for all $n\geq 1$, we have 
$$\Diss^n = \sum _{K \in \M } r_K(\u_K^n) a_K^\Lambda(\w_K^n,\w_K^n) \geq 0.$$
We can now establish the following discrete counterpart of~\eqref{nonlin:relation}.
\begin{prop}[Entropy dissipation]\label{prop:dissip}
	If $\big(\u^n_\D\in\V_\D\big) _{n \geq 1}$ is a (positive) solution to~\eqref{sch:nonlin}, then
	\begin{equation} \label{diss:nonlin}
		\forall n \in \mathbb{N}, \qquad \frac{\N^{n+1} - \N ^n }{\Delta t} + \Diss^{n+1} \leq 0. 
	\end{equation}
%% Moreover, we have the following estimates:
%% 	\begin{equation} \label{dec}
%% 		\forall n \in \mathbb{N} ,\qquad \N ^{n+1} \leq \N^n \leq \N ^0 \quad\text{ and }\quad \sum_{n \geq 1} \Diss^n < \frac{\N^0}{\Delta t}.
%% 	\end{equation}
\end{prop}
\begin{proof}
  Let $n\in\mathbb{N}$. By the expression~\eqref{entro1} of the discrete relative entropy, and the convexity of $\Phi_1$, we have
\[
	\N ^{n+1} - \N ^n 
	\leq \sum_{K\in\M}|K|u^\infty _K  \Phi_1' \left ( \frac{u^{n+1}_K}{u^\infty_K} \right )\left (\frac{u^{n+1}_K - u^n _K}{u^\infty_K}  \right ) .
\]
Thus, by~\eqref{sch3:w}, we get
\[
	\N ^{n+1} - \N ^n 
		\leq  \int_{\Omega}  \left ( u^{n+1}_\M - u^{n}_\M \right ) \log \left ( \frac{u^{n+1}_\M}{u^\infty _\M } \right )  
		=  (u^{n+1}_\M - u^n_\M, w^{n+1}_\M)_{\Omega} .
\]
By~\eqref{sch1} and~\eqref{entro1}, we finally infer
\[ \N ^{n+1} - \N ^n \leq - \Delta t \,T_\D (\u_\D^{n+1}, \w_\D^{n+1},\w_\D^{n+1}) = - \Delta t \,\Diss^{n+1}, 
\]
which yields~\eqref{diss:nonlin}.
%% From the non-negativity of $\Diss^n$, we deduce the entropy decay, and the second estimation of \eqref{dec} is obtained summing \eqref{diss:nonlin} for $n \in\mathbb{N}$ and noticing that $\N^\infty=0$. 
\end{proof}

We finally state the main result of Section~\ref{sse:nonlin}, about the existence of (positive) solutions to the nonlinear scheme~\eqref{sch:nonlin}.
The proof of this result is the subject of the next subsection.

\begin{theorem}[Existence of positive solutions] \label{th:existencenonlin}
Let $u^{in} \in L^2(\Omega)$ be a non-negative function such that $\int_\Omega u^{in} = M > 0$.
There exists at least one positive solution $\big(\u_\D^n\in\V_\D)_{n \geq 1}$ to the nonlinear scheme~\eqref{sch:nonlin}.
%% , satisfying the following discrete mass conservation and entropy dissipation properties:
%% $$\forall n \in \mathbb{N}, \qquad 
%% 	\int_\Omega u^n_\M  = M\quad 
%% 	\text{ and } \quad \frac{\N^{n+1} - \N ^n }{\Delta t} + \Diss^{n+1} \leq 0, $$
%% where $\N ^n$ and $\Diss^{n+1}$ are, respectively, some discrete relative entropy and dissipation (defined by~\eqref{entro1}).
Moreover, there exists $\varepsilon >0$, depending on $\Lambda$, $\phi$, $u^{in}$, $M$, $\Omega$, $d$, $\Delta t$, and $\D$ such that
\begin{equation} \label{eq:unifpos}
  \forall n \geq 1, \qquad u_K^n\geq\varepsilon\quad\forall K\in\M \quad\text{ and }\quad u_\s^n\geq\varepsilon \quad\forall\s\in\E.
\end{equation}
\end{theorem}
\noindent
The uniform-in-time positivity result~\eqref{eq:unifpos} on discrete solutions is the equivalent in the HFV context of~\cite[Lemma 3.7]{CaGui:17} and~\cite[Lemma 3.5]{CCHKr:18} obtained, respectively, in the VAG and DDFV contexts.

\subsubsection{Existence of discrete solutions} \label{ssse:exist}

The existence of discrete solutions to the nonlinear scheme~\eqref{sch:nonlin} is proved in two steps. First, we introduce a regularised scheme, for which we prove the existence of solutions by a fixed-point argument, {inspired from the proof of existence in~\cite{BKS:14}}. Then, we prove that sequences of regularised solutions satisfy uniform a priori bounds, which allows us to pass to the limit in the regularisation parameter.
{Notice that our proof of existence uses the same estimates, but follows a quite different path than the ones in the VAG~\cite{CaGui:17} and DDFV~\cite{CCHKr:18} contexts, in which the proof is based on the topological degree, together with a monotonicity argument.}
Henceforth, we reason in the $\w_\D$ variable, and we recall that $\u_\D=\u_\D^\infty\times{\rm exp}(\w_\D)$ according to~\eqref{sch3:w}. The advantage of doing so is that we can seek for solutions $\w_\D$ in the whole space $\V_\D$, with bijective correspondence with solutions $\u_\D$ that are automatically positive. Recalling the definition~\eqref{entro1} of the discrete relative entropy and dissipation, and using~\eqref{sch2} combined with~\eqref{eq:TK-AK}, we let, for all $\w_\D\in\V_\D$,
\begin{equation} \label{ED}
  \N(w_\M) \defi \sum_{K\in\M}|K|u^\infty_K\Phi_1(\e^{w_K}),\qquad\Diss(\w_\D) \defi \sum_{K \in \M} r_K \big(\u^\infty_K\times{\rm exp}(\w_K)\big)\, \delta_K \w_K \cdot\A_K \delta_K \w_K,
\end{equation}
in such a way that $\N^n=\N\big(w_\M(u_\M^n)\big)$ and $\Diss^n=\Diss\big(\w_\D(\u_\D^n)\big)$ for all $n\geq 1$. Using the fact that $\Phi_1(0)=1$, we extend the definition of $\N(w_\M)$ to the case where some $w_K$'s are equal to $-\infty$.

Before proceeding with the proof of Theorem~\ref{th:existencenonlin}, we state two preliminary lemmas. The first one, that can be found, e.g., in~\cite[Section 9.1]{Evans:10}, is a corollary of Brouwer's fixed-point theorem. This result is instrumental to show the existence of solutions to the regularised scheme.
\begin{lemma} \label{Lemme:exist}
  Let $N \in \mathbb{N}^\star$, and let $P : \R ^N \to \R^N $ be a continuous vector field. Assume that there is $r > 0$ such that 
  \begin{equation*}
    P(x) \cdot x \geq 0 \qquad\text{ if \;}  |x| = r. 
  \end{equation*}
  Then, there exists a point $x_0 \in \R^N$ such that $P(x_0) =0$ and $|x_0| \leq r$.
\end{lemma}
\noindent
The second lemma, whose proof is postponed until Appendix~\ref{ap:lemma}, establishes sufficient boundedness conditions on the discrete mass and (relative) dissipation so that a priori bounds hold for vectors of discrete unknowns. This result is instrumental to show that sequences of regularised solutions satisfy (uniform) a priori bounds.
\begin{restatable}{lemma}{lem} \label{lemma:positivity}
	Let $\w_\D \in \V_\D$, and assume that there exist {$C_\sharp > 0$, and $M_\sharp \geq M_\flat >0$} such that
	\begin{equation} \label{Entropycontrol}
		M_\flat\leq\sum_{K\in\M}|K|u_K^\infty \e^{w_K} \leq M_\sharp \qquad\text{ and } \qquad\Diss(\w_\D) \leq C_\sharp.
	\end{equation}
 	Then, there exists $C > 0$, depending on $\Lambda$, $u^\infty_\flat$, $u^\infty_\sharp$, $M_\flat$, $M_\sharp$, $C_\sharp$, $\Omega$, $d$, and $\D$ such that 
	$$|w_K|\leq C\quad\forall K\in\M \quad\text{ and }\quad |w_\s|\leq C \quad\forall\s\in\E.$$
\end{restatable}

We can now proceed with the proof of Theorem~\ref{th:existencenonlin}. Let us first define the following inner product and corresponding norm on the space $\V_\D$: for all $\underline{z}_\D,\v_\D\in\V_\D$,
$$\langle \underline{z}_\D, \v_\D \rangle \defi \sum_{K \in \M} z_K v_K + \sum_{\s \in \E} z_\s v_\s\quad\text{ and }\quad\|\v_\D\| \defi\sqrt{\langle \v_\D,\v_\D \rangle}.$$
{Letting $N\defi|\M|+|\E|$, and identifying $\V_\D$ to $\R^N$, the inner product $\langle\cdot,\cdot\rangle$ is nothing but the standard inner product on $\R^N$.}
For all $K\in\M$, and all $\s\in\E_K$, we let, for $\u_K\in\V_K$ positive,
\begin{equation} \label{def:G:flux}
  \F_{K,\s}^{{\rm nl}}(\u_K) \defi r_K(\u_K) \sum_{\s' \in \E_K }  A_K^{\s\s'} \left ( \log \left ( \frac{u_K}{u_K^\infty} \right ) - \log \left ( \frac{u_{\s'}}{u_{\s'}^\infty} \right )   \right ),
\end{equation}
where the $A_K^{\s\s'}$ are defined by~\eqref{def:flux:A}. Combining~\eqref{eq:TK} and \eqref{sch3:w} with~\eqref{aKL-flux} and \eqref{diff.flux}, there holds that $T_K(\u_K,\w_K,\v_K)=\sum_{\s\in\E_K}\F_{K,\s}^{{\rm nl}}(\u_K)(v_K-v_\s)$ for all $\v_K\in\V_K$. In what follows, we let $n\in\mathbb{N}^\star$ and $u_\M^{n-1}\geq 0$ be given. We assume that $M^{n-1}\defi\int_{\Omega}u_\M^{n-1}>0$ and that $u_\M^{n-1}>0$ if $n>1$. We also assume that $\N(w^{n-1}_\M)>0$. If $\N(w^{n-1}_\M)=0$ (which is equivalent to $u_\M^{n-1}=u_\M^\infty$), then necessarily, by~\eqref{diss:nonlin}, $\u_\D^n=\u_\D^\infty$ uniquely solves~\eqref{sch1}. Letting, for any $\u_\D\in\V_\D$ positive, $\Gu^{n}_\D(\u_\D)$ be the element of $\V_\D$ such that
\begin{subequations}\label{def:G}
	\begin{align}
	 & \G_K^{n}(\u_\D)  \defi  |K| \frac{u_K - u_K^{n-1}}{ \Delta t } +  \sum_{\s \in \E_K}\F_{K,\s}^{{\rm nl}}(\u_K) \quad&&\forall K\in\M,  \label{def:G:mesh} \\
         & \G_\s^{n}(\u_\D) \defi - \big( \F_{K,\s}^{{\rm nl}}(\u_K) + \F_{L,\s}^{{\rm nl}}(\u_L) \big) \quad&&\forall\s=K\mid L\in\E_{int},  \label{def:G:edgeint}\\  
         & \G_\s^{n}(\u_\D) \defi - \F_{K,\s}^{{\rm nl}}(\u_K) \quad&&\forall\s\in\E_{ext}\text{ with }\M_{\s}=\{K\}, \label{def:G:edgeext}
	\end{align}
\end{subequations}
we infer that, for all $\v_\D\in\V_\D$,
\begin{equation} \label{Def:G:eq}
  \frac{1}{\Delta t } \left( u_\M - u^{n-1}_\M, v_\M \right)_\Omega +  T_\D (\u_\D, \w_\D(\u_\D),\v_\D) = \left \langle \Gu_\D^{n} (\u_\D) , \v_\D \right \rangle. 
\end{equation}
Hence, a positive vector $\u_\D^{n} \in \V_\D$ is a solution to the nonlinear equation~\eqref{sch1} if and only if $\Gu_\D^{n} (\u_\D^{n}) = \underline{0}_\D$. With this observation in hand, we now detail the two steps of the proof. \\

\noindent
{{\bf Step 1:}} Using the relation $\u_\D=\u_\D^\infty\times{\rm exp}(\w_\D)$, we define the vector field {$\underline{{\cal P}}^{n,\mu}_\D:\V_\D\to\V_\D$} such that, for all $\w_\D\in\V_\D$, 
\begin{equation}\label{def:P_esp}
	\underline{{\cal P}}^{n,\mu}_\D(\w_\D) \defi \Gu_\D^{n} \big(\u_\D^\infty \times {\rm exp}(\w_\D)\big)+\mu\,\w_\D,
\end{equation}
with $\Gu_\D^n$ defined by~\eqref{def:G} and $\mu\geq 0$.
Notice that, unlike $\Gu_\D^{n}$, the vector field $\underline{{\cal P}}^{n,\mu}_\D$ is continuous on the whole space $\V_\D$ for any $\mu\geq 0$.
If $\w_\D^{n} \in \V_\D$ satisfies $\underline{{\cal P}}_\D^{n,0} (\w_\D^{n}) =\underline{0}_\D$, then letting $\u_\D^n\defi\u_\D^\infty\times{\rm exp}(\w_\D^n)$, we have $\Gu_\D^{n} ( \u_\D^n ) = \underline{0}_\D$, therefore $\u_\D^{n}$ is a (positive) solution to~\eqref{sch1}.
For $\mu>0$, the problem of finding $\w_\D^{n,\mu}\in\V_\D$ such that $\underline{{\cal P}}_\D^{n,\mu}(\w_\D^{n,\mu})=\underline{0}_\D$ can thus be seen as a regularisation of the original problem.
By~\eqref{def:P_esp} and~\eqref{Def:G:eq}, for all $\w_\D\in\V_\D$, we have
\begin{equation} \label{Pmu}
  \begin{split}
    \left\langle  \underline{{\cal P}}_\D^{n,\mu}(\w_\D) , \w_\D \right\rangle = \sum_{K \in \M} \frac{|K|}{\Delta t} &\left (u_K^\infty \e^{w_K} - u_K ^{n-1} \right ) w_K \\&+\sum_{K \in \M } r_K\big( \u^\infty_K\times{\rm exp}(\w_K) \big) \delta_K \w_K \cdot \A_K \delta_K \w_K+ \mu \|\w_\D\|^2.
  \end{split}
\end{equation}
By~\eqref{ED}, we recognise in the second term of the right-hand side the quantity $\Diss(\w_\D)\geq 0$. As far as the first term is concerned, for $n>1$, by positivity of the $(u_K ^{n-1})_{K \in \M}$, there exist real numbers $(w_K^{n-1})_{K\in\M}$ such that $u_K ^{n-1} =  u_K^\infty \e^{w^{n-1}_K}$ for all $K\in\M$, and since $\Phi_1$ is convex,
\begin{equation} \label{Nmu}
  \begin{split}
    \sum_{K \in \M} \frac{|K|}{\Delta t} \left (u_K^\infty \e^{w_K} - u_K ^{n-1} \right ) w_K & = \sum_{K \in \M} \frac{|K|}{\Delta t}u_K^\infty w_K\big( \e^{w_K} -\e^{w^{n-1}_K}\big) \\
    & \geq \sum_{K \in \M} \frac{|K|}{\Delta t} u_K^\infty\left ( \Phi_1\big(\e^{w_K}\big) -\Phi_1\big(\e^{w^{n-1}_K}\big)\right )
    =  \frac{\N(w_\M)-\N(w^{n-1}_\M)}{\Delta t},
  \end{split}
\end{equation}
where we have used the definition~\eqref{ED} of $\N(w_\M)$. For $n=1$, now, it may happen that $u_K^0$ be zero for some $K\in\M$, and then $w^0_K$ such that $u_K^0=u_K^\infty \e^{w_K^0}$ cannot be defined. However, letting in that case $w_K^0\defi-\infty$, the inequality above still holds since $\Phi_1(0)=1$ and $\Phi_1(\e^s)-1\leq s\e^s$ for all $s\in\R$.
By non-negativity of $\N(w_\M)$, we finally infer from~\eqref{Pmu} and~\eqref{Nmu} that
$$\left\langle  \underline{{\cal P}}_\D^{n,\mu}(\w_\D) , \w_\D \right\rangle \geq \mu \|\w_\D\|^2 - \frac{\N(w_\M^{n-1})}{\Delta t},$$
so that, for $\mu>0$, there holds $\left \langle \underline{{\cal P}}_\D^{n,\mu}(\w_\D)  , \w_\D \right \rangle \geq 0$ if $\|\w_\D\| = \sqrt{ \frac{\N(w_\M^{n-1})}{\mu \Delta t}}>0$.
By Lemma~\ref{Lemme:exist}, we then conclude about the existence of solutions to the regularised scheme. There exists $\w_\D^{n,\mu} \in \V_\D$ such that 
\begin{equation}\label{def:w_eps}
  \underline{{\cal P}}_\D^{n,\mu}(\w^{n,\mu}_\D)=\underline{0}_\D \quad\text{ and }\quad \|\w_\D^{n,\mu} \| \leq \sqrt{ \frac{\N(w_\M^{n-1})}{\mu \Delta t}}.
\end{equation}

\noindent
{{\bf Step 2:}}
Since $\left \langle \underline{{\cal P}}_\D^{n,\mu}(\w^{n,\mu}_\D) , \w_\D^{n,\mu} \right \rangle = 0$, by~\eqref{Pmu} and~\eqref{Nmu}, we have
$$\frac{\N(w_\M^{n,\mu})}{\Delta t} + \Diss(\w_\D^{n,\mu}) + \mu \| \w_\D^{n,\mu} \|^2 \leq \frac{\N(w_\M^{n-1})}{\Delta t}.$$
The three terms on the left-hand side being non-negative, we infer that
\begin{equation}\label{bound:entro}
  \Diss(\w_\D^{n,\mu}) \leq C_\sharp,
\end{equation}
with $C_\sharp\defi\frac{\N(w_\M^{n-1})}{\Delta t}>0$.
Moreover, since $\left \langle  \underline{{\cal P}}_\D^{n,\mu}(\w^{n,\mu}_\D) , \one_\D \right \rangle = 0$, by~\eqref{def:P_esp} and~\eqref{Def:G:eq}, we have
$$\sum_{K \in \M} |K| u_K ^\infty \e^{w_K^{n,\mu}} -M^{n-1}= - \mu \Delta t\left \langle \w^{n,\mu}_\D , \one_\D \right \rangle.$$
Applying a Cauchy--Schwarz inequality, and recalling the bound~\eqref{def:w_eps}, we obtain 
$$\left |\sum_{K \in \M} |K| u_K ^\infty \e^{w_K^{n,\mu}} - M^{n-1} \right | \leq    \mu\Delta t \|\one_\D\| \|\w^{n,\mu}_\D\|\leq\sqrt{\mu}\sqrt{N\Delta t\,\N(w_\M^{n-1})},$$
so that, letting $\mu_0 \defi \frac{(M^{n-1})^2}{4 N\Delta t\,\N(w_\M^{n-1})}>0$, the following holds for all { $0 < \mu\leq\mu_0$}:
\begin{equation} \label{bound:mass}
  0<\frac{M^{n-1}}{2}\ifed M_\flat\leq \sum_{K \in \M} |K| u_K ^\infty \e^{w_K^{n,\mu}} \leq  M_\sharp\defi\frac{3 M^{n-1}}{2}.
\end{equation}
%% Remark that if $\N(w_\M^{n-1})=0$, then $\sum_{K \in \M} |K| u_K ^\infty \e^{w_K^{n,\mu}}=M^{n-1}$ and the inequality~\eqref{bound:mass} remains valid for all $\mu\geq 0$.
By~\eqref{bound:mass} and \eqref{bound:entro}, we infer that $\w^{n,\mu}_\D$ satisfies~\eqref{Entropycontrol} for $\mu$ sufficiently small with constants that are uniform in $\mu$, so that by Lemma~\ref{lemma:positivity} the family $(\w_\D^{n,\mu})_{0 <\mu\leq\mu_0}$ is bounded uniformly in $\mu$. As a consequence, by compactness, there is $\w_\D^n\in\V_\D$ such that, up to extraction, $\w_\D^{n,\mu}$ converges towards $\w_\D^n$ when $\mu$ tends to zero. Since $\underline{{\cal P}}_\D^{n,\mu}$ converges to $\underline{{\cal P}}_\D^{n,0}$ as $\mu$ tends to zero, we finally infer that $\underline{{\cal P}}_\D^{n,0}(\w_\D^n) =\underline{0}_\D$ (also, $\sum_{K \in \M} |K| u_K ^\infty \e^{w_K^{n}}=M^{n-1}$).\\

\noindent
{{\bf Conclusion:}} Letting $\u_\D^n\defi\u_\D^\infty\times{\rm exp}(\w_\D^n)$, we have $\Gu_\D^{n} ( \u_\D^n ) = \underline{0}_\D$, therefore $\u_\D^{n}$ is a (positive) solution to~\eqref{sch1}. By Propositions~\ref{prop:mass} and~\ref{prop:dissip}, and since $\Diss\big(\w_\D^n\big)=\Diss^n$, $\N^n$ is non-negative, and $\N^n$ is non-increasing in $n$ according to~\eqref{diss:nonlin}, we deduce that
$$\sum_{K \in \M} |K| u_K ^\infty \e^{w_K^n} = M\quad\text{ and }\quad\Diss(\w_\D^n) \leq \frac{\N^0}{\Delta t}.$$
By Lemma~\ref{lemma:positivity} (recall also that $u_\flat^\infty,u^\infty_\sharp$ only depend on $M$, $\phi$, and $\Omega$), there exists $C>0$, depending on $\Lambda$, $\phi$, $u^{in}$, $M$, $\Omega$, $d$, $\Delta t$, and $\D$, but not on $n\in\mathbb{N}^\star$, such that $-C \one_\D \leq \w_\D^n \leq C  \one_\D$. By~\eqref{bornesstationnaires}, we finally infer that $\u_\D^n\geq\varepsilon\one_\D$, with $\varepsilon\defi u_\flat^\infty \e^{-C}>0$ still independent of $n\in\mathbb{N}^\star$. This concludes the proof of Theorem~\ref{th:existencenonlin}.

\section{Long-time behaviour} \label{Time}

In this section, we analyse the long-time behaviour of the three HFV schemes we have introduced in Section~\ref{description}, thereby proving the main results of this paper.
{
\begin{rem}[Linear schemes and nonhomogeneous data]
  In order to stay consistent with Section~\ref{description}, we here below state our asymptotic results of Theorems~\ref{Th:hmm} and~\ref{Th:omega}, which respectively concern the (linear) standard and exponential fitting schemes, for discrete problems that feature homogeneous data (i.e., $g^D=0$ when $|\Gamma^D|>0$, or $M=0$ otherwise). Nonetheless, Theorems~\ref{Th:hmm} and~\ref{Th:omega} remain valid in the general case of nonhomogeneous data (we refer to Remark~\ref{rem:ls} for the straightforward adaptation of the schemes to this situation). Indeed, the proofs of the latter results only hinge on the fact that the difference between the discrete transient and steady-state solutions belongs to the homogeneous space $\V_{\D,0}$ or $\V^\omega_{\D,0}$, which is always true. This remark does not apply, however, to the nonlinear scheme.
\end{rem}}

\subsection{Asymptotic behaviour of the standard HFV scheme}

%% We are now interested in the long-time behaviour of the the HMM scheme \eqref{sch:HMM}. Our strategy relies on mimicking the arguments used in \cite{BLMVi:14} in the continuous framework. In fact, our approach is very similar to the one used in Theorem 3.2 of \cite{CHHer:20} for the analysis of the approximation of porous medium equation solutions. Especially, we use the so-called ``$\psi$-entropy''  described in \cite{BLMVi:14}, and functional inequality of the Poincaré type.

We recall that $u$ is the solution to Problem~\eqref{evol:mixed:ino}, and that $u^\infty$ is the corresponding steady-state, solution to Problem~\eqref{sta:mixed:ino}, and we consider the following definition of the relative entropy and dissipation: 
$$\N (t) \defi\frac{1}{2}  \left\| u(t) - u^\infty \right\|_{L^2(\Omega)} ^2,\qquad\Diss(t) \defi \int_\Omega \big(\Lambda \nabla (u(t)-u^\infty)-(u(t)-u^\infty)V^\phi\big)\cdot \nabla (u(t) - u^\infty).$$
It can be easily verified that the following entropy/dissipation relation holds at the continuous level:
$$\frac{{\rm d}}{{\rm d}t} \N (t) + \Diss(t) = 0.$$
It is assumed that $V^\phi$ is such that $\Diss(t)\geq C\|\nabla(u(t)-u^\infty)\|_{L^2(\Omega;\R^d)}^2$ for some $C>0$, so that $\Diss(t)$ indeed defines a dissipation.

At the discrete level, recalling that $\big(\u_\D^n\in\V_{\D,0}\big)_{n\geq 1}$ is the solution to Problem~\eqref{sch:HMM}, and that $\u_\D^\infty\in\V_{\D,0}$ is the corresponding steady-state, solution to Problem~\eqref{sccomp3}, we consider the following equivalents of the relative entropy and dissipation: for all $n\in\mathbb{N}^\star$,
$$\N ^n  \defi \frac{1}{2}  \left\| u^n_\M - u^\infty_\M \right\|_{L^2(\Omega)} ^2,\qquad\Diss^n  \defi a_\D\big(\u_\D^n - \u_\D^\infty , \u_\D^n - \u_\D^\infty\big),$$
where the discrete bilinear form $a_\D$ is defined {by~\eqref{eq:aK} and \eqref{eq:aD}}. The definition of the relative entropy is seamlessly extended to the case $n=0$. Our main result on the standard HFV scheme is the following.

\begin{theorem}[Asymptotic stability] \label{Th:hmm}
  Assume that the advection field $V^\phi$ satisfies the conditions~\eqref{cond:coerc} of Proposition~\ref{prop:coerc}, with constant $\beta$.
  Then, the following discrete entropy/dissipation relation holds true:
  \begin{equation} \label{diss}
    \forall n \in \mathbb{N}, \qquad \frac{\N ^{n+1} - \N^n }{\Delta t } + \Diss ^{n+1} \leq 0.
  \end{equation}
  Furthermore, the discrete entropy decays exponentially fast in time: there is $\nu \defi \frac{2  \kappa }{ C_P^2}>0 $, where $\kappa$ is the constant of~\eqref{def:kappa} (only depending on $\Lambda$, $\beta$, $\Omega$, $d$, $\Gamma^D$, and $\theta_\D$), and $C_P$ is either equal to $C_{PW}$ if $|\Gamma^D|=0$ or to $C_{P,\Gamma^D}$ otherwise (where $C_{PW}$, $C_{P,\Gamma^D}$ are the Poincar\'e constants of Proposition~\ref{prop:poinca}), such that
  \begin{equation} \label{decexp}
    \forall n \in \mathbb{N},\qquad \N^{n+1} \leq \left (    1 + \nu \Delta t  \right )^{-1} \N ^n .
  \end{equation}
  Consequently, the discrete solution converges exponentially fast in time towards its associated discrete steady-state: for all $n\in\mathbb{N}^\star$,
  \begin{equation} \label{decexpsol1}
    \left\| u^n_\M - u^\infty_\M \right\|_{L^2(\Omega)}
    \leq   ( 1 + \nu \Delta t  )^{-\frac{n}{2}} \left\| u^0_\M - u^\infty_\M \right\|_{L^2(\Omega)}.
  \end{equation}
\end{theorem}
%% The proof of these results mimics the proof of the continuous result. It is separated in three steps: establish the dissipation relation, get a control of the entropy in term of its dissipation and show exponential decay of entropy.
\begin{proof}
  {Let $n\in\mathbb{N}$. One has 
  $$\N ^{n+1} - \N ^n = \sum _{K \in \M}  \frac{|K|}{2} \left (
   \left ( u_K ^{n+1} - u^\infty_K \right )^2	   - \left ( u_K ^{n} - u_K^\infty \right )^2  \right ).$$}
  Since $x \mapsto x^2$ is convex, for all $x,y\in\R$, we have $y^2 - x^2 \leq 2 y (y-x)$, therefore
  \begin{equation} \label{ineqentro}
    \N ^{n+1} - \N ^n \leq \sum _{K \in \M}  |K| \big( u_K ^{n+1} - u_K ^{n} \big) \big( u_K ^{n+1} - u_K ^\infty \big)= \big(u^{n+1}_\M - u^n_\M, u^{n+1}_\M - u^\infty_\M\big)_\Omega.
  \end{equation}
  Now, testing~\eqref{sch1:HMM} with $\v_\D^{n+1} \defi \u_\D^{n+1} - \u_\D^\infty \in \V_{\D,0}$ yields
  $$\big(u^{n+1}_\M - u^n_\M, v^{n+1}_\M\big)_\Omega = -\Delta t a_\D \big(\u_\D^{n+1}, \v_\D^{n+1}\big) +\Delta t\big( (f, v^{n+1}_\M)_\Omega + (g^N, v^{n+1}_\E)_{\Gamma^N}\big).$$
  By definition~\eqref{sccomp3} of the discrete steady-state  $\u_\D^\infty$, we also have
  $$(f, v^{n+1}_\M)_\Omega + (g^N, v^{n+1}_\E)_{\Gamma^N}= a_\D \big(\u_\D^\infty, \v_\D^{n+1}\big),$$
  whence, by bilinearity of $a_\D$, we infer
  $$\big(u^{n+1}_\M - u^n_\M,v^{n+1}_\M \big)_\Omega = -\Delta t \,a_\D \big(\u_\D^{n+1} - \u_\D^\infty, \v_\D^{n+1}\big) = - \Delta t \,\Diss^{n+1}.$$
  Combined to~\eqref{ineqentro}, this proves the entropy/dissipation relation~\eqref{diss}.
  Now, since the advection field $V^\phi$ satisfies~\eqref{cond:coerc}, we can invoke~\eqref{def:kappa} from Proposition~\ref{prop:coerc} to infer that  
  $$\Diss^{n+1} \geq  \kappa |\u_\D^{n+1} - \u_\D^\infty|^2 _{1, \D},$$
  where $\kappa>0$ only depends on $\Lambda$ ,$\beta$, $\Omega$, $d$, $\Gamma^D$, and $\theta_\D$.
  Combining this estimate with a discrete Poincaré inequality from Proposition~\ref{prop:poinca} (applied to $\u_\D^{n+1}-\u_\D^{\infty}\in\V_{\D,0}$), and with the definition of the discrete (relative) entropy, yields
  $$\Diss^{n+1} \geq \frac{\kappa }{C_P^2}  \left\| u^{n+1}_\M - u^\infty_\M \right\|^2_{L^2(\Omega)}=\frac{2\kappa }{C_P^2}\N^{n+1},$$
  where $C_P$ is either equal to $C_{PW}$ if $|\Gamma^D|=0$ or to $C_{P,\Gamma^D}$ otherwise.
  This last inequality, combined with the entropy/dissipation relation~\eqref{diss}, implies the entropy decay~\eqref{decexp}.
  The inequality~\eqref{decexpsol1} is then a straightforward consequence of the definition of $\N^n$.
\end{proof}
\noindent
The result of Theorem~\ref{Th:hmm} does not use the fact that $V^\phi$ is related to the gradient of a potential, it thus extends to general advection fields.

\subsection{Asymptotic behaviour of the exponential fitting scheme}

We recall that $\rho=\frac{u}{\omega}$, with $\omega=\e^{-\phi}$, is the solution to Problem~\eqref{def:evol:modeq}, and that $\rho^\infty$ is the corresponding steady-state, solution to Problem~\eqref{def:sta:modeq}. We consider the following $\omega$-weighted definitions of the relative entropy and dissipation: 
$$\N_\omega (t) \defi\frac{1}{2}  \left\| \rho(t) - \rho^\infty \right\|_{L^2_\omega(\Omega)} ^2,\qquad\Diss_\omega(t) \defi \int_\Omega \omega\Lambda\nabla (\rho(t)-\rho^\infty)  \cdot  \nabla (\rho(t) - \rho^\infty).$$
It can be easily verified that the following entropy/dissipation relation holds at the continuous level:
$$\frac{{\rm d}}{{\rm d}t} \N_\omega (t) + \Diss_\omega(t) = 0.$$

At the discrete level, let us recall that $\big(\ro_\D^n\in\V^\omega_{\D,0}\big)_{n\geq 1}$ is the solution to Problem~\eqref{sch:omega:evol}. We then set $\u_\D^n=\underline{\omega}_\D\times\ro_\D^n$ with $\underline{\omega}_\D$ defined by~\eqref{def-omega}, in such a way that $\u_\D^n\in\V_{\D,0}$. Similarly, $\ro_\D^\infty\in\V^\omega_{\D,0}$ is the corresponding steady-state, solution to Problem~\eqref{def:omega:sta}, and we set $\u_\D^\infty=\underline{\omega}_\D\times\ro_\D^\infty\in\V_{\D,0}$. We consider the following equivalents of the $\omega$-weighted (relative) entropy and dissipation: for all $n\in\mathbb{N}^\star$,
$$\N ^n_\omega  \defi \frac{1}{2}  \left\| \rho^n_\M - \rho^\infty_\M \right\|_{L^2_\omega(\Omega)} ^2,\qquad\Diss_\omega^n  \defi a^\omega_\D\big(\ro_\D^n - \ro_\D^\infty , \ro_\D^n - \ro_\D^\infty\big),$$
where the discrete bilinear form $a^\omega_\D$ is defined (locally) by~\eqref{def:omega:bilin}. The definition of the relative entropy is seamlessly extended to the case $n=0$. Our main result on the exponential fitting HFV scheme is the following, whose proof is very similar to the one of Theorem~\ref{Th:hmm}.

\begin{theorem}[Asymptotic stability] \label{Th:omega}
  The following discrete entropy/dissipation relation holds true:
  \begin{equation} \label{dissip:omega}
    \forall n \in \mathbb{N}, \qquad \frac{\N_\omega ^{n+1} - \N_\omega^n }{\Delta t } + \Diss_\omega ^{n+1} \leq 0.
  \end{equation}
  Furthermore, the discrete entropy decays exponentially fast in time: there is $\nu_\omega \defi\frac{2 \omega_\flat\lambda_\flat \alpha_\flat }{\omega_\sharp C_P^2} >0 $, where $\lambda_\flat\alpha_\flat$ is the coercivity constant of~\eqref{def:globalcoercivity} (only depending on $\Lambda$, $\Omega$, $d$, and $\theta_\D$), (we recall that) the bounds $\omega_\flat,\omega_\sharp$ only depend on $\phi$ and $\Omega$, and $C_P$ is either equal to $2C_{PW}$ if $|\Gamma^D|=0$ or to $C_{P,\Gamma^D}$ otherwise (where $C_{PW}$, $C_{P,\Gamma^D}$ are the Poincar\'e constants of Proposition~\ref{prop:poinca}), such that
  \begin{equation} \label{decay_omega}
    \forall n \in \mathbb{N},\qquad \N^{n+1}_\omega \leq \left (    1 + \nu_\omega \Delta t  \right )^{-1} \N_\omega ^n .
  \end{equation}
  Consequently, the discrete solution converges exponentially fast in time towards its associated discrete steady-state: for all $n\in\mathbb{N}^\star$,
  \begin{equation} \label{decexpsolome}
    \|u_\M^n - u_\M^\infty \|_{L^2(\Omega)} \leq \sqrt{\frac{\omega_\sharp}{\omega_\flat}} \left ( 1 + \nu_\omega \Delta t\right )^{-\frac{n}{2}} \|u_\M^0 - u_\M^\infty \|_{L^2(\Omega)}.
  \end{equation}
\end{theorem}
%% The proof relies on the following arguments: as for the previous linear scheme, it should be possible to prove a good long-time behaviour for the solution $(\ro^n)_{n \geq 1}$ of the scheme \eqref{sch1:omega} - \eqref{sch2:omega}. Then, we can use bounds on $\omega$ to pass the estimation on $\ro^n$ to $\u^n$.
\begin{proof}
 Let $n\in\mathbb{N}$. {Reasoning as in the proof of Theorem~\ref{Th:hmm}, we infer
\[ \N^{n+1}_\omega-\N^n_\omega \leq
	 -\Delta t \,a_\D^\omega \big( \ro_\D^{n+1} -\ro_\D^\infty , \ro_\D^{n+1} -\ro_\D^\infty \big)
	  = - \Delta t \,\Diss^{n+1}_\omega, 
\]
which proves~\eqref{dissip:omega}. By the coercivity estimate~\eqref{ef-coer}, we have that
  $\displaystyle \Diss^{n+1}_\omega \geq \omega_\flat \lambda_\flat\alpha_\flat \big|\ro_\D^{n+1}-\ro_\D^\infty\big|_{1,\D}^2,$
  where $\alpha_\flat>0$ from Section~\ref{sse:wp} only depends on $\Omega$, $d$, and $\theta_\D$, and $\omega_\flat>0$ only depends on $\phi$ and $\Omega$.}
  Reasoning as in the proof of Proposition~\ref{pro:wpef}, and using a discrete Poincar\'e inequality from Proposition~\ref{prop:poinca} (combined with~\cite[Lemma 5.2]{CCHHK:20} in the case $|\Gamma^D|=0$), we also infer that
  $$\|\rho^{n+1}_\M-\rho^\infty_\M\|_{L^2_\omega(\Omega)}\leq \sqrt{\omega_\sharp}\,C_{P}|\ro_\D^{n+1}-\ro^\infty_\D|_{1,\D},$$
  where $\omega_\sharp>0$ only depends on $\phi$ and $\Omega$, and $C_P$ is either equal to $2C_{PW}$ if $|\Gamma^D|=0$ or to $C_{P,\Gamma^D}$ otherwise. Thus, we finally get that
  $$\Diss^{n+1}_\omega \geq \frac{2 \omega_\flat\lambda_\flat \alpha_\flat }{\omega_\sharp C_{P}^2 } \N^{n+1}_\omega.$$
  Combined to~\eqref{dissip:omega}, this yields~\eqref{decay_omega}. Deriving~\eqref{decexpsolome} is then straightforward.
\end{proof}

\subsection{Asymptotic behaviour of the nonlinear scheme} \label{asnon}

Recall that $u>0$ is the solution to Problem~\eqref{evol:mixed:ino} endowed with pure Neumann boundary conditions ($|\Gamma^D|=0$), and data $f=0$, $g^N=0$, and $u^{in}\geq 0$ with $M=\int_{\Omega}u^{in}>0$ {(see Appendix~\ref{Ap:NLmixted} for the case of mixed Dirichlet-Neumann boundary conditions)}, and that $u^\infty>0$, solution to Problem~\eqref{sta:mixed:ino} with same data, is the thermal equilibrium $u^\infty_{th}$ given by~\eqref{eq:therm} with ${\bar \rho}=\rho^M$.
The analysis of the nonlinear scheme relies on the entropy/dissipation structure~\eqref{nonlin:relation}-\eqref{nonlin:entro} introduced in Section~\ref{sse:nonlin}. Notice that the relative dissipation {(or relative Fisher information)} of~\eqref{nonlin:entro} can be equivalently rewritten
$$\Diss(t)=\int_{\Omega } u(t) \Lambda\nabla \log \left ( \frac{u(t)}{u^\infty} \right ) \cdot  \nabla \log \left ( \frac{u(t)}{u^\infty} \right )= 4 \int_\Omega u^\infty  \Lambda\nabla \sqrt {\frac{u(t)}{u^\infty} } \cdot  \nabla \sqrt {\frac{u(t)}{u^\infty} }.$$

%% As for the analysis of the continuous equation established in \cite{BLMVi:14} and the discrete case of the TPFA and DDFV schemes of \cite{CCHHK:20}, the long-time behaviour relies on the study of the right-hand side of this equality.
At the discrete level, recalling that $\big(\u_\D^n\in\V_\D\big)_{n\geq 1}$ is a (positive) solution to Problem~\eqref{sch:nonlin}, and that $\u_\D^\infty\in\V_\D$ is the corresponding steady-state, solution to Problem~\eqref{def:nonlin:sta}, that is equal to the HFV interpolate of $u^\infty_{th}$, we consider the discrete entropy $\N^n$ and dissipation $\Diss^n$ defined by~\eqref{entro1}, and we define a discrete counterpart of the relative dissipation written in root-form: for all $n\geq 1$,
\begin{equation} \label{hatD}
		\hat{\Diss} ^n \defi 
		4 \sum_{K \in \M } u_{K,\flat}^\infty \int_K  \Lambda\nabla_K \rac_K^n \cdot  \nabla_K \rac_K^n
		= 4 \sum_{K \in \M } u^\infty_{K,\flat} \delta_K  \rac_K^n \cdot \A _K  \delta_K \rac_K^n , 
\end{equation}
where, for all $K\in\M$, {$u^\infty_{K,\flat}  \defi \min \left ( u_K^\infty,  \underset{\s \in \E_K }{\min} u_\s^\infty \right )$} and the matrix $\A_K$ is defined by~\eqref{def:flux:A}, and $\rac_\D^n$ is the element of $\V_\D$ such that
\begin{equation} \label{xi}
  \xi_K^n  \defi \sqrt {\frac{u^n_K}{u^\infty_K}}\quad\forall K \in \M,\qquad \xi_\s^n  \defi\sqrt {\frac{u^n_\s}{u^\infty_\s}}\quad\forall \s \in \E.
\end{equation}
At the discrete level, and as opposed to the continuous level, the quantities $\Diss^n$ and $\hat{\Diss} ^n$ are not equal, therefore we need to compare them. The definition of $ u^\infty_{K,\flat}$ results from the following observation: according to the structures of $\Diss^n$ and $\hat{\Diss} ^n$, locally, we expect to have to compare $u^\infty_{K,\flat}$ with $r_K(\u_K^n)$, which depends on $u_K^n$ and on the $(u_\s^n)_{\s \in \E_K}$.

\begin{prop}[Fisher information]\label{prop:Fisher}
  %% The discrete entropy dissipation controls the discrete Fisher information:
  There is $C_F>0$, only depending on $\Lambda$, $\Omega$, $d$, and $\theta_\D$ such that 
  \begin{equation} \label{Fisher:dissipation}
    \forall n \geq 1, \qquad \hat{\Diss} ^n \leq C_F \Diss^n.
  \end{equation}
\end{prop}
\begin{proof}
  Let $n\in\mathbb{N}^\star$, and $K\in\M$. By~\eqref{CB} from Lemma~\ref{lemma:cond}, we first have that
  \begin{equation}\label{minD}
    %% \Diss ^n \geq \frac{1}{C_B} 
    %% \sum_{K \in \M } r^K (\u^n ) \delta_K  \logg ^n  \cdot \B _K  \delta_K  \logg ^n  
    %% \text{ and }
    4 u^\infty_{K,\flat} \,\delta_K  \rac_K^n \cdot \A _K  \delta_K \rac_K^n
    \leq 4 u^\infty_{K,\flat}\, \delta_K  \rac_K ^n  \cdot \B _K  \delta_K \rac_K ^n , 
  \end{equation}
  where the matrix $\B_K$ is the diagonal matrix defined by~\eqref{BK}.
  Since for all $(x,y) \in \big(\R_+^\star\big)^2$, $\int_x ^y \frac{{\rm d}z}{\sqrt{z}} = 2  \left ( \sqrt{y} - \sqrt{x} \right )$, the Cauchy--Schwarz inequality yields
  \begin{equation*}
    4 \left ( \sqrt{y} - \sqrt{x} \right ) ^2 \leq (y-x) \big(\log(y) - \log(x)\big).
  \end{equation*}
  By the property~\eqref{hyp:g:bounds} of the function $m$, we then get that 
  \begin{equation} \label{sym1}
    \forall (x,y) \in \big(\R^\star_+\big)^2,\qquad 
    4 \left ( \sqrt y - \sqrt x\right )^2  \leq m(x,y)  \big( \log(y) - \log (x) \big) ^2.
  \end{equation}
  Since $\B_K$ is diagonal, the combination of~\eqref{minD},~\eqref{xi}, and~\eqref{sym1}, yields 
  \[
  \begin{split}
    4 u^\infty_{K,\flat}\, \delta_K  \rac_K^n \cdot \A _K  \delta_K \rac_K^n
    & \leq  \sum_{\s \in \E_K} 4 u^\infty_{K,\flat} B_K ^{\s\s} \left (\sqrt {\frac{u^n_K}{u^\infty_K} } -  \sqrt {\frac{u^n_\s}{u^\infty_\s} } \right )  ^2\\
    & \leq \sum_{\s \in \E_K} u^\infty_{K,\flat} B_K ^{\s\s} m \left ( \frac{u^n_K}{u^\infty_K} , \frac{u^n_\s}{u^\infty_\s}  \right ) \left (\log \left ( {\frac{u^n_K}{u^\infty_K} }  \right )-  \log \left (  {\frac{u^n_\s}{u^\infty_\s} } \right )  \right )  ^2 .
  \end{split}
  \]
  By definition of $u^\infty_{K,\flat}$, and monotonicity \eqref{hyp:g:monotony} and homogeneity \eqref{hyp:g:homogeneity} of $m$, we infer that, for all $\s\in\E_K$,
  \[
  u^\infty_{K,\flat} m \left ( \frac{u^n_K}{u^\infty_K} , \frac{u^n_\s}{u^\infty_\s}  \right ) \leq 
  u^\infty_{K,\flat} m \left ( \frac{u^n_K}{u^\infty_{K,\flat}} , \frac{u^n_\s}{u^\infty_{K,\flat}}  \right )
  = m \left ( u^n_K, u^n_\s \right ).
  \]
  By definition \eqref{def:f} of $f_{|\E_K|}$, and the bound \eqref{def:boundfaces} on $|\E_K|$, we then have 
  \begin{equation*}
    \underset{\s \in \E_K }{\max } \ u^\infty_{K,\flat} m \left ( \frac{u^n_K}{u^\infty_K} , \frac{u^n_\s}{u^\infty_\s}  \right )\leq\underset{\s \in \E_K }{\max } \ m \left ( u^n_K, u^n_\s \right ) \leq |\E_K| r_K (\u_K^n) \leq d \theta_\D ^2 \ r_K (\u_K^n).
  \end{equation*}
  We deduce that 
  \[
  \begin{split}
    4 u^\infty_{K,\flat}\, \delta_K  \rac_K^n \cdot \A _K  \delta_K \rac_K^n &
    \leq d \theta_\D ^2  \ r_K (\u_K^n)  \sum_{\s \in \E_K}  B_K ^{\s\s} \left (\log \left ( {\frac{u^n_K}{u^\infty_K} }  \right )-  \log \left (  {\frac{u^n_\s}{u^\infty_\s} } 	\right )  \right )  ^2\\
    & = d \theta_\D ^2 \ r_K (\u_K^n) \  \delta_K  \logg_K ^n  \cdot \B _K  \delta_K  \logg_K ^n,
  \end{split}
  \]
  where $\w_K^n\in\V_K$ is such that $\u_K^n=\u_K^\infty\times{\rm exp}(\w_K^n)$. Using again~\eqref{CB} from Lemma~\ref{lemma:cond}, we finally infer that
  $$4 u^\infty_{K,\flat}\, \delta_K  \rac_K^n \cdot \A _K  \delta_K \rac_K^n\leq d \theta_\D ^2 C_B\ r_K (\u_K^n) \  \delta_K  \logg_K ^n  \cdot \A _K  \delta_K  \logg_K ^n,$$
  with $C_B>0$ only depending on $\Lambda$, $\Omega$, $d$, and $\theta_\D$.
  Summing over $K \in \M$, and recalling the definitions~\eqref{hatD} of $\hat{\Diss}^n$, and~\eqref{entro1} of $\Diss^n$, eventually yields~\eqref{Fisher:dissipation} with $C_F =d \theta_\D ^2  C_B $.
\end{proof}

\noindent
The long-time behaviour of the nonlinear HFV scheme is studied in the following result.

\begin{theorem}[Asymptotic stability] \label{Th:nonlin}
  Recall the discrete entropy/dissipation relation of Proposition~\ref{prop:dissip}.
  The discrete entropy decays exponentially fast in time: there is ${\nu_{\mathrm{nl}}}\defi\frac{4  u_\flat^\infty \lambda_\flat \alpha_\flat}{C_FC_{LS,\infty}^2}>0$, depending on $\Lambda$, $\phi$, $M$, $\Omega$, $d$, and $\theta_\D$ such that 
  \begin{equation} \label{nonlin:decentropie}
    \forall n \in \mathbb{N}, \qquad \N^{n+1} \leq (1 + \nu_{\mathrm{nl}} \,\Delta t ) ^{-1} \N^n. 
  \end{equation} 
  Consequently, the discrete solution converges exponentially fast in time towards its associated discrete steady-state: for all $n\in\mathbb{N}^\star$,
  \begin{equation} \label{decL1}
    \| u^n_\M - u^\infty_\M \|_{L^1(\Omega)} \leq \sqrt{2 M \N^0} \left ( 1 + \nu_{\mathrm{nl}} \Delta t \right ) ^{-\frac{n}{2}}.
  \end{equation}
\end{theorem}
\begin{proof}
  Let $n\in\mathbb{N}$. By definition~\eqref{hatD} of $\hat{\Diss}^n$, and from the coercivity estimate~\eqref{def:globalcoercivity}, we first infer that
  \[
  \hat{\Diss} ^{n+1}  = 4 \sum_{K \in \M } u_{K,\flat}^\infty a_K^\Lambda\big( \rac_K^{n+1} , \rac_K^{n+1}\big)
  \geq  4  u_\flat^\infty a^\Lambda_\D \big( \rac_\D^{n+1} , \rac_\D^{n+1}   \big)
  \geq  4  u_\flat^\infty \lambda_\flat \alpha_\flat\big| \rac_\D^{n+1}  \big| _{1,\D}^2,
  \]
  which, combined with~\eqref{Fisher:dissipation}, implies that
  $$\Diss^{n+1}\geq\frac{1}{C_F}\hat{\Diss}^{n+1}\geq\frac{4  u_\flat^\infty \lambda_\flat \alpha_\flat}{C_F}\big| \rac_\D^{n+1}  \big| _{1,\D}^2.$$
  In order to compare this quantity with the entropy, we use the discrete $\log$-Sobolev inequality~\eqref{LogSob} from Proposition~\ref{prop:logsob}, applied to the couple $(\u_\D^{n+1}, \u_\D^\infty)$ (which satisfies the mass condition owing to Proposition~\ref{prop:mass}). We get
  \[
  \N^{n+1}  = \int_{\Omega}u_\M^\infty \Phi_1 \left ( \frac{u^{n+1}_\M}{u_\M^\infty} \right )
  \leq C_{LS,\infty}^2 \, \big|\rac_\D^{n+1} \big|_{1,\D}^2,
  \] 
  which, combined with the previous estimate, yields 
  \[
  \Diss^{n+1}\geq\frac{4  u_\flat^\infty \lambda_\flat \alpha_\flat}{C_FC_{LS,\infty}^2} \,\N^{n+1}.
  \]
  Combined with~\eqref{diss:nonlin} from Proposition~\ref{prop:dissip}, this shows~\eqref{nonlin:decentropie}.
  The $L^1$-norm estimate~\eqref{decL1} is then a direct consequence of~\eqref{nonlin:decentropie} and of the Csisz\'ar--Kullback lemma (cf., e.g.,~\cite[Lemma 5.6]{CCHHK:20}) applied to the probability measure $\mu(x)\,{\rm d}x=u_\M^\infty(x)\,\frac{{\rm d}x}{M}$ and to the function $g=\frac{u_\M^n}{u_\M^\infty}$ such that $\int_\Omega g{\rm d}\mu=1$, which yields $\|u_\M^n-u_\M^\infty\|_{L^1(\Omega)}\leq\sqrt{2M\N^n}$ for all $n\geq 1$.
\end{proof}
{
\begin{rem}[Norms and long-time behaviour]
Notice that Theorem \ref{Th:nonlin} states an exponential decay in $L^1$-norm, whereas Theorems~\ref{Th:hmm} and~\ref{Th:omega} assert a convergence in $L^2$, and thus a convergence in $L^p$ for any $p\in [1,2]$. 
This is reminiscent of the fact that the natural topologies for the linear and nonlinear problems differ.
\end{rem}
}

\section{Numerical results} \label{Numerical}

\subsection{Implementation}

In this section, we discuss some practical aspects concerning the implementation of the schemes described in this paper.
In all the test-cases presented below, the two-dimensional domain is taken to be $\Omega = (0,1)^2$.
The meshes used for the numerical tests, presented on Figure \ref{fig:meshes}, are the classical Cartesian, triangular, and Kershaw meshes from the FVCA V benchmark (see~\cite{HerHu:08}), as well as a tilted hexagonal-dominant mesh (cf.~\cite{DPLe:15}).
\begin{figure}[h]
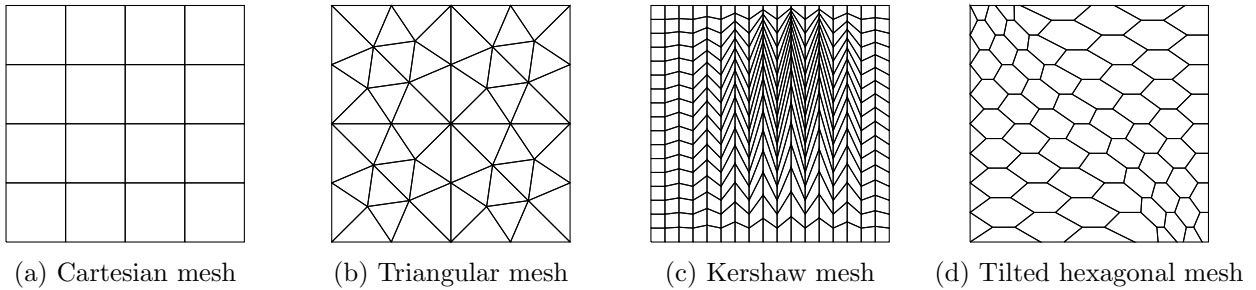

\begin{subfigure}{.25\textwidth}
  \centering
  \includegraphics[width=.75\textwidth]{meshes/mesh2_1.tikz}
  \caption{Cartesian mesh}
  \label{fig:cartesian}
\end{subfigure}
\begin{subfigure}{.25\textwidth}
  \centering
  \includegraphics[width=.75\textwidth]{meshes/mesh1_1.tikz}
  \caption{Triangular mesh}
  \label{fig:triangular}
\end{subfigure}%
\begin{subfigure}{.25\textwidth}
  \centering
  \includegraphics[width=.75\textwidth]{meshes/mesh4_1_1.tikz}
  \caption{Kershaw mesh}
  \label{fig:kershaw}
\end{subfigure}%
\begin{subfigure}{.25\textwidth}
  \centering
  \includegraphics[width=.75\textwidth]{meshes/pi6_tiltedhexagonal_2.tikz}
  \caption{Tilted hexagonal mesh}
  \label{fig:tiltedhexa}
\end{subfigure}
\caption{\textbf{Implementation.} Coarsest meshes of each family used in the numerical tests.}
\label{fig:meshes}
\end{figure}
These meshes have convex cells, hence we always choose $x_K$ to be the barycentre of $K\in\M$.
{In our implementation, we compute the meshsize as $\tilde{h}_\D=\underset{ K \in \M}{\max} \frac{|K|}{|\partial K|}$. Observe that $\tilde{h}_\D/h_\D$ is framed by constants only depending on the mesh regularity.} 
{Notice also that the Kershaw mesh family is not uniformly regular in the sense defined in Section~\ref{sec:discre}.}
{In practice, we use a fixed value $\sqrt{2}<\eta = 1.5<2$ of the stabilisation parameter (see Remark~\ref{rem:stab})}.
In the sequel, we denote by HMM the classical HFV linear scheme for advection-diffusion, and we restrict our attention to the \SG discretisation of the flux \eqref{def_conv}, namely to the function $A(s) = \frac{s}{\e^s -1} -1$, extended by continuity to $0$ at $s=0$.
%% In the following figures and tables, the schemes will be denoted by \emph{Nonlinear}, \emph{HMM}, \emph{ExpF} and \emph{ExpF (harmonic)} for the two exponential fitting schemes.

\subsubsection{Linear systems and static condensation}

The two linear (HMM and exponential fitting) schemes are implemented in the same way. To fix ideas, we consider the evolution problem with pure Neumann boundary conditions, of unknown solution $\u_\D \in \V_\D$. We denote by $U_\M \in \R^{|\M|}$ and $U_\E \in \R^{|\E|}$ the unknown vectors $(u_K)_{K\in\M}$ and $(u_\s)_{\s \in \E}$. The linear schemes result in the following block system: 
\begin{equation} \label{sysmatrix}
	\begin{pmatrix}
	\Mat_\M & \Mat_{\M,\E} \\ 
	\Mat_{\E,\M} & \Mat_\E
	\end{pmatrix} 
	\begin{pmatrix} U_\M \\  U_\E \end{pmatrix} 
	= \begin{pmatrix} S_\M \\  S_\E \end{pmatrix},  
\end{equation}
where $\Mat_\M \in \R^{|\M| \times|\M| }$, $\Mat_{\M,\E} \in \R^{|\M| \times|\E|}$, $\Mat_{\E,\M} \in \R^{|\E| \times|\M|}$, $\Mat_\E \in \R^{|\E|\times |\E|}$, and  $S_\M\in\R^{|\M|}$ and $S_\E\in\R^{|\E|}$ stem from the loading term and the boundary conditions.
By construction, the matrix $\Mat_\M$ is diagonal with non-zero diagonal entries, and can therefore be inverted at a very low computational cost.
Thus, one can eliminate the cell unknowns, noticing that
\begin{equation} \label{bulkunk}
	U_\M = \Mat_\M ^{-1} \left ( S_\M - \Mat_{\M,\E} U_\E \right ).
\end{equation}
Using this relation, one infers that $U_\E$ is the solution to the following linear system:
\begin{equation}\label{squelunk}
	\left ( \Mat_\E - \Mat_{\E,\M}\Mat_\M ^{-1} \Mat_{\M,\E} \right ) U_\E 
		= S_\E - \Mat_{\E,\M} \Mat_\M ^{-1} S_\M, 
\end{equation}
where $\displaystyle \Mat_\D \defi \Mat_\E - \Mat_{\E,\M}\Mat_\M ^{-1} \Mat_{\M,\E}$, the so-called Schur complement of the matrix $\Mat_\M$, is an invertible matrix of size $|\E|\times |\E|$. 
In practice, we solve the linear system \eqref{squelunk} using an LU factorisation algorithm, and we use the solution $U_\E$ to reconstruct $U_\M$ from~\eqref{bulkunk}. This method, called \textit{static condensation}, allows one to replace a system of size $|\M|+|\E|$ by a system of size $|\E|$ without additional fill-in.
In the case of mixed Dirichlet-Neumann boundary conditions, the Dirichlet face unknowns are eliminated from the global linear system.

% 
% Note that for the evolutive schemes, we need to solve linear systems of the form 
% \begin{equation} \label{sysmatrixevol}
% 	\begin{pmatrix}
% 	\Mat_\M & \Mat_{\M,\E} \\ 
% 	\Mat_{\E,\M} & \Mat_\E
% 	\end{pmatrix} 
% 	\begin{pmatrix} U_\M ^{n+1}\\  U_\E^{n+1} \end{pmatrix} 
% 	= \begin{pmatrix} S_\M^n \\  S_\E^n \end{pmatrix},  
% \end{equation}
% and we also perform static condensation. Here, the matrices do not depend on $n$, and so is its Schur complement $\Mat_\D$. Thus, one only needs to perform the LU factorization of $\Mat_\D$ once in order to compute the $(\u^n_\D)_{n \geq 1}$: the computational cost will be rather low, since the resolution after the factorisation cost almost nothing compared to the factorisation.
% \newline 
% However, we did not perform test on very refined mesh, and we expect that LU factorization can not be used if $|\E|$ is too high, which would impose the use of iterative methods for solving \eqref{sysmatrixevol}. In such situation, one needs to solve a system at each time step.
%
\subsubsection{Exponential fitting scheme: choice of unknown and harmonic averaging}

The exponential fitting scheme can be expressed in either the $u$ or the $\rho = u\e^\phi$ variable. In the $\rho$ variable, the resulting linear system is symmetric. One can then use, e.g., Cholesky factorisation or a conjugate gradient method. However, the formulation in $\rho$ is ill-conditioned. In our numerical experiments, the ratio between the condition numbers of the linear systems in $\rho$ and in $u$ often exceeds $10^5$. Because of this, we chose and we recommend to solve the linear system in the unknown $u$.
{
Notice that solving the system in $u$ is equivalent to right pre-condition the system in $\rho$ with the inverse of the diagonal matrix with entries the coordinates of the interpolate of $\omega = \e^{-\phi}$.
}

In order to implement the exponential fitting scheme, one needs to evaluate averages of the diffusion tensor $\frac{1}{|P_{K,\s}|}\int_{P_{K,\s}} \omega \Lambda$. Observe that $\omega(x)/\omega(x_K)$ is of order $\e^{h_K \|\nabla \phi \|_{L^\infty(P_{K,\s})}}$ in $P_{K,\s}$. Therefore, with large advection fields, the diffusion problem \eqref{def:sta:modeq} becomes strongly heterogeneous. It is pointed out in \cite{BMaPi:89} that an (empirical) solution to improve robustness to this heterogeneity is to use harmonic averages to approximate integrals of the diffusion tensor. In the numerical tests of the following subsections, we compare the ``classical" exponential fitting scheme (for which the integral is approximated by a standard - second order - quadrature) with the ``harmonic" one, in which case we choose to use the following approximation:
\[
\frac{1}{|P_{K,\s}|}\int_{P_{K,\s}} \omega \Lambda \approx
		 3 \left (\ \sum_{F \in \E_{P_{K,\s}}} \frac{1}{\omega(\overline{x}_F)} \right )^{-1}\Lambda(x_K), 
\]
where $\E_{P_{K,\s}}$ denotes the set of edges of the triangle $P_{K,\s}$ (recall that $d=2$ in our experiments), and $\overline{x}_F$ is the barycentre of $F\in\E_{P_{K,\s}}$.

\subsubsection{Nonlinear scheme and Newton's method}

The implementation of the nonlinear scheme relies on the following formulation: given $\u^{n-1}_\D \in \V_\D$ positive, we want to solve the nonlinear system $\Gu_\D^{n,\delta t}(\u^{n}_\D) =\underline{0}_\D$, where $\Gu_\D^{n,\delta t}$ is defined as in \eqref{def:G} but with a time step $\delta t$ instead of $\Delta t$.
The resolution of this system relies on Newton's method.

First, one initialises the method with $\delta t = \Delta t$, and initial guess $\max(\u^{n-1}_\D, \epsilon \one_\D) \in \V_\D$ (where the maximum is taken coordinate by coordinate), in order to avoid potential problems due to the singularity of the $\log$ near $0$.
The successive linear systems to compute the residue have the same structure as \eqref{sysmatrix}. We thus perform static condensation at each Newton iteration. As a stopping criterion, we compare the $l^\infty$ relative norm of the residue with a threshold $tol$.
If the method does not converge after $i_{max}$ iterations, we divide the time step by $2$, and we restart the resolution. 
When the method converges, one can proceed with the approximation of $\u^{n+1}_\D$, with an initial time step of $\min(\Delta t,  2 \delta t)$. 
In practice, we use $\epsilon = 10^{-11}$, $i_{max} = 50$, and $tol = 10^{-11}$.
%	
%The nonlinear scheme has a higher computational cost than the two linear schemes. Indeed, one needs at least to assemble and  solve one linear system at each time step. 

The implementation of the nonlinear scheme relies on the computation of $\log(\omega_K)$ and $\log(\omega_\s)$. Since we have chosen $x_K$ to be the barycentre of $K$, we choose to approximate $\frac{1}{|K|} \int_K \e^{-\phi}$ by $\e^{-\phi(x_K)}$. Therefore, $\log(\omega_K)$ is computed as $\log(\e^{-\phi(x_K)}) = -\phi(x_K)$.
The same holds true for $\log(\omega_\s)$.

In the simulations shown below, we use arithmetic means for the functions $m$ and $f_{|\E_K|}$ of the reconstruction $r_K$ defined by \eqref{eq:rK}-\eqref{def:f}. For all $K\in\M$, and all $\u_K \in \V_K$, we thus consider 
\[
	r_K(\u_K) = \frac{1}{2} \left ( u_K + \frac{1}{|\E_K|}\sum_{\s \in \E_K}  u_\s \right ).
\]
This choice is close to the one advocated in~\cite[Eq.~(58)]{Cances:18}.
For a discussion on other choices of reconstructions, we refer to \cite[Section 6.2]{CCHHK:20}.
% The numerical investigation carried out in Section 6.2 of \cite{CCHHK:20} in the framework of TPFA nonlinear scheme suggests that the choice of $m$ doesn't have a real impact, providing that the function is regular enough (the choice $m(x,y) = \max (x,y)$ proved to reduce the order of convergence in space). We expect the same behaviour with our HVF nonlinear scheme.

\subsection{Long-time behaviour of discrete solutions}

In this section, we present some numerical illustration of the long-time behaviour of discrete solutions.
We focus on a test-case from \cite{CaGui:17,CCHKr:18,CCHHK:20}. We consider homogeneous pure Neumann boundary conditions ($\Gamma^D = \emptyset$ and $g^N=0$), and zero loading term ($f = 0$).
The advective potential and diffusion tensor are set to
$\displaystyle \phi(x,y) = - x $ and 
$\displaystyle \Lambda = \begin{pmatrix}
l_x & 0 \\ 
0 & 1
\end{pmatrix} $ for $l_x>0$.
The exact solution is given by 
\[
	u(t,x,y) = C_1\e^{-\alpha t + \frac{x}{2}} \left ( 2\pi \cos( \pi x) + \sin(\pi x) \right )
	+ 2 C_1 \pi \e^{ x - \frac{1}{2}  }, 
\]
where $C_1> 0$ and $\displaystyle \alpha = l_x \left ( \frac{1}{4} + \pi^2  \right )$. Note that $u^{in}$ vanishes on 
$\{ x = 1 \}$, but for any $t > 0$, $u(t, \cdot) > 0$. The associated steady-state is 
\[
	u^\infty(x,y) = 2 C_1 \pi \e^{ x - \frac{1}{2}  }.
\]
Our experiments are performed using the following values: 
\[
	l_x =  10^{-2}  \qquad\text{ and } \qquad C_1 = 10^{-1} .
\]
We compute the solution on the time interval $[0,T_f]$, and we denote by $(\u^n_\D)_{1 \leq n \leq N_f}$ the corresponding approximate solution. Note that the number of time steps $N_f$ may differ between the linear and nonlinear schemes, because of the adaptive time step refinement procedure used for the nonlinear scheme.

%\bibliographystyle{siam} 
%\bibliography{HFV_longtime}
%\end{document}
We set $T_f = 350$ in order to see the complete evolution.
Since the long-time behaviour of the schemes does not depend on the size of the discretisation, we can explore the evolution using a large time step $\Delta t = 10^{-1}$. We perform the numerical experiments on two Kershaw meshes (see Figure \ref{fig:kershaw}) of sizes $0.02$ and $0.006$. 
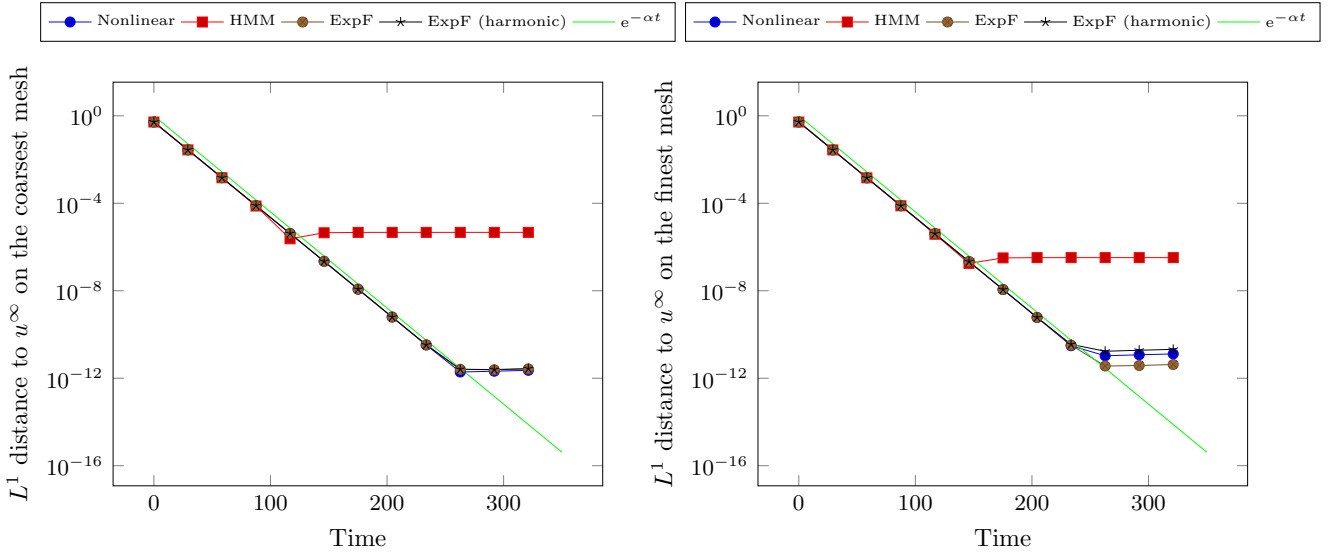
\begin{figure}[h]
\begin{minipage}[c]{.5\linewidth}
\begin{tikzpicture}[scale= 0.94]
        \begin{semilogyaxis}[
            legend style = { 
              at={(0.5,1.1)},
              anchor = south,
              tick label style={font=\footnotesize},
              legend columns=-1
            },
            ylabel=\small{$L^1$ distance to $u^\infty$ on the coarsest mesh},
            xlabel=\small{Time}
            ]
          \addplot table[x=Temps,y=Diff_L1] {graph_data/tps_long/mesh4_1_1/tps_nonlin};
          \addplot table[x=Temps,y=Diff_L1] {graph_data/tps_long/mesh4_1_1/tps_HMM};
          \addplot table[x=Temps,y=Diff_L1] {graph_data/tps_long/mesh4_1_1/tps_EFari}; 
          \addplot table[x=Temps,y=Diff_L1] {graph_data/tps_long/mesh4_1_1/tps_EFharmo};
          \addplot[green] coordinates {
			(0,1)
			(350,4.1484969e-16)
			};
          \legend{\tiny Nonlinear , \tiny HMM,\tiny ExpF,\tiny ExpF (harmonic) , \tiny $ \e^{- \alpha t} $ }          
	      \end{semilogyaxis}
      \end{tikzpicture}    
\end{minipage}
\begin{minipage}[c]{.5\linewidth}
\begin{tikzpicture}[scale= 0.94]
        \begin{semilogyaxis}[
            legend style = { 
              at={(0.5,1.1)},
              anchor = south,
              tick label style={font=\footnotesize},
              legend columns=-1
            },ylabel=\small{$L^1$ distance to $u^\infty$ on the finest mesh},xlabel=\small{Time}
          ]
          \addplot table[x=Temps,y=Diff_L1] {graph_data/tps_long/mesh4_1_4/tps_nonlin};
          \addplot table[x=Temps,y=Diff_L1] {graph_data/tps_long/mesh4_1_4/tps_HMM};
          \addplot table[x=Temps,y=Diff_L1] {graph_data/tps_long/mesh4_1_4/tps_EFari}; 
          \addplot table[x=Temps,y=Diff_L1] {graph_data/tps_long/mesh4_1_4/tps_EFharmo};
          \addplot[green] coordinates {
			(0,1)
			(350,4.1484969e-16)
			};
          \legend{\tiny Nonlinear , \tiny HMM, \tiny ExpF,\tiny ExpF (harmonic), \tiny $\e^{- \alpha t} $ }          
        \end{semilogyaxis}
      \end{tikzpicture}
\end{minipage}
\caption{\textbf{Long-time behaviour of discrete solutions.} Comparison of the long-time behaviour on Kershaw meshes for $T_f = 350$ and $\Delta t = 0.1$.}
\label{fig:longtime}
\end{figure}
In Figure \ref{fig:longtime}, we depict, as a function of time, the $L^1$ distance between $\u^n_\D$ and $u^\infty$ (the exact steady-state) computed as 
\[
		\sum_{K \in \M } |K| |u^n_K - u^\infty(x_K)|.
\]
We observe the exponential decay towards the steady-state, until some precision is reached. 
Note that for the HMM scheme, some saturation occurs at precision of magnitude $10^{-6}$ and $10^{-7}$: the scheme does not preserve the thermal equilibrium (see Remark \ref{rem:SG}). This saturation corresponds to the accuracy of the stationary scheme (see Section~\ref{acc}), so the threshold is lower on the refined mesh.
Note that one could also consider the error measure $\displaystyle \sum_{K \in \M } |K| |u^n_K - u^\infty_K|$ between the discrete solution and the discrete steady-state: this quantity decays exponentially, with a lower saturation of magnitude $10^{-12}$, corresponding to machine precision. The other (nonlinear and exponential fitting) schemes have the same decay rate, and the saturation occurs at machine precision.
For the four schemes, the rates of convergence are similar to the real one $\alpha$.
In particular, the use of harmonic averages in the exponential fitting scheme does not have any impact on the long-time behaviour.

\subsection{Positivity of discrete solutions}

We are now interested in the positivity of the discrete solutions. 
We use the following test-case with anisotropic diffusion and homogeneous pure Neumann boundary conditions.
We set $\Gamma^D = \emptyset$, $f=0$, $g^N = 0$, 
\[
\phi(x,y) = -\left (  (x-0.4)^2 +(y-0.6)^2   \right )\qquad \text{ and }\qquad 
\Lambda = \begin{pmatrix}
0.8 & 0 \\ 
0 & 1
\end{pmatrix} . \]
For the initial datum, we take
\[u^{in} = 10^{-3}  \,\1_{B} +   \1_{\Omega \setminus B},\]
where $B$ is the Euclidean ball  $\left \lbrace (x,y) \in \R^2 \mid (x-0.5)^2 + (y-0.5)^2 \leq 0.2^2 \right \rbrace $.
These data ensure that the solution $u$ is positive on $\R_{+}\times\Omega$. 
The experiment is performed on a tilted hexagonal-dominant mesh (see Figure \ref{fig:tiltedhexa}) of size $4.3 \cdot 10 ^{-3}$, made up of 4192 cells and 12512 edges. 
Since we deal with a diffusive phenomenon, the smallest values of $u$ are expected for small time, hence we perform the simulation with a relatively small final time $T_f = 5 . 10^{-4}$, alongside with a time step of $\Delta t = 10^{-5}$.

The results are collected in Table \ref{table:positivity}.
The cost is defined as the number of linear systems solved in order to compute the solution 
$(\u^n_\D)_{1 \leq n \leq N_f}$,
and the minimum values min\_cells and min\_edges are defined by 
\[
	\min \left \{ u_K^n \mid 1 \leq n \leq N_f, \ K \in \M \right \}
	\quad\text{ and }\quad
	 \min \left \{ u_\s^n \mid 1 \leq n \leq N_f, \ \s \in \E \right \}.
\]
The indicated number of negative unknowns is for the whole simulation. Here, $N_f=50$ also for the nonlinear scheme (no sub-division of the time step was needed in Newton's method).         
\begin{table}[h]
\center
\begin{tabular}{|c|c|c|c|c|c|c|} 
\hline 
  & cost  & min\_cells & min\_edges & \# negative unknowns \\ 
\hline 
		Nonlinear 	& 175 	& 9.93e-04	& 7.36e-04 	& 0 \\ 
\hline 
		HMM 			& 50  	& -5e-03 		& --7.74e-02 & {593} \\ 
\hline 
		ExpF 			& 50 	& -4.98e-03 	&-7.72e-02 & {590} \\ 
\hline 
	ExpF (harmonic) & 50 	& -4.98e-03 	& -7.74e-02& {588}\\ 
\hline 
\end{tabular} 
\caption{\textbf{Positivity of discrete solutions.} Numerical results for $T_f =5. 10^{-4}$ and $\Delta t= 10^{-5}$ on a tilted hexagonal-dominant mesh. 
At each time step, there are 4192 cell unknowns and 12512 edge unknowns.
}\label{table:positivity}
\end{table}
As expected, the nonlinear scheme has positive discrete solutions, whereas the linear ones exhibit a violation of positivity (the value of $\eta$ can have some influence on positivity; see~\cite{EGaHe:08}). 
Note that the use of harmonic averages for the exponential fitting scheme has no impact on the undershoots. 

We observe that the nonlinear scheme requires approximately $3.5$ times more linear system inversions than the linear schemes. However, this value depends strongly on the final time of simulation $T_f$.
Indeed, the number of linear systems solved at step $n$ decreases when $n$ increases.
The first time step costs 9 resolutions, but this number rapidly decreases as the solution approaches the steady-state (the second and the third time steps respectively cost 5 and 4 resolutions). 

\subsection{Accuracy of stationary solutions} \label{acc}

In this section, we aim at comparing the accuracy of the different schemes for the stationary problem. To do so, we define the discrete $L^2$-norm and $H^1$-seminorm errors as (i) the $L^2$-norm of the difference $u_\M-\Pi_\M(u)$, and (ii) the $|\cdot|_{1,\D}$-seminorm of the difference $\underline{u}_\D-\underline{\Pi}_\D(u)$, where $\u_\D$ is the discrete solution, and $\underline{\Pi}_\D(u)$ is the HFV interpolate of the continuous solution $u$ (computed as $\underline{\Pi}_\D(u)\approx((u(x_K))_{K\in\M}, (u(\overline{x}_\s))_{\s\in\E})$). In what follows, we reason in relative errors.
% \[
% 	errL^2 = \sqrt{ 
% 	\frac{\sum_{K \in \M} |K|(u_K - u(x_K))^2 } 
% 	{\sum_{K \in \M} |K| u(x_K)^2 } 
% 	} 
% 	\text{ and } 
% 	errH^1
% 	=
% 	\sqrt{
% 	\frac{
% 	\sum_{K \in \M} \sum_{\s \in \E_K}
% 	\frac{|\s|}{d_{K,\s}} \left (u_K - u_\s - u( x_K) +u( \overline{x}_\s) \right )^2
% 	}{
% 	\sum_{K \in \M} \sum_{\s \in \E_K}
% 	\frac{|\s|}{d_{K,\s}} \left (u( x_K) - u( \overline{x}_\s) \right )^2
% 	}
% 	}.
% \]
% %

The nonlinear scheme is extended to a more general setting, in order to take into account a loading term $f \geq 0$ and mixed Dirichlet-Neumann boundary conditions ($|\Gamma^D|>0$) with $g^D>0$ and $g^N\geq 0$.
The scheme writes:
\begin{equation} \label{def:sch:nonlinsta}
	\begin{array}{c}
		\text{Find } \u_\D \in \V_\D \text{ positive such that } \Gu_\D(\u_\D) = \underline{0}_\D , 
	\end{array} 
\end{equation}
where $ \Gu_\D : \left(\V_\D\right)^{\star}_+  \to \V_\D $ is the vector field defined by
\begin{subequations}\label{def:Gsta}
	\begin{align}
	 & \G_K(\u_\D)  \defi \sum_{\s \in \E_K}\F^{{\rm nl}}_{K,\s}(\u_K) - \int_K f \quad&&\forall K\in\M, \label{def:Gsta:mesh}\\
         & \G_\s(\u_\D)  \defi - \left ( \F^{{\rm nl}}_{K,\s}(\u_K) + \F^{{\rm nl}}_{L,\s}(\u_L) \right ) \quad&&\forall \s = K \mid L  \in \E_{int},  \label{def:Gsta:edgeint}\\  
         & \G_\s(\u_\D)  \defi -\int_\s g^N- \F^{{\rm nl}}_{K,\s}(\u_K) \quad&&\forall\s\in\E_{ext}^N\text{ with }\M_{\s}=\{K\}, \label{def:Gsta:edgeextNeu}\\
         & \G_\s(\u_\D)  \defi\frac{1}{|\s|}\int_\s g^D - u_\s \quad&&\forall \s   \in \E_{ext}^D, \label{def:Gsta:edgeextDir}\\
         & \F^{{\rm nl}}_{K,\s}(\u_K) = r^K(\u_K) \sum_{\s' \in \E_K }  A_K^{\s\s'} \left ( \log (u_K ) + \phi(x_K) - \log (u_{\s'} ) - \phi(\overline{x}_{\s'}) \right ) \quad&&\forall K \in \M, \forall \s \in \E_K.
	\end{align}
\end{subequations}
The implementation of this scheme still relies on a Newton method similar to the one used for the evolution scheme.
It is here initialised with $\underline{\Pi}_\D(\omega)=\underline{\omega}_\D$.
%In the tests presented here, the nonlinear scheme converges, and it needs between 3 and 10 linear system resolutions (the Newton iterations). We observe that the number of iteration decrease with the mesh size.

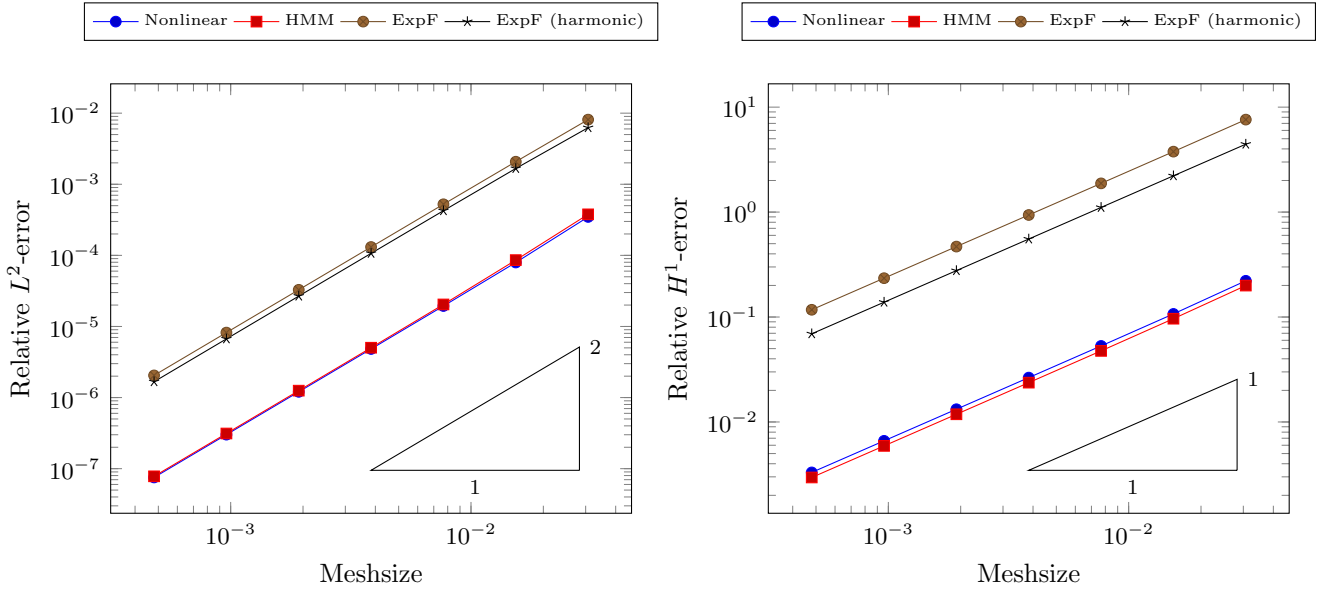
\begin{figure}[h]
\begin{minipage}[c]{.51\linewidth}
\begin{tikzpicture}[scale= 1]
        \begin{loglogaxis}[
            legend style = { 
              at={(0.5,1.1)},
              anchor = south,
              tick label style={font=\footnotesize},
              legend columns=-1
            },ylabel=\small{Relative $L^2$-error},xlabel=\small{Meshsize}
          ]
          \addplot table[x=meshsize,y=errl2] {graph_data/HMMtest/nonlin};
          \addplot table[x=meshsize,y=errl2] {graph_data/HMMtest/HMM};
          \addplot table[x=meshsize,y=errl2] {graph_data/HMMtest/expfittari}; 
          \addplot table[x=meshsize,y=errl2] {graph_data/HMMtest/expfittharmo};
          \logLogSlopeTriangle{0.90}{0.4}{0.1}{2}{black};
          \legend{\tiny Nonlinear , \tiny HMM,\tiny ExpF,\tiny ExpF (harmonic)}
        \end{loglogaxis}
      \end{tikzpicture}    
\end{minipage}
\begin{minipage}[c]{.51\linewidth}
\begin{tikzpicture}[scale= 1]
        \begin{loglogaxis}[
            legend style = { 
              at={(0.5,1.1)},
              anchor = south,
              tick label style={font=\footnotesize},
              legend columns=-1
            },ylabel=\small{Relative $H^1$-error},xlabel=\small{Meshsize}
          ]
          \addplot table[x=meshsize,y=errh1] {graph_data/HMMtest/nonlin};
          \addplot table[x=meshsize,y=errh1] {graph_data/HMMtest/HMM};
          \addplot table[x=meshsize,y=errh1] {graph_data/HMMtest/expfittari}; 
          \addplot table[x=meshsize,y=errh1] {graph_data/HMMtest/expfittharmo};
          \logLogSlopeTriangle{0.90}{0.4}{0.1}{1}{black};
          \legend{\tiny Nonlinear , \tiny HMM, \tiny ExpF,\tiny ExpF (harmonic)}          
        \end{loglogaxis}
      \end{tikzpicture}
\end{minipage}
\caption{\textbf{Accuracy of stationary solutions.} Relative errors in discrete $L^2$-norm and $H^1$-seminorm for the first test-case on triangular meshes.}
\label{fig:HMMtest}
\end{figure}

The first test-case we consider is the same as in \cite{BdVDM:11}. It is an isotropic problem, with $\Lambda = I_2$, $\Gamma^N = \emptyset$, $\phi(x,y) = -(2x + 3y)$, and exact solution  
\[
	u(x,y) = \left ( x - \e^{2(x-1)} \right ) \left (y- \e^{3(y-1)} \right ), 
\]
the other data $f$ and $g^D$ being set accordingly.
Note that for this test-case, the diffusion and advection terms are of the same order of magnitude. The numerical experiments are performed on the triangular mesh family (see Figure \ref{fig:triangular}).
The convergence results are depicted in Figure \ref{fig:HMMtest}. As expected, the two linear schemes are of order two in $L^2$-norm, and one in $H^1$-seminorm. The same holds for the nonlinear scheme, whose accuracy is rather the same as the classical HMM scheme, one order of magnitude better than the exponential fitting schemes.
On this test-case, the use of harmonic averages for the exponential fitting scheme does not have a significant impact.

The second test-case is an advection-dominated problem, with anisotropic diffusion and mixed Dirichlet-Neumann boundary conditions. We set $\Gamma^D = (\{0\}\times (0,1)) \cup (\{1\}\times (0,1))$, 
$\Gamma^N = ( (0,1)\times\{0\}) \cup ( (0,1)\times\{1\})$, $g^D=1$, $g^N=0$, and $f=0$. 
The diffusion tensor and the potential are defined by the following expressions:
\[
	\Lambda = \begin{pmatrix}
	1 & 0 \\ 
	0 & l_y
	\end{pmatrix}
	\qquad\text{ and }\qquad
	\phi(x,y) = \log \left( \frac{1}{v} + x \right ),
\]
with $l_y,v>0$. Note that the advection field 
$\displaystyle V^\phi = -
\begin{pmatrix}
\frac{v}{1+vx} \\ 
0
\end{pmatrix}
$ has a magnitude of order $v$ when $x$ is small. Thus, near the boundary $\{0\}\times [0,1]$, the problem is advection-dominated if $v$ is large enough.
Moreover, $\displaystyle \divergence(V^\phi) = \frac{v^2}{\left ( 1+vx\right ) ^2 } > 0$ and $V^\phi\cdot n=0$ on $\Gamma^N$, so the problem is coercive.
The exact solution is given by 
\[
	u(x,y) = \frac{v}{1+vx} \left (
	\frac{2v x }{2+v} \left ( \frac{1}{v} + \frac{x}{2} \right )
	+\frac{1}{v}
	 \right ).
\]
We perform our numerical experiments on the Cartesian mesh family (see Figure \ref{fig:cartesian}), with 
\[ 
l_y = 100 \qquad\text{ and } \qquad v = 200.
\] 

\begin{figure}[h]
\begin{minipage}[c]{.5\linewidth}
\begin{tikzpicture}[scale= 1]
        \begin{loglogaxis}[
            legend style = { 
              at={(0.5,1.1)},
              anchor = south,
              tick label style={font=\footnotesize},
              legend columns=-1
            },ylabel=\small{Relative $L^2$-error},xlabel=\small{Meshsize}
          ]
          \addplot table[x=meshsize,y=errl2] {graph_data/log(P)/nonlin};
          \addplot table[x=meshsize,y=errl2] {graph_data/log(P)/HMM};
          \addplot table[x=meshsize,y=errl2] {graph_data/log(P)/expfittari};
          \addplot table[x=meshsize,y=errl2] {graph_data/log(P)/expfittharmo};
          \logLogSlopeTriangle{0.90}{0.4}{0.1}{1}{black};
          \logLogSlopeTriangle{0.90}{0.4}{0.1}{2}{black};
          \legend{\tiny Nonlinear , \tiny HMM, \tiny ExpF,\tiny ExpF (harmonic)}  
        \end{loglogaxis}
      \end{tikzpicture}    
\end{minipage}
\begin{minipage}[c]{.5\linewidth}
\begin{tikzpicture}[scale= 1]
        \begin{loglogaxis}[
            legend style = { 
              at={(0.5,1.1)},
              anchor = south,
              tick label style={font=\footnotesize},
              legend columns=-1
            },ylabel=\small{Relative $H^1$-error},xlabel=\small{Meshsize}
          ]
          \addplot table[x=meshsize,y=errh1] {graph_data/log(P)/nonlin};
          \addplot table[x=meshsize,y=errh1] {graph_data/log(P)/HMM};
          \addplot table[x=meshsize,y=errh1] {graph_data/log(P)/expfittari};
          \addplot table[x=meshsize,y=errh1] {graph_data/log(P)/expfittharmo};
          \logLogSlopeTriangle{0.90}{0.4}{0.1}{1}{black};
          \logLogSlopeTriangle{0.90}{0.4}{0.1}{0.5}{black};
          \legend{\tiny Nonlinear , \tiny HMM, \tiny ExpF,\tiny ExpF (harmonic)}  
        \end{loglogaxis}
      \end{tikzpicture}
\end{minipage}
\caption{\textbf{Accuracy of stationary solutions.} Relative errors in discrete $L^2$-norm and $H^1$-seminorm for the second test-case on Cartesian meshes.}
\label{fig:log(P)}
\end{figure}
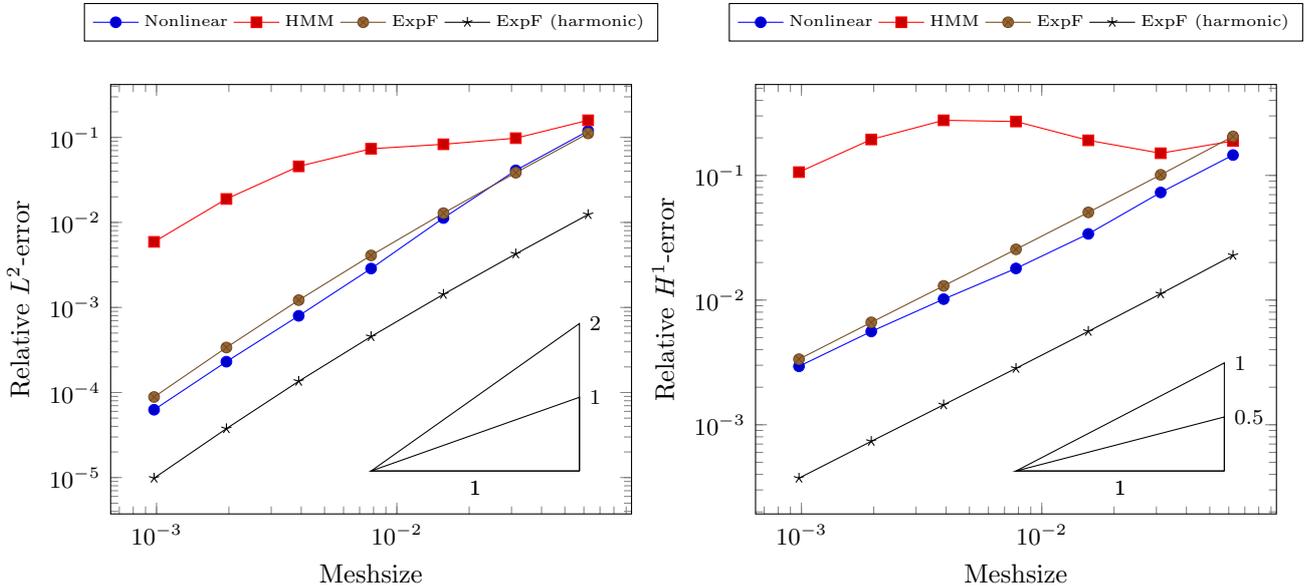
\noindent
The results are depicted in Figure \ref{fig:log(P)}.
They show that the HMM scheme suffers, most probably because of the fact that the advective term predominates over the diffusive term, at least in some part of the domain. The order of convergence of the HMM scheme is less than one in $H^1$-seminorm, and than two in $L^2$-norm. The other schemes converge with order one in $H^1$-seminorm, and two in $L^2$-norm. Moreover, on this test-case, their accuracy is better than that of the HMM scheme.
Notice that $\omega = \e^{-\phi} = \frac{v}{1+vx}$ has small variations (i.e., not exponential) in the cells, even if $v$ is large. Therefore, the diffusion tensor $\omega \Lambda$ of the problem in the $\rho$ unknown for the exponential fitting schemes is not that heterogeneous (locally). It could explain the good performances of the exponential fitting schemes in this case.
Moreover, on this test-case, using harmonic averages in the exponential fitting scheme gives a substantial gain in accuracy for both the $L^2$ and $H^1$ relative errors, of magnitude $10^1$.

\section{Conclusion} \label{con}

In this paper, by means of discrete entropy methods, we have analysed the long-time behaviour of three hybrid finite volume schemes for linear advection-diffusion equations.
We have proved that the solutions to all schemes converge exponentially fast in time towards the associated discrete steady-states.
Two schemes among the three are new, that are the (linear) exponential fitting scheme (adapting known ideas to the HFV context) and the nonlinear scheme, for which we have proved the existence of solutions.
All schemes can handle anisotropy and general meshes.
The two linear schemes can deal with general data and mixed Dirichlet-Neumann boundary conditions, however they do not preserve the positivity of solutions.
On the other hand, the nonlinear scheme preserves positivity and can be used in practice with general boundary conditions. However, at the moment, its asymptotic analysis is limited to systems that converge in time towards the thermal equilibrium, restricting the admissible data. We have finally validated our theoretical findings on different numerical tests, assessing long-time behaviour, positivity, and spatial accuracy of the schemes.

\appendix
\section{Functional inequalities}\label{Appendix_funcinequalities}

\subsection{Discrete Poincar\'e inequalities}

We recall the following hybrid discrete Poincar\'e inequalities (cf.~\cite[Lemmas B.25 and B.32, $p=2$] {DEGGH:18}).

\begin{prop}[Discrete Poincar\'e inequalities] \label{prop:poinca}
  Let $\D$ be a given discretisation of $\Omega$, with regularity parameter $\theta_\D$.
  There exists $C_{PW}>0$, only depending on $\Omega$, $d$, and $\theta_\D$ such that 
  \begin{equation} \label{poincawir}
    \forall \v_{\D} \in \V_{\D,0}^N , \qquad \|v_\M \|_{L^2(\Omega)} \leq {C_{PW}} | \v_{\D} |_{1,\D}.
  \end{equation}
  Assume that $|\Gamma^D| > 0$.
  Then, there exists $C_{P,\Gamma^D}>0$, only depending on $\Omega$, $d$, $\Gamma^D$, and $\theta_\D$ such that 
\begin{equation} \label{poincadir}
	\forall \v_{\D} \in \V_{\D,0}^D, \qquad \|v_\M \|_{L^2(\Omega)} \leq {C_{P,\Gamma^D}} | \v_{\D} |_{1,\D}.
\end{equation}
\end{prop}
{

\subsection{Logarithmic Sobolev inequalities}
  
In this section, we derive logarithmic Sobolev inequalities on a bounded domain, in the continuous setting. The intermediate results of Proposition~\ref{prop:nearly_logSob} below will be useful in the discrete setting. In the following, $\mu$ is a probability measure on the bounded domain $\Omega$, and the space $L^q_\mu(\Omega)$ denotes the space endowed with the norm $\|f\|_{L^q_\mu(\Omega)}^q =\int_\Omega |f|^q\dd\mu$. We start with a preliminary lemma, which is an adaptation of part of the proof of \cite[Theorem~6.1.22]{DS:89} (see also~\cite{GZ:02}). We recall that $\Phi_1(s)=s\log(s)-s+1$.
\begin{lemma}\label{lem:tech_log_sob}
 For all $t\in\mathbb{R}$ and $\psi\in L^2_\mu(\Omega)$ such that $\|\psi\|_{L^2_\mu(\Omega)} = 1$, one has
 \[
  \int_\Omega\Phi_1\left((1+t\psi)^2\right)\dd \mu \leq t^2\int_\Omega\psi^2\log(\psi^2)\dd\mu + (1+t^2)\log(1+t^2) +(1+|\langle\psi\rangle_\mu|)t^2,
 \]
 where $\langle\psi\rangle_\mu \defi \int_\Omega\psi\,\dd\mu$. 
\end{lemma}

\begin{proof}

 Let us define, for $\delta>0$, \[f_\delta(t) = \int_\Omega\Phi_1\left((1+t\psi)^2+\delta\right)\dd \mu - t^2\int_\Omega\psi^2\log(\psi^2)\dd\mu - (1+t^2)\log(1+t^2) \,.\]
 Differentiating $f_\delta$ yields
 \[
  f_\delta'(t) = 2\int_\Omega(1+t\psi)\psi\log\left((1+t\psi)^2+\delta\right)\dd\mu -2t\int_\Omega\psi^2\log(\psi^2)\dd\mu - 2t\log(1+t^2) - 2t\,.
 \]
 In particular, $f_\delta'(0) = 2\log(1+\delta)\langle\psi\rangle_\mu$. Differentiating once more, and using that $\|\psi\|_{L^2_\mu(\Omega)}^2 = 1$, we obtain
\[
  f_\delta''(t) = 2\int_\Omega\psi^2\log\left(\frac{(1+t\psi)^2+\delta}{(1+t^2)\psi^2}\right)\dd\mu  + 4\int_\Omega\psi^2\frac{(1+t\psi)^2}{\delta+(1+t\psi)^2}\dd\mu -\frac{4t^2}{1+t^2}-2\,.
 \]
 Therefore, using that $\log(x)\leq x-1$ in the first  term, that $\frac{x}{\delta+x}\leq 1$ in the second, together with the fact that $\mu$ is a probability measure and that $\|\psi\|_{L^2_\mu(\Omega)}^2 = 1$, one gets
 \[
 f_\delta''(t)\leq \frac{2\delta}{1+t^2} + \frac{4t}{1+t^2}\langle\psi\rangle_\mu + 4-\frac{4t^2}{1+t^2}-2\leq 2\delta + 2|\langle\psi\rangle_\mu|+2\,.
 \]
 One concludes by integrating this inequality twice between $0$ and $t$, using that $f_{\delta}(0)=\Phi_1(1+\delta)$, and letting $\delta\to0$.
\end{proof}
\begin{prop}\label{prop:nearly_logSob}
For any $\xi\in L^q_\mu(\Omega)$ with $q>2$, one has 
 \begin{equation}\label{eq:nearly_logSob1}
  \int_\Omega\xi^2\log\left(\frac{\xi^2}{\|\xi\|_{L^2_\mu(\Omega)}^2}\right)\mathrm{d}\mu\leq \frac{q}{q-2}\|\xi-\langle\xi\rangle_\mu\|_{L^q_\mu(\Omega)}^2 + \frac{q-4}{q-2}\|\xi-\langle\xi\rangle_\mu\|_{L^2_\mu(\Omega)}^2,
 \end{equation}
 where $\langle\xi\rangle_\mu \defi \int_\Omega\xi\,\dd\mu$. Besides, one also has
 \begin{equation}\label{eq:nearly_logSob2}
  \int_\Omega\Phi_1(\xi^2)\,\mathrm{d}\mu\leq \frac{q}{q-2}\|\xi-1\|_{L^q_\mu(\Omega)}^2 + \frac{2q-6}{q-2}\|\xi-1\|_{L^2_\mu(\Omega)}^2 + \Phi_1\left(1+\|\xi-1\|_{L^2_\mu(\Omega)}^2\right)\,.
 \end{equation}
\end{prop}
\begin{proof}
 \begin{itemize}
  \item[i)] Assume that $\langle\xi\rangle_\mu\neq 0$, and take $t$ and $\psi$ such that $\xi = \langle\xi\rangle_\mu(1+t\psi)$ and $\|\psi\|_{L^2_\mu(\Omega)} = 1$. In particular, $\langle\psi\rangle_\mu = 0$, and $1+t^2=\frac{\|\xi\|_{L^2_\mu(\Omega)}^2}{\langle\xi\rangle_\mu^2}$. Using Lemma~\ref{lem:tech_log_sob}, a somewhat tedious but straightforward computation yields
  \[
  \int_\Omega\xi^2\log\left(\frac{\xi^2}{\|\xi\|_{L^2_\mu(\Omega)}^2}\right)\dd \mu\leq\int_{\Omega}(\xi-\langle\xi\rangle_\mu)^2\log\left(\frac{(\xi-\langle\xi\rangle_\mu)^2}{\|\xi-\langle\xi\rangle_\mu\|_{L^2_\mu(\Omega)}^2}\right)\dd\mu+2\|\xi-\langle\xi\rangle_\mu\|_{L^2_\mu(\Omega)}^2\,.
 \]
  Observe that the last inequality also holds if $\langle\xi\rangle_\mu= 0$. Let $\phi  = \xi-\langle\xi\rangle_\mu$. Then,
  \[
   \int_\Omega\xi^2\log\left(\frac{\xi^2}{\|\xi\|_{L^2_\mu(\Omega)}^2}\right)\dd \mu\leq\frac{2}{q-2}\|\phi\|_{L^2_\mu(\Omega)}^2\int_{\Omega}\frac{\phi^2}{\|\phi\|_{L^2_\mu(\Omega)}^2}\log\left(\frac{\phi^{q-2}}{\|\phi\|_{L^2_\mu(\Omega)}^{q-2}}\right)\dd\mu+2\|\phi\|_{L^2_\mu(\Omega)}^2\,.
  \]
 Therefore, by Jensen's inequality for the probability measure $\frac{\phi^2}{\|\phi\|_{L^2_\mu(\Omega)}^2}\dd \mu$ applied to the concave function $\log$, one obtains
  \begin{multline*}
  \int_\Omega\xi^2\log\left(\frac{\xi^2}{\|\xi\|_{L^2_\mu(\Omega)}^2}\right)\dd \mu
  \leq\frac{2}{q-2}\|\phi\|_{L^2_\mu(\Omega)}^2\log\left(\frac{\|\phi\|_{L^q_\mu(\Omega)}^q}{\|\phi\|_{L^2_\mu(\Omega)}^q}\right)+2\|\phi\|_{L^2_\mu(\Omega)}^2\\
 =\frac{q}{q-2}\|\phi\|_{L^2_\mu(\Omega)}^2\log\left(\frac{\|\phi\|_{L^q_\mu(\Omega)}^2}{\|\phi\|_{L^2_\mu(\Omega)}^2}\right)+2\|\phi\|_{L^2_\mu(\Omega)}^2\,,
 \end{multline*}
 and one concludes using that $\log(x)\leq x-1$.
  \item[ii)] Take $t$ and $\psi$ such that $\xi = 1+t\psi$ and $\|\psi\|_{L^2_\mu(\Omega)} = 1$. Remark that $t=\|\xi-1\|_{L^2_\mu(\Omega)}$. Using that $|\langle\psi\rangle_\mu|\leq \|\psi\|_{L^2_\mu(\Omega)}=1$, Lemma~\ref{lem:tech_log_sob} yields
  \[
  \int_\Omega\Phi_1(\xi^2)\,\mathrm{d}\mu - \Phi_1\left(1+\|\xi-1\|_{L^2_\mu(\Omega)}^2\right)\leq\int_{\Omega}(\xi-1)^2\log\left(\frac{(\xi-1)^2}{\|\xi-1\|_{L^2_\mu(\Omega)}^2}\right)\dd\mu+3\|\xi-1\|_{L^2_\mu(\Omega)}^2\,.
 \]
  Letting $\phi = \xi-1$, the proof goes on as for i). 
 \end{itemize}
\end{proof}

\noindent
From there, logarithmic Sobolev inequalities are immediate consequences of Poincaré--Sobolev inequalities, of~\cite[Lemma 5.2]{CCHHK:20}, and of the fact that $\Phi_1(1+s) \leq s\log(1+s)$.
\begin{corollary}[Logarithmic Sobolev inequalities]\label{cor:logsob}
Assume that $\mu$ has a density (still denoted by $\mu$) with respect to the Lebesgue measure such that $0<\mu_\flat\leq \mu(x) \leq \mu_\sharp$ for a.e.~$x\in\Omega$. Then, for any $\xi\in H^1(\Omega)$, one has 
\[
 \int_\Omega\xi^2\log\left(\frac{\xi^2}{\|\xi\|_{L^2_\mu(\Omega)}^2}\right)\mathrm{d}\mu\leq C(\Omega,d,\mu_\flat, \mu_\sharp) \|\nabla\xi\|^2_{L^2_\mu(\Omega;\R^d)}\,.
\]
Besides, if $|\Gamma^D|>0$ and $\xi-1\in H^{1,D}_0(\Omega)\defi\{v\in H^1(\Omega)\mid v=0\text{ on }\Gamma^D\}$, then 
\[
  \int_\Omega\Phi_1(\xi^2)\,\mathrm{d}\mu\leq C'(\Omega,d,\mu_\flat, \mu_\sharp) \left(1+\log\left(1+\|\xi-1\|_{L^2_\mu(\Omega)}^2\right)\right)\|\nabla\xi\|^2_{L^2_\mu(\Omega;\R^d)}\,.
\]
\end{corollary}
}

\subsection{Discrete logarithmic Sobolev inequalities}
{

Similarly to what was done in \cite{CCHHK:20}, one can derive discrete logarithmic Sobolev inequalities adapted to the hybrid setting.
}
\begin{prop}[Discrete logarithmic Sobolev inequality{, Neumann case}] \label{prop:logsob}
  Let $\D$ be a given discretisation of $\Omega$, with regularity parameter $\theta_\D$. 
  Let $\v_\D, \v^\infty_\D\in\V_\D$ be two positive vectors of unknowns such that
  $$\int_{\Omega}v_\M = \int_{\Omega} v_\M ^\infty  \ifed {M},$$
  and set {$v^\infty_{\M,\sharp} \defi\underset{K \in \M}{\sup} v_K^\infty$}.
  Define $\rac_\D$ as the element of $\V_\D$ such that
  $$\xi_K \defi \sqrt{\frac{v_K}{v^\infty_K}}\quad\forall K\in\M,\qquad\xi_\s \defi \sqrt{\frac{v_\s}{v^\infty_\s}}\quad\forall\s\in\E.$$
  Then, there exists $C_{LS,\infty}>0$, only depending on $M$, $v^\infty_{\M,\sharp}$, $\Omega$, $d$, and $\theta_\D$ such that
  \begin{equation} \label{LogSob}
    \int_{\Omega} v_\M^\infty\Phi_1 \left ( \xi^2_\M \right )
    \leq C_{LS,\infty}^2  \,\big| \rac_\D \big|_{1,\D}^2\,.
    %% \sqrt{M v_{\M,\sharp}^\infty}
  \end{equation}
\end{prop}
\begin{proof}
{
  By \eqref{eq:nearly_logSob1} and \cite[Lemma 5.2]{CCHHK:20} }applied to the probability measure $\mu(x)\,{\rm d}x=v_\M^\infty(x)\,\frac{{\rm d}x}{M}$ and to the function $\xi_\M=\sqrt{\frac{v_\M}{v_\M^\infty}}$, we first infer that, for $q>2$,
  $$\int_{\Omega}v_\M\log\left(\xi^2_\M\right)\leq C(M,v^\infty_{\M,\sharp},q)\left(\|\xi_\M-\overline{\xi_\M}\|^2_{L^q(\Omega)}+\|\xi_\M-\overline{\xi_\M}\|^2_{L^2(\Omega)}\right),$$
  where we let $\overline{\xi_\M}\defi\frac{1}{|\Omega|}\int_{\Omega}\xi_\M\,{\rm d}x$. The conclusion then falls in two steps. On the one hand, since $\int_{\Omega}v_\M = \int_{\Omega} v_\M ^\infty$, we remark that
  $$\int_{\Omega} v_\M^\infty\Phi_1 \left ( \xi^2_\M \right )=\int_{\Omega}v_\M\log\left(\xi^2_\M\right).$$
  On the other hand, we invoke~\eqref{poincawir} and the discrete Poincar\'e--Sobolev inequality of~\cite[Lemma B.25, $p=2$]{DEGGH:18} for $2<q<\frac{2d}{d-2}$:
  $$\forall \w_\D \in \V_{\D,0}^N, \qquad \| w_\M \|_{L^q(\Omega)} \leq C_{PS} |\w_\D|_{1,\D},$$
  where $C_{PS}>0$ only depends on $\Omega$, $d$, and $\theta_\D$, that we apply to $\w_\D=\underline{\xi}_\D-\overline{\xi_\M}\,\one_\D\in\V_{\D,0}^N$.
  This proves~\eqref{LogSob}.
  
\end{proof}
{
Starting from \eqref{eq:nearly_logSob2}, a similar proof yields the following result. The relevant discrete Poincaré--Sobolev inequality in this case is given in \cite[Lemma B.32, $p=2$]{DEGGH:18}.
\begin{prop}[Discrete logarithmic Sobolev inequality, Dirichlet case] \label{prop:logsob_dir_disc}
  Assume that $|\Gamma^D|>0$.
 Let $\D$ be a given discretisation of $\Omega$, with regularity parameter $\theta_\D$. 
  Let $\v_\D, \v^\infty_\D\in\V_\D$ be two positive vectors of unknowns such that
  $$\v_\D- \v^\infty_\D  \in \V_{\D,0}^D,$$
  and set {$v^\infty_{\M,\sharp} \defi\underset{K \in \M}{\sup} v_K^\infty$} and $M^\infty:=\int_\Omega v_\M ^\infty$.
  Define $\rac_\D$ as the element of $\V_\D$ such that
  $$\xi_K \defi \sqrt{\frac{v_K}{v^\infty_K}}\quad\forall K\in\M,\qquad\xi_\s \defi \sqrt{\frac{v_\s}{v^\infty_\s}}\quad\forall\s\in\E.$$
  Then, letting $\mu \defi \frac{v^\infty_{\M}}{M^\infty}$, there exists $C_{LS,\Gamma^D,\infty}>0$, only depending on $M^\infty$, $v^\infty_{\M,\sharp}$, $\Omega$, $d$, $\Gamma^D$, and $\theta_\D$ such that
\begin{equation}\label{eq:LogSobDirDisc}
	\int_\Omega v_\M^\infty\Phi_1(\xi^2_\M)  \leq   
		C_{LS,\Gamma^D,\infty}^2\left(1+\log\left (1 + \| \xi_\M-1 \|_{L^2_\mu(\Omega)}^2\right)\right)			
	 	| \rac_\D |^2_{1,\D}.
\end{equation}
\end{prop}
}

{
\section{Nonlinear scheme for mixed Dirichlet-Neumann boundary conditions} \label{Ap:NLmixted}

In this appendix, we introduce and analyse a version of the nonlinear scheme for the evolution problem~\eqref{evol:mixed:ino} when $|\Gamma^D|>0$. In order to perform the asymptotic analysis, we need to assume that the data are compatible with the thermal equilibrium: 
\[
	f= 0, \quad g^N = 0, \quad\text{ and there exists } \rho^D > 0 \text{ such that } g^D = \rho^D \e^{-\phi} = \rho^D \omega.
\]
For such data, given $u^{in} \geq 0$, the solution $u$ to \eqref{evol:mixed:ino} is positive for $t>0$, and converges towards $u^\infty = \rho^D \e^{-\phi}$ when $t \to \infty$.

\subsection{Scheme and well-posedness}

Accordingly to this setting, we define $\u_\D^\infty = \rho^D \underline{\omega}_\D$. One has
$u^\infty_\flat\underline{1}_\D \leq \u^\infty_\D \leq u^\infty_\sharp\underline{1}_\D$, where 
$0< u^\infty_\flat\leq u^\infty_\sharp$ only depend on $\rho^D$, $\phi$, and $\Omega$.
Remind that, as in \eqref{sch3:w}, given a positive $\u_\D \in \V_\D$, one defines $\w_\D(\u_\D) \in \V_\D$ as
\[
  w_K\defi\log\left(\frac{u_K}{u_K^\infty}\right)\quad\forall K\in\M,\qquad w_\s\defi\log\left(\frac{u_\s}{u_\s^\infty}\right)\quad\forall\s\in\E.
\]
For mixed boundary conditions, the discrete problem reads: Find $\big(\u_\D^n\in\V_\D\big)_{n\geq 1}$ positive such that
\begin{subequations}\label{sch:nonlin:mix}
        \begin{empheq}[left = \empheqlbrace]{align}
	            \frac{1}{\Delta t } 
    	       (u^{n}_\M - u^{n-1}_\M, v_\M)_{\Omega} 
	            +  T_\D \left(\u_\D^{n}, \w_\D(\u_\D^{n}), \v_\D\right)
	             &= 0\qquad\qquad\qquad\;\forall\v_{\D}\in\V_{\D,0}^{D},\label{sch1:mix}  \\
            \w_\D(\u_\D^n)& \in \V_{\D,0}^{D},  \label{sch2:mix} \\
            u^0_K &= \frac{1}{|K|}\int_K u^{in } \qquad\forall K\in\M. \label{sch3:mix} 
        \end{empheq}
\end{subequations}
Notice that, since for all $\s \in \E$, $w_\s^n = \log \left ( \dfrac{u^n_\s}{u^\infty_\s} \right )$, the equation~\eqref{sch2:mix} only means that, for all $\s \in  \E^D_{ext}$, $u^n_\s = u^\infty_\s$, which enforces strongly the Dirichlet boundary condition on $\Gamma^D$. One can show the following existence result.

\begin{theorem}[Existence of positive solutions and entropy dissipation] \label{th:existencenonlin:mix}
Let $u^{in} \in L^2(\Omega)$ be a non-negative function.
There exists at least one positive solution $\big(\u_\D^n\in\V_\D)_{n \geq 1}$ to the nonlinear scheme~\eqref{sch:nonlin:mix}. It satisfies the following entropy/dissipation relation:
\begin{equation} \label{eq:Diss:mix}
	\forall n \in \mathbb{N}, \qquad 
		\frac{\N^{n+1} - \N ^n }{\Delta t} + \Diss^{n+1} \leq 0,
\end{equation}
where $\N ^n$ and $\Diss^n$ are, respectively, the discrete relative entropy and dissipation defined in~\eqref{entro1}.
Moreover, there exists $\varepsilon >0$, depending on $\Lambda$, $\phi$, $u^{in}$, $\rho^D$, $\Omega$, $d$, $\Delta t$, and $\D$ such that, for all $n\geq 1$, $u_K^n\geq\varepsilon$ for all $K\in\M$ and $u_\s^n\geq\varepsilon$ for all $\s\in\E$.
\end{theorem}
The proof of this theorem relies on the same arguments as the one of Theorem~\ref{th:existencenonlin} for (homogeneous) pure Neumann boundary conditions. The major difference lies in the counterpart of Lemma~\ref{lemma:positivity}, which is no longer based on the positivity of the mass, but on the prescribed (zero) value on the Dirichlet faces.
%\simo{Also, in that case, it is not necessary to introduce a regularised problem to prove existence.}

\subsection{Long-time behaviour}

In the next theorem, we state the long-time behaviour of the discrete solutions to the nonlinear scheme \eqref{sch:nonlin:mix}. % Using \eqref{eq:LogSobDir} instead of \eqref{eq:LogSobDirDisc}, the adaptation to the continuous case is straightforward.
\begin{theorem}[Asymptotic stability] \label{Th:nonlin:mix}
  If $\big(\u^n_\D\in\V_\D\big) _{n \geq 1}$ is a (positive) solution to~\eqref{sch:nonlin:mix}, then the discrete entropy decays exponentially fast in time: there is $\displaystyle \nu_{\mathrm{nl},\Gamma^D}>0$, depending on $\Lambda$, $\phi$, $\Gamma^D$, $\rho^D$, $u^{in}$, $\Omega$, $d$, and $\theta_\D$ such that 
  \begin{equation} \label{nonlin:decentropie:mix}
    \forall n \in \mathbb{N}, \qquad \N^{n+1} \leq (1 + \nu_{\mathrm{nl},\Gamma^D} \,\Delta t ) ^{-1} \N^n. 
  \end{equation} 
 Consequently, the discrete solution converges exponentially fast in time towards its associated discrete steady-state.
%  : there exists $C_D>0$ depending only on $\phi$, $\rho^D$, $u^{in}$, $\Omega$, $d$ such that for all $n\in\mathbb{N}^\star$
% $
%     \| u^n_\M - u^\infty_\M \|_{L^1(\Omega)} \leq 
%     			C_D \left ( 1 + \nu_{\mathrm{nl},\Gamma^D} \Delta t \right ) ^{-\frac{n}{2}}.
% $
\end{theorem}
% 
% To perform the long-time analysis, we will need a bound on the mass of $u$, which is not preserved, unlike the $|\Gamma^D| = 0$ case. In fact, one has the following estimate on the mass: given the measure $\mu$ of Proposition \ref{cor:logsob} there is $C_\mu>0$ depending only on $\mu$ such that for any non-negative $g \in L^1(\Omega)$, 
% \begin{equation} \label{eq:mass:mix}
% \displaystyle \| g \| _{L^1_{\tilde{\mu}}} \leq \| \Phi_1(g) \| _{L^1_{\mu}} + C_{\mu}.
% \end{equation}
% Indeed, for $\varepsilon> 0$, we define $F_\varepsilon : x \mapsto \varepsilon \Phi_1(x)$. One can compute its Legendre transform $F_\varepsilon^* (y) \defi  \displaystyle \sup_{x > 0} xy - F_\varepsilon(y)$, namely $F_\varepsilon^* (y)= \varepsilon \left (\exp \left (\dfrac{y}{\varepsilon} \right ) - 1 \right )$. Thus for any $x,y$ positive, we have $xy \leq F_\varepsilon(x) + F_\varepsilon ^*(y)$. Applying this with $x = g$, $y = 1$, $\varepsilon = \mu(\Omega)$ and integrating over $\Omega$ with respect to $\tilde{\mu}$, we get \eqref{eq:mass:mix}.

\begin{proof}
Let $n \in \mathbb{N}^{\star}$. As in Section~\ref{asnon}, one has $\displaystyle \Diss^{n}\geq\frac{1}{C_F}\hat{\Diss}^{n}\geq\frac{4  u_\flat^\infty \lambda_\flat \alpha_\flat}{C_F}\big| \rac_\D^{n}  \big| _{1,\D}^2$, where $C_F>0$ depends on the data.
Using the discrete log--Sobolev inequality \eqref{eq:LogSobDirDisc} from Proposition~\ref{prop:logsob_dir_disc}, we get 
\begin{multline}\label{eq:logopt}
  \N^n \leq C^2_{LS,\Gamma^D,\infty} 
		\left(1+\log \left (1 + \| \xi^n_\M-1 \|_{L^2_\mu(\Omega)}^2\right)\right) | \rac_\D ^n|^2_{1,\D} \\
		\leq  \frac{C_{LS,\Gamma^D,\infty}^2 C_F} {  4  u_\flat^\infty \lambda_\flat \alpha_\flat  }
		\left(1+\log \left (1 + \| \xi^n_\M-1 \|_{L^2_\mu(\Omega)}^2\right)\right) \Diss^n.
\end{multline}
Then, there is $C> 0$ such that (recall that $\rac^n_\D$ is positive)  
\[
 	\| \xi_\M^{n} - 1 \|^2_{L^2_{\mu}(\Omega)} 
 		\leq \| (\xi^n_\M)^2 \|_{{L^1_\mu(\Omega)}}+1
 		\leq \left \|  \Phi_1  \left( (\xi_\M^{n})^2 \right ) \right  \|_{{L^1_\mu(\Omega)}} + C
 		= (M^\infty)^{-1}\N^n + C\,,
\]
where the last inequality is an application of the Fenchel--Young inequality $x\leq \Phi_1(x) + \Phi_1^\star(1)$, where $\Phi_1^\star$ is the convex conjugate of $\Phi_1$ and $x = (\xi_\M^{n})^2$. But, since the entropy/dissipation relation \eqref{eq:Diss:mix} holds, the discrete entropy decays and $\N^n \leq \N^0$. Therefore, one has 
\[
\| \xi_\M^{n} - 1 \|^2_{L^2_{\mu}(\Omega)}  \leq (M^\infty)^{-1}\N^0 + C\,.
\]
Combining this estimate with \eqref{eq:logopt}, we deduce that there exists $\nu_{\mathrm{nl},\Gamma^D}>0$, depending on $\Lambda$, $\phi$, $\Gamma^D$, $\rho^D$, $u^{in}$, $\Omega$, $d$, and $\theta_\D$ such that $\displaystyle \N^n \leq \nu_{\mathrm{nl},\Gamma^D} \Diss^n$. Then, using the entropy/dissipation relation \eqref{eq:Diss:mix}, we get \eqref{nonlin:decentropie:mix}.
% \newline
% To conclude about the convergence in $L^1$, one can apply the Csisz\'ar--Kullback inequality \eqref{eq:CKDir} with $g = \dfrac{u^n_\M}{u^\infty_\M}$ and $\mu = u^\infty_\M$:	
% \[
% 	\| u^n_\M-u^\infty_\M\|_{L^1(\Omega)}^2  =
% 	\| g - 1 \|_{L^1_\mu (\Omega)}^2 \leq C_{DCK}  \mu(\Omega) 
% 	\frac{ \| g \|_{L^1_{\tilde{\mu}}(\Omega)} }{ \log \left ( 1 + \| g \|_{L^1_{\tilde{\mu}}(\Omega)}  \right )} 
% 	 \| \Phi_1(g) \|_{L^1_\mu (\Omega)}
% 	 =  C_{DCK}  \mu(\Omega) 
% 	\frac{ \| g \|_{L^1_{\tilde{\mu}}(\Omega)} }{ \log \left ( 1 + \| g \|_{L^1_{\tilde{\mu}}(\Omega)}  \right )} 
% 	 \N^n.
% \]
% As before, by \eqref{eq:mass:mix}, one has $\| g \|_{L^1_{\tilde{\mu}}(\Omega)} \leq \N^n + C_{u^\infty} \leq \N^0 + C_{u^\infty} $, so there exists $C_2>0$ depending on $u^{in}$, $\rho^D$, $\phi$ and $\Omega$ such that 
% $\dfrac{ \| g \|_{L^1_{\tilde{\mu}}(\Omega)} }{ \log \left ( 1 + \| g \|_{L^1_{\tilde{\mu}}(\Omega)} \right )} \leq C_2$. We conclude by \eqref{nonlin:decentropie:mix}.
\end{proof}
}

\section{Proofs of technical results} \label{Appendix_technicalresults}

\subsection{Discrete boundedness by mass and dissipation} \label{ap:lemma}

We prove Lemma~\ref{lemma:positivity} from Section~\ref{ssse:exist}. To ease the reading, we first recall the result.
\lem*
\begin{proof}
  %% We denote by $C$ any positive constant depending only on $\D$, $\u^\infty$, $\Lambda$, $M_0^+$, $M_0^-$ and $C_0$.
  {For $K\in\M$, using \eqref{def:a_K} and \eqref{def:localcoercivity}, we first infer that
  $$\delta_K \w_K \cdot \A_K \delta_K \w_K = a_K^{\Lambda} (\w_K,\w_K) \geq \lambda_\flat \alpha_\flat  |\w_K|_{1,K}^2= \lambda_\flat \alpha_\flat  \sum_{\s \in \E_K} \frac{|\s|}{d_{K,\s}} \big(w_K-w_\s\big)^2 .$$
  By definition~\eqref{def:regmesh} of the regularity parameter $\theta_\D$, we have that $\frac{|\s|}{d_{K,\s}} \geq \frac{h_K^{d-2}}{\theta_\D}$ for all $\s\in\E_K$, so that
  \begin{equation} \label{lobo}
    \delta_K \w_K \cdot \A_K \delta_K \w_K \geq \frac{\lambda_\flat \alpha_\flat}{\theta_\D} h_K^{d-2} |\delta_K \w_K|^2.
  \end{equation}}
  By the expression~\eqref{ED} of $\Diss(\w_\D)$, and the local lower bound~\eqref{lobo}, we thus get
  \begin{align*}
    \Diss(\w_\D)  & = \sum_{K \in \M } r_K\big( \u_K^\infty \times{\rm exp}(\w_K) \big) \delta_K \w_K \cdot \A_K \delta_K \w_K \\
    &  \geq  \frac{\lambda_\flat\alpha_\flat}{\theta_\D}\sum_{K \in \M } h_K^{d-2}r_K\big( \u_K^\infty \times{\rm exp}(\w_K) \big) |\delta_K \w_K |^2 \\ 
    & = \frac{\lambda_\flat\alpha_\flat}{\theta_\D}\sum_{K \in \M }  \sum_{\s \in \E_K}h_K^{d-2} r_K\big( \u_K^\infty \times{\rm exp}(\w_K) \big) (w_K -  w_\s)^2.
  \end{align*}
  Let $K \in \M$ and $\s \in \E_K$ be fixed. Using, successively, the definition~\eqref{eq:rK} of $r_K$ combined with the definition~\eqref{def:f} of $f_{|\E_K|}$, the combination of~\eqref{bornesstationnaires} with assumptions~\eqref{hyp:g:monotony} and~\eqref{hyp:g:homogeneity}, and the assumptions~\eqref{hyp:g:cons+cons} and~\eqref{hyp:g:bounds} combined with the bound~\eqref{def:boundfaces} on $|\E_K|$, we infer, for $w_\s \neq w_K$, 
  \begin{align*}
     r_K\big( \u_K^\infty \times{\rm exp}(\w_K) \big)(w_K -  w_\s)^2 & \geq
    \frac{1}{|\E_{K}|}\,m  \big(u_K^\infty \e^{w_K}, u_\s^\infty \e^{w_\s} \big)  (w_K -  w_\s)^2 \\
    & \geq \frac{u^\infty_\flat}{|\E_{K}|}\, m  \big(\e^{w_K}, \e^{w_\s} \big) (w_K -  w_\s)^2 \\
    & \geq  \frac{u^\infty_\flat}{d\theta_\D^2}\left ( \e^{w_K}- \e^{w_\s} \right ) \left ( w_K - w_\s \right )\geq 0, 
  \end{align*}
  and we verify that this inequality still holds when $w_\s = w_K$.
  Since $\Diss(\w_\D)\leq C_\sharp$ by~\eqref{Entropycontrol}, for all $K\in\M$, and all $\s\in\E_K$, we have
  \begin{equation} \label{prox}
    0 \leq \left ( \e^{w_K}- \e^{w_\s} \right ) \left ( w_K - w_\s \right ) \leq \zeta h_K^{2-d},
  \end{equation}
  with $\zeta\defi\frac{dC_\sharp\theta_\D^3}{\lambda_\flat\alpha_\flat u^\infty_\flat}>0$ (recall that $\alpha_\flat$ depends on $\Omega$, $d$, and $\theta_\D$).
  Besides, since $\sum_{K \in \M} |K| u_K ^\infty \e^{w_K} \leq M_\sharp$ again by~\eqref{Entropycontrol}, we have $|K| u_K ^\infty \e^{w_K} \leq M_\sharp$ for all $K \in \M$.
  Similarly, since $\sum_{K \in \M} |K| u_K ^\infty \e^{w_K} \geq M_\flat$, there exists $K_0 \in \M$ such that $|\Omega|u_{K_0} ^\infty \e^{w_{K_0}} \geq M_\flat>0$.
  Combining these bounds, we infer that there exists $K_0\in\M$ such that
  \begin{equation} \label{eq:wK0}
    \log\left(\frac{M_\flat}{|\Omega|u_\sharp^\infty}\right)\leq w_{K_0}\leq \log\left(\frac{M_\sharp}{|K_0|u_\flat^\infty}\right).
  \end{equation}
  
  Now, let us show that we can similarly frame all the other components of $\w_\D$.
  
  For $a,x \in \R$, let us define $E(a,x)=\left ( \e^x - \e^a  \right ) (x - a)\geq 0$. Observe that
  $E(a,y+a)\e^{-a} = (\e^y-1)y =:\xi(y)$ and that $\xi$ is continuous, strictly decreasing for $y<0$, strictly increasing for $y>0$, $\xi(0)=0$, and $\xi(y)\to+\infty$ when $y\to\pm\infty$. Let $b, a_\sharp>0$, and take $|a|\leq a_\sharp$. By the properties of $\xi$, if $E(a,x)\leq b$, then $|x|\leq \kappa_b(a_\sharp) := a_\sharp + \max\{|y|\ \text{s.t.}~\xi(y)=b\e^{a_\sharp}\}$. We can thus infer that if $(x_k)_{k=0,\dots,m}$ is a finite sequence of real numbers such that $E(x_k, x_{k+1})\leq b$ and $|x_0|\leq a_\sharp$, then $|x_m| \leq \kappa_b^{(m)}(a_\sharp)$ where $\kappa_b^{(m)}$ is $m$ compositions of $\kappa_b$. In particular, the bound only depends on $a_\sharp$, $m$ and $b$. 
  
  Now we can conclude the proof. Because of the connectivity of the mesh, for any cell $K$ (respectively, face $\s$) there is a finite sequence of components of $\w_\D$, denoted $(x_k)_{k=0,\dots,m}$, starting at $x_0 = w_{K_0}$ and finishing at $x_m = w_K$ (respectively, $x_m = w_\s$) such that, by \eqref{prox}, $E(x_k, x_{k+1})\leq b:=\zeta h_\D^{2-d}$. The inequality \eqref{eq:wK0} yields the initial bound on $|x_0|$, and one concludes by the above argument.

\end{proof}

\subsection{A local comparison result} \label{ap:AK}

We prove a local comparison result between the matrices $\A_K$ and some (local) diagonal matrices.
The proof relies on arguments that are similar to those advocated in~\cite{CaGui:17} to analyse the VAG scheme.
{
\begin{lemma} \label{lemma:cond}
  For $K \in \M$, let $\A_K\in\R^{|\E_K|\times|\E_K|}$ be the matrix defined by~\eqref{def:flux:A}.
  The matrices $\A_K$ are symmetric positive-definite, and there exists $C_A>0$, only depending on $\Lambda$, $\Omega$, $d$, and $\theta_\D$ such that
  \begin{equation*} 
    \forall K \in \M, \quad \text{{\rm Cond}}_2(\A_K) \leq C_A,
  \end{equation*}
  where $\text{{\rm Cond}}_2(\A_K) \defi \|\A_K^{-1}\|_2\|\A_K\|_2$ is the condition number of the matrix $\A_K$.
  Moreover, letting for $K\in\M$, $\B_K \in \R^{|\E_K|\times|\E_K|}$ be the diagonal matrix with entries
  \begin{equation} \label{BK}
    B_K^{\s\s} \defi \sum_{\s' \in \E_K} |A_K^{\s\s'}|\qquad\text{for all }\s\in\E_K,
  \end{equation}
  there exists $C_B>0$, only depending on $\Lambda$, $\Omega$, $d$, and $\theta_\D$ such that
  \begin{equation} \label{CB}
    \forall K \in \M, \,\forall w  \in \R ^{|\E_K|}, \qquad w \cdot \A_K w \leq w \cdot \B_K w \leq C_B \,w \cdot \A_K w.
  \end{equation}
\end{lemma}}
\begin{proof}
  Let $K \in \M$ and $k \defi |\E_K|$. As a direct consequence of its definition~\eqref{def:flux:A}, the matrix $\A_K\in\R^{k\times k}$ is symmetric and positive semi-definite.
  Now, let $w \defi (w_\s)_{\s \in \E_K} \in \R^k$, and define $\v_K \in \V_K$ such that
  $$v_K =  0\qquad\text{and}\qquad v_\s = -w_\s \text{ for all } \s \in \E_K.$$
  Then, $\delta_K \v_K = (v_K-v_\s)_{\s\in\E_K}=w$. {By~\eqref{lobo}, we immediately get that}
  \begin{equation*}
    w \cdot \A_K w \geq \frac{\lambda_\flat \alpha_\flat}{\theta_\D} h_K^{d-2} |w|^2,
  \end{equation*}
  which implies, since $w\in\R^k$ is arbitrary, that $\A_K$ is invertible, and gives us a lower bound on its smallest eigenvalue.
  By the same arguments advocated to prove~\eqref{lobo}, noticing that $\frac{|\s|}{d_{K,\s}} \leq \theta_\D h_K ^{d-2}$ for all $\s \in \E_K$, we infer that
  $$w \cdot \A_K w \leq \lambda_\sharp \alpha_\sharp \theta_\D h_K^{d-2} |w|^2.$$
  We eventually get, using the estimates on the eigenvalues of $\A_K$, that
  \begin{equation} \label{def:boundCond}
    \text{{\rm Cond}}_2(\A_K)\leq \frac{\lambda_\sharp \alpha_\sharp}{\lambda_\flat \alpha_\flat } \theta_\D ^2\ifed C_A,
  \end{equation}
  with $C_A>0$ only depending on $\Lambda$, $\Omega$, $d$, and $\theta_\D$.
  Now, by~\eqref{BK}, since $\A_K$ is symmetric, we have 
  $$w \cdot \B_K w= \sum _{\s \in \E_K}  \sum _{ \s' \in \E_K} |A^{\s\s'}_K| w_\s ^2 = \sum _{\s \in \E_K}  \sum _{ \s' \in \E_K} |A^{\s\s'}_K| w_{\s'} ^2 ,$$
  and we can use the half-sum to get
  $$w \cdot \B_K w= \sum _{\s \in \E_K}  \sum _{ \s' \in \E_K} |A^{\s\s'}_K| \frac{w_\s ^2+ w_{\s'}^2}{2}.$$
  Using Young's inequality, we infer
  \begin{align*}
	w \cdot \A_K w 	& =  \sum _{\s \in \E_K}  \sum _{ \s' \in \E_K} A^{\s\s'}_K w_\s w_{\s'} 
	\leq \sum _{\s \in \E_K}  \sum _{ \s' \in \E_K} |A^{\s\s'}_K| |w_\s| |w_{\s'}| \\
					& \leq \sum _{\s \in \E_K}  \sum _{ \s' \in \E_K} |A^{\s\s'}_K| \frac{w_\s^2 +w_{\s'}^2}{2}	= w \cdot \B_K w  .	
  \end{align*}
  For the second inequality, by symmetry of $\A_K$, we have
  $$w \cdot \B_K w = \sum_{\s \in \E_K} B_K^{\s\s} w_\s^2  \leq  \underset{ \s \in \E_K}{\max} (B_K^{\s\s} ) \sum_{\s \in \E_K}  w_\s^2 =  \underset{ \s \in \E_K}{\max} \left ( \sum_ {\s' \in \E_K}  |A_K^{\s'\s}| \right ) |w|^2 = \|\A_K\|_1 |w|^2.$$
  The space $\R^{k\times k}$ being of finite dimension, the norms $\| \cdot\|_1$ and $\| \cdot\|_2$ are equivalent, and there exists $\gamma_k>0$ such that  $\| \cdot\|_1 \leq \gamma_k \| \cdot\|_2$.
  Moreover, since $\A_K$ is symmetric positive-definite, the following inequality holds:
  $$w \cdot \A_K w \geq  \frac{\|\A_K\|_2}{\text{Cond}_2(\A_K) } |w|^2.$$
  From the previous estimates and \eqref{def:boundCond}, we deduce that
  $$w \cdot \B_K w \leq \gamma_{k} \,\text{Cond}_2(\A_K) \,w \cdot \A_K w \leq \gamma_{k} \,C_A \,w \cdot \A_K w.$$
  But, according to \eqref{def:boundfaces}, we have $\displaystyle\underset{ K \in \M} {\max} \gamma_{k} \leq \underset{(d+1) \leq l \leq d \theta_\D^2} {\max}\gamma_l$, therefore
  $$w \cdot \B_K w \leq C_B \,w \cdot \A_K w,$$
  where $\displaystyle C_B = C_A \underset{(d+1) \leq l \leq d \theta_\D^2} {\max}\gamma_{l}$ is a positive constant only depending on $\Lambda$, $\Omega$, $d$, and $\theta_\D$. This completes the proof of the comparison result~\eqref{CB}.
\end{proof}

\section*{Acknowledgements}         %this creates the heading for the acknowlegments
 { The authors would like to thank the anonymous reviewers for their remarks and suggestions which helped improving the quality of this paper.} The authors acknowledge support from the LabEx CEMPI
(ANR-11-LABX-0007). Claire Chainais-Hillairet also acknowledges support from the ANR MOHYCON
(ANR-17-CE40-0027).

\bibliographystyle{siam} 
\bibliography{HFV_longtime}
\end{document}